\documentclass[10pt,reqno]{amsart}
\usepackage[foot]{amsaddr}
\usepackage[margin=1in]{geometry}

\RequirePackage[OT1]{fontenc}

\RequirePackage{amsthm,amsmath,amssymb,enumerate,makecell,graphicx}
\RequirePackage[colorlinks,citecolor=blue,urlcolor=blue]{hyperref}

\DeclareFontFamily{U}{mathx}{\hyphenchar\font45}
\DeclareFontShape{U}{mathx}{m}{n}{
      <5> <6> <7> <8> <9> <10>
      <10.95> <12> <14.4> <17.28> <20.74> <24.88>
      mathx10
      }{}
\DeclareSymbolFont{mathx}{U}{mathx}{m}{n}
\DeclareFontSubstitution{U}{mathx}{m}{n}
\DeclareMathAccent{\widecheck}{0}{mathx}{"71}

\allowdisplaybreaks

\newcommand{\Id}{\operatorname{Id}}
\renewcommand{\P}{\mathbb{P}}
\newcommand{\R}{\mathbb{R}}
\newcommand{\C}{\mathbb{C}}
\newcommand{\hv}{\widehat{v}}
\newcommand{\hl}{\widehat{\lambda}}
\newcommand{\hmu}{\widehat{\mu}}
\newcommand{\Tr}{\operatorname{Tr}}
\newcommand{\tr}{\operatorname{tr}}
\newcommand{\E}{\mathbb{E}}
\newcommand{\eps}{\varepsilon}
\newcommand{\N}{\mathcal{N}}
\newcommand{\hSigma}{\widehat{\Sigma}}
\newcommand{\hLambda}{\widehat{\Lambda}}
\newcommand{\hK}{\widehat{K}}
\newcommand{\hT}{\widehat{T}}
\newcommand{\oSigma}{\mathring{\Sigma}}
\newcommand{\1}{\mathbf{1}}
\newcommand{\cR}{\mathcal{R}}
\newcommand{\A}{\mathcal{A}}
\newcommand{\B}{\mathcal{B}}

\newcommand{\supp}{\operatorname{supp}}
\newcommand{\spec}{\operatorname{spec}}
\newcommand{\dist}{\operatorname{dist}}
\newcommand{\diag}{\operatorname{diag}}
\newcommand{\ordereddist}{\operatorname{ordered-dist}}
\newcommand{\sT}{{\mathsf{T}}}
\renewcommand{\Im}{\operatorname{Im}}
\renewcommand{\Re}{\operatorname{Re}}
\newcommand{\wGamma}{\widetilde{\Gamma}}
\newcommand{\vT}{\widecheck{T}}
\newcommand{\wT}{\widetilde{T}}
\newcommand{\D}{\mathcal{D}}
\newcommand{\cH}{\mathcal{H}}
\newcommand{\cG}{\mathcal{G}}
\newcommand{\cF}{\mathcal{F}}
\newcommand{\cE}{\mathcal{E}}
\newcommand{\NC}{\text{NC}}
\newcommand{\HS}{\text{HS}}
\renewcommand{\i}{\mathbf{i}}
\renewcommand{\j}{\mathbf{j}}
\newcommand{\cI}{\mathcal{I}}
\newcommand{\Q}{\mathcal{Q}}
\newcommand{\hu}{\widehat{u}}
\newcommand{\Var}{\operatorname{Var}}
\newcommand{\Exponential}{\operatorname{Exponential}}
\renewcommand{\a}{\mathbf{a}}
\renewcommand{\b}{\mathbf{b}}

\newtheorem{theorem}{Theorem}[section]
\newtheorem{corollary}[theorem]{Corollary}
\newtheorem{lemma}[theorem]{Lemma}
\newtheorem{proposition}[theorem]{Proposition}

\newtheorem{assumption}[theorem]{Assumption}

\newtheorem*{remark*}{Remark}

\begin{document}
\title[PCA in linear mixed models]{Principal components in linear mixed models with general bulk}

\author{Zhou Fan}
\address{Z.F.: Department of Statistics and Data Science \\ Yale University}
\email{zhou.fan@yale.edu}

\author{Yi Sun}
\address{Y.S.: Department of Statistics \\ The University of Chicago}
\email{yisun@statistics.uchicago.edu}

\author{Zhichao Wang}
\address{Z.W.: Department of Mathematics\\ University of California, San Diego}
\email{zhw036@ucsd.edu}

\begin{abstract}
We study the principal components of covariance estimators in multivariate
mixed-effects linear models. We show that, in high dimensions, the principal
eigenvalues and eigenvectors may exhibit bias and aliasing effects that are not
present in low-dimensional settings. We derive the
first-order limits of the principal eigenvalue locations and eigenvector
projections in a high-dimensional asymptotic framework,
allowing for general population spectral distributions for the
random effects and extending previous results from a more restrictive spiked
model. Our analysis uses free probability techniques, and we develop two
general tools of independent interest---strong asymptotic freeness of GOE and
deterministic matrices and a free deterministic equivalent approximation for
bilinear forms of resolvents.
\end{abstract}

\maketitle

\section{Introduction}

Principal components analysis (PCA) is a commonly used technique for identifying
linear low-rank structure in high-dimensional data \cite{jolliffe}. For
$n$ independent samples in a comparably large dimension $p$, 
it is now well-established that the principal components of the sample
covariance matrix may be inaccurate for their population
counterparts \cite{johnstonelu}. A body of work has quantified
the behavior of PCA in this setting
\cite{johnstone,baiketal,baiksilverstein,paul,BGNrectangular,baiyao},
connecting to the Marcenko-Pastur and Tracy-Widom laws of asymptotic random 
matrix theory \cite{marcenkopastur,tracywidom}. We
refer readers to the review articles \cite{paulaue,johnstonepaul}
for more discussion and references to this and related lines of work.

Similar phenomena occur in statistical models where samples are
not independent, but instead exhibit complex dependence structure
\cite{burdaetal,zhang,liuauepaul,wangauepaul}.
However, the behavior of PCA in many such models is less well-understood. 
In this work, we consider the setting of mixed effects
linear models \cite{searleetal}, where dependence across observed samples
arises via linear combinations of independent latent variables.
These models are commonly used in statistical genetics to model
quantitative phenotypes in related individuals \cite{lynchwalsh}.
We study the behavior of principal eigenvalues and eigenvectors
of MANOVA covariance estimates for the random effects.

Our main results quantify several spectral bias and aliasing
phenomena that may occur in high-dimensional applications.
In particular, we show that large principal eigenvalues in the
covariance of one random effect may bias the principal eigenvectors
and also yield spurious eigenvalues in the estimated covariances of the other
effects. These phenomena are unique to mixed-effects models, and they
do not arise in similar spiked models of sample covariance matrices for
independent samples \cite{baiketal,baiksilverstein,paul}.
In \cite{fanjohnstonesun}, such phenomena for mixed models were first
described under an ``isotropic noise'' assumption, where the population
covariance of each random effect is a low-rank perturbation of the
identity. Our work extends these results to
the setting of general population spectral distributions for
the random effects. We derive generalizations of
the first-order limits for eigenvalues and
eigenvector projections in \cite{fanjohnstonesun} involving quantities
appearing in the fixed-point equations
for the empirical spectral law in \cite{fanjohnstonebulk}. We describe these
results in Section \ref{sec:mixedeffects}.

Our proofs are very different from the analytic
approach of \cite{fanjohnstonesun}. Instead, they
are based in free probability theory and
its connection to random matrices \cite{voiculescu,mingospeicher}.
Our work also establishes two general results in this
area---strong asymptotic freeness of independent GOE and
deterministic matrices and a method of deriving anisotropic resolvent
approximations using free deterministic equivalents \cite{speichervargas}. 
We describe these in Section \ref{sec:freeprob}.

The connection between free probability and random matrices was
introduced in \cite{voiculescu} for deterministic and GUE matrices and has
been extended to many other matrix models
\cite{dykema,voiculescustrengthened,hiaipetz,collins,capitainecasalis,collinssnaidy,benaychgeorges,speichervargas}.
Strong asymptotic freeness extending the approximation
from the trace to the operator norm was first proven in
\cite{haagerupthorbjornsen} for GUE matrices and extended to other
models in \cite{schultz,capitainedonatimartin,male,belinschicapitaine}.
Free probability techniques have recently been applied to study outlier
eigenvalues in other matrix models \cite{belinschietal,belinschietal2}
and spectral behavior in other statistical applications, including
autocovariance estimates for high-dimensional time series
\cite{bhattacharjeebose1,bhattacharjeebose2,bhattacharjeebose3,bosebook,bhattacharjeebose4} and sketching
methods for linear regression \cite{dobribanliu}. The tools we develop may be
of broader interest to the analysis of structured random matrices arising in
other applications.

\section{Probabilistic results in the linear mixed model}\label{sec:mixedeffects}

Extending the representation of
\cite{rao} to a multivariate setting, we consider the mixed-effects linear model
\begin{equation}\label{eq:mixedmodel}
Y=X\beta+U_1\alpha_1+\ldots+U_k\alpha_k \in \R^{n \times p}
\end{equation}
where $Y$ contains $n$ dependent observations in dimension $p$,
each a combination of fixed effects $X\beta$ and
random effects constituting the rows of
$\alpha_1,\ldots,\alpha_k$. Here,
\begin{itemize}
\item $X \in \R^{n \times m}$ is an observed design matrix of a small number $m$
of fixed effects,
with unknown regression coefficients $\beta \in \R^{m \times p}$.
\item For each $r=1,\ldots,k$, the matrix $\alpha_r \in \R^{n_r \times p}$ is
unobserved, its rows constituting $n_r$ i.i.d.\ realizations of a $p$-dimensional random effect.
\item Each $U_r \in \R^{n \times n_r}$ is a known, deterministic incidence matrix specified by the model design.
\end{itemize}
We study the behavior of PCA for estimates of
the \emph{variance components}, which are the 
covariance matrices $\Sigma_1,\ldots,\Sigma_k$ for the random effects in
$\alpha_1,\ldots,\alpha_k$.

In quantitative genetics, $U_1,\ldots,U_k$ may encode a
classification design, as commonly used in twin/sibling studies and
breeding experiments. Examples are discussed in
\cite{fanjohnstonebulk,fanjohnstonesun}. In genomewide association study
designs, $U_1,\ldots,U_k$ may contain genotype measurements at a set of
single-nucleotide polymorphisms (SNPs) \cite{yangetal,zhoustephens}.
It has been recognized since \cite{fisher,wright} that
variance components in these models can provide a decomposition of
the total population variance of quantitative phenotypes into constituent
genetic and non-genetic effects, yielding estimates of heritability. In
high-dimensional applications, including the analysis of gene expression traits
and other molecular phenotypes, the principal
eigenvectors of the genetic components may indicate phenotypic subspaces near
which responses to selection or random mutational drift are likely to be
constrained \cite{hineblows,blowsmcguigan,colletetal}. Principal eigenvectors
of the non-genetic components may correspond to hidden experimental
confounders, to be removed before performing downstream analyses
\cite{leekstorey,stegleetal}.

As $\alpha_1,\ldots,\alpha_k$ are not individually observed, one cannot
construct the usual
sample covariance estimator for $\Sigma_1,\ldots,\Sigma_k$. Instead, each
$\Sigma_r$ may be classically estimated by a MANOVA
estimator of the form
\[\hSigma_r=Y^\sT B_rY,\]
where the symmetric matrix $B_r$ is chosen to satisfy the properties
\[B_rX=0, \qquad \E[Y^\sT B_rY]=\Sigma_r.\]
Such an estimator $\hSigma_r$ is unbiased and equivariant to
rotations of coordinates in $\R^p$---these properties are analogous to those
holding for a sample covariance matrix for independent samples. 
For example, in a balanced one-way classification design, the within-group
covariance matrix $\Sigma_1$ is estimated by the MANOVA estimator
$\hSigma_1=Y^\top B_1Y$ where $B_1$ is a
scaled difference of two orthogonal projections, the first
onto a subspace of group means and the second onto its orthogonal
complement. See Appendix \ref{appendix:simulations} for further details of this
example.

Our main results, Theorems \ref{thm:outliers} and \ref{thm:eigenvectors} below,
characterize the first-order limiting behavior of the principal
eigenvalues and eigenvectors of any such matrix $\hSigma=Y^\top BY$
in a high-dimensional asymptotic framework.

\subsection{Model assumptions} \label{sec:model}
We assume that the random effects arise in the following way.
\begin{assumption}\label{assump:alpha}
The matrices $\alpha_1,\ldots,\alpha_k$ are independent. The rows of each
$\alpha_r$ are independent, with the $i^\text{th}$ row given by
\[\sum_{j=1}^{\ell_r} \gamma_j^{(r)} \xi_{ij}^{(r)}+\eps_i^{(r)}.\]
Here $\gamma_1^{(r)},\ldots,\gamma_{\ell_r}^{(r)} \in \R^p$
are $\ell_r$ deterministic vectors, and
$\xi_{ij}^{(r)} \in \R$ are independent random variables satisfying
\[\E[\xi_{ij}^{(r)}]=0, \qquad \E[(\xi_{ij}^{(r)})^2]=1, \qquad \E[|\xi_{ij}^{(r)}|^k]
\leq C_k\]
for all $k \geq 1$ and some constants $C_k>0$.
For a covariance $\oSigma_r \in \R^{p \times p}$,
the noise $\eps_i^{(r)} \in \R^p$ is Gaussian with distribution
$\eps_i^{(r)} \sim \N(0,\oSigma_r)$.
\end{assumption}

Stacking $\gamma_1^{(r)},\ldots,\gamma_{\ell_r}^{(r)}$ as the rows of
\begin{equation}\label{eq:Gamma}
\Gamma_r=\begin{pmatrix} - & \gamma_1^{(r)} & - \\ & \vdots & \\ - &
\gamma_{\ell_r}^{(r)} & - \end{pmatrix} \in \R^{\ell_r \times p},
\end{equation}
each $\alpha_r$ has independent rows with mean 0 and covariance of the spiked
form
\begin{equation}\label{eq:Sigmar}
\Sigma_r=\Gamma_r^\sT \Gamma_r+\oSigma_r.
\end{equation}
The leading term $\Gamma_r^\sT \Gamma_r$ induces up to $\ell_r$ ``signal''
eigenvalues that separate from the ``noise'' eigenvalues of $\oSigma_r$. Our
results should be interpreted in the setting where the noise covariance
$\oSigma_r$ does not itself have isolated eigenvalues that
separate from the bulk of its eigenvalue distribution.

As a compromise between generality of the model and simplicity of the analysis,
Assumption \ref{assump:alpha} follows the approach in \cite{nadler2008} and
imposes a Gaussian assumption on $\eps_i^{(r)}$ but not on $\xi_{ij}^{(r)}$.
The signal directions $\gamma_1^{(r)},\ldots,\gamma_{\ell_r}^{(r)}$ are not
required to be orthogonal for each $r$. It is likely that our
theoretical results in Theorems \ref{thm:sticktobulk}, \ref{thm:outliers}, and
\ref{thm:eigenvectors} all remain correct under a milder moment
assumption for this noise $\eps_i^{(r)}$, and it may be possible to prove
such an extension using cumulant expansions of the remainder terms in the
Gaussian integration-by-parts formula, as done in \cite{belinschicapitaine}. 
However, we will not pursue this direction in the current work.

For the linear mixed model (\ref{eq:mixedmodel}),
we study an asymptotic framework summarized as follows.
\begin{assumption}\label{assump:asymptotic}
The dimensions $n,p,n_1,\ldots,n_k \to \infty$, where $k$ is a fixed constant.
There are universal constants $C,c>0$ such that for each $r=1,\ldots,k$,
\begin{itemize}
\item $c<p/n<C$ and $c<n_r/n<C$,
\item $\|U_r\|<C$ and $\|B\|<C/n$
\item $\|\Gamma_r\|<C$, $\|\oSigma_r\|<C$, and $\ell_r<C$.
\end{itemize}
\end{assumption}
Thus, the number of samples is proportional to the number of
realizations of each random effect (and also to the dimension $p$).
This and the assumption $\|U_r\|<C$ are discussed in greater detail
in \cite{fanjohnstonebulk,fanjohnstonesun}, and hold for many classification
and experimental designs. 
The scaling $\|B\|<C/n$ is usual for MANOVA estimators, to yield $\hSigma_r$ on
the same scale as its estimand $\Sigma_r$.

The last statement
implies a bounded number of signal eigenvalues in each variance component,
where each eigenvalue remains bounded in size. It is an 
important open problem to extend our results beyond this setting.

\subsection{Bulk eigenvalue distribution} \label{sec:det-eq-measure}

Under the above assumptions, a characterization of a deterministic approximation
for the empirical eigenvalue
distribution of
\begin{equation}\label{eq:hSigma}
\hSigma=Y^\top BY
\end{equation}
was derived in \cite{fanjohnstonebulk}. We review this result here.

Consider the setting of no signal, meaning $\ell_r=0$ and $\Sigma_r=\oSigma_r$
for each $r=1,\ldots,k$. We introduce the notations $n_+=n_1+\cdots+n_k$ and
\begin{equation}\label{eq:Frs}
    F_{rs}=\sqrt{n_rn_s} U_r^\sT BU_s \in \R^{n_r \times n_s},
\quad F=(F_{rs})_{r,s=1}^k \in \R^{n_+ \times n_+},
\end{equation}
\[\diag_n(\a)=\diag(a_1\Id_{n_1},\ldots,a_k\Id_{n_k}) \in \R^{n_+ \times n_+},\]
\[\b \cdot \oSigma=b_1\oSigma_1+\cdots+b_k\oSigma_k \in \R^{p \times p}.\]
Let $\Tr_r$ be the trace of the $(r,r)$ block (of size $n_r \times n_r$)
in the $k \times k$ matrix block decomposition corresponding to $\C^{n_+}=\C^{n_1} \oplus \cdots \oplus \C^{n_k}$.
The Stieltjes transform of a measure $\mu$ is $m(z)=\int
(x-z)^{-1}d\mu(x)$.

\begin{theorem}[\cite{fanjohnstonebulk}]\label{thm:bulk}
Suppose Assumptions \ref{assump:alpha} and \ref{assump:asymptotic} hold, and $\ell_r=0$ for each $r=1,\ldots,k$. Let $\hSigma$ be as in (\ref{eq:hSigma}), and let $\hmu=p^{-1}\sum_{i=1}^p \delta_{\lambda_i(\hSigma)}$ be the empirical distribution of its eigenvalues.

For each $z \in \C^+$, there exist unique $z$-dependent values $a_1,\ldots,a_k \in \C^+ \cup \{0\}$ and $b_1,\ldots,b_k \in \overline{\C^+}$ that satisfy the equations
\begin{align}
a_r&=-n_r^{-1}\Tr \Big((z\Id+\b \cdot \oSigma)^{-1}\oSigma_r\Big),\label{eq:arecursion}\\
b_r&=-n_r^{-1}\Tr_r\Big((\Id+F\diag_n(\a))^{-1}F\Big).\label{eq:brecursion}
\end{align}
The function $m_0:\C^+ \to \C^+$ defined by
\begin{equation}\label{eq:m0}
m_0(z)=-p^{-1}\Tr\Big((z\Id+\b \cdot \oSigma)^{-1}\Big)
\end{equation}
is the Stieltjes transform of a deterministic
probability measure $\mu_0$ on $\R$, for which
$\hmu-\mu_0 \to 0$ weakly almost surely.
\end{theorem}

The distribution $\mu_0$ is an $n$-dependent deterministic equivalent measure
\cite{hachemetal} for the empirical eigenvalue distribution of $\hSigma$. An
example is depicted in Figure \ref{fig:oneway}. It is
defined by the noise covariances $\oSigma_1,\ldots,\oSigma_k$ and the
structure of the linear model (\ref{eq:mixedmodel}),
via the fixed-point equations (\ref{eq:arecursion}--\ref{eq:m0}).

\subsection{Noise eigenvalues stick to the support}
\label{sec:firstorder}

For any $\delta>0$, denote the $\delta$-neighborhood
of the support of the above law $\mu_0$ as
\[\supp(\mu_0)_\delta=\{x \in \R:\dist(x,\supp(\mu_0))<\delta\}.\]
We first strengthen the weak convergence statement of
Theorem \ref{thm:bulk} to show that in the same
setting of no signal, all eigenvalues of $\hSigma$ belong to
$\supp(\mu_0)_\delta$ for any fixed $\delta>0$ and large $n$.

\begin{theorem}\label{thm:sticktobulk}
Suppose Assumptions \ref{assump:alpha} and \ref{assump:asymptotic} hold, and
$\ell_r=0$ for each $r=1,\ldots,k$. Let $\hSigma$ be as in (\ref{eq:hSigma}).
Then for any constant $\delta>0$, almost surely for all large $n$,
\[\spec(\hSigma) \subset \supp(\mu_0)_\delta.\]
\end{theorem}

We defer the proof to Appendix \ref{sec:prelim-results}.
The proof is an application of a strong asymptotic freeness result for GOE
and deterministic matrices, which we describe in Section \ref{sec:freeprob}.

\subsection{Limits of signal eigenvalues and eigenvectors}

We now consider the setting where $\ell_s \neq 0$ for at least one component
$s \in \{1,\ldots,k\}$. This may induce ``outlier'' eigenvalues of
$\hSigma$ that separate from the support of $\mu_0$---these and their 
eigenvectors are typically the focus of analysis in PCA.
(The component where $\ell_s \neq 0$ may or may not be
the component estimated by $\hSigma \equiv \hSigma_r$.)

Our main results describe the first-order limits of
these eigenvalues and eigenvectors. This description involves the $z$-dependent
quantities $\{b_r\}_{r = 1}^k$ from Theorem \ref{thm:bulk}. We check in
Proposition \ref{prop:analyticextension} that each $b_r$ extends as an analytic
function in $z$ to all of $\C \setminus \supp(\mu_0)$, and we denote this
extension by $b_r(z)$. Let us write as shorthand
\[\Gamma=\begin{pmatrix} \Gamma_1 \\ \vdots \\ \Gamma_k \end{pmatrix}
\in \R^{\ell_+ \times p}, \qquad \ell_+=\ell_1+\cdots+\ell_k,\]
where $\Gamma_r \in \R^{\ell_r \times p}$ are as defined in (\ref{eq:Gamma}).
For $\lambda \in \R \setminus \supp(\mu_0)$, let us denote
\begin{equation}\label{eq:bnotation}
\b \cdot \oSigma= \sum_{r = 1}^k b_r(\lambda) \oSigma_r,\qquad
\diag_\ell(\b) = \diag(b_1(\lambda) \Id_{\ell_1}, \ldots,
b_k(\lambda)\Id_{\ell_k}).
\end{equation}
Then, in the asymptotic limit, the outlier eigenvalue locations are approximated
by the deterministic multiset
\begin{equation} \label{eq:Lambda0}
\Lambda_0= [\,\lambda \in \R \setminus \supp(\mu_0) : 0 = \det T(\lambda)\,]
\end{equation}
where, for $\lambda \in \R \setminus \supp(\mu_0)$, we define
\begin{equation} \label{eq:Tnew}
T(\lambda)= \Id + \Gamma (\lambda\Id
    + \b \cdot \oSigma)^{-1} \Gamma^\sT \diag_\ell(\b) \in \R^{\ell_+ \times
    \ell_+}.
\end{equation}
The roots of the equation $0=\det T(\lambda)$ are counted with their
analytic multiplicities in this multiset.

\begin{theorem}\label{thm:outliers}
Suppose Assumptions \ref{assump:alpha} and \ref{assump:asymptotic} hold, let
$\hSigma$ be as in (\ref{eq:hSigma}), and let $\Lambda_0$ be defined by
(\ref{eq:Lambda0}). Fix any constant $\delta>0$. Almost surely as $n \to
\infty$, there exist $\Lambda_\delta \subseteq \Lambda_0$ and $\hLambda_\delta
\subseteq \spec(\hSigma)$, where $\Lambda_\delta$ and $\hLambda_\delta$
respectively contain all elements of $\Lambda_0$ and $\spec(\hSigma)$ outside
$\supp(\mu_0)_\delta$, such that
\[
\ordereddist(\Lambda_\delta,\hLambda_\delta) \to 0.
\]
\end{theorem}

Here, for two finite multisets $A, B \subset \R$, we denote
\[
\ordereddist(A, B) = \begin{cases} \infty & \text{if $|A| \neq |B|$} \\ \max_i\{|a_{(i)} - b_{(i)}|\} & \text{if $|A| = |B|$}, \end{cases}
\]
where $a_{(i)}$ and $b_{(i)}$ are the ordered values of $A$ and $B$ counting multiplicity. 
We state the result as a matching of $\spec(\hSigma)$
and $\Lambda_0$, rather than convergence of $\spec(\hSigma)$ to $\Lambda_0$, as
$\Lambda_0$ is also $n$-dependent. A phase-transition phenomenon analogous to
that of \cite{baiketal} is implicit in this result, in that the cardinality of
the multiset $\Lambda_0$ may transition from 0 to a positive value with the
increase of signal strength in $\Gamma$.

For the corresponding outlier eigenvectors 
of $\hSigma$, the following characterizes their inner products with the signal
vectors $\gamma_i^{(r)}$ that constitute the rows of $\Gamma_1,\ldots,\Gamma_k$.
We denote, in addition to (\ref{eq:bnotation}), $\partial_\lambda$ as the
derivative in $\lambda$, and
\[\diag_\ell(\b')=\partial_\lambda \diag_\ell(\b)
=\diag(b_1'(\lambda)\Id_{\ell_1},\ldots,b_k'(\lambda)\Id_{\ell_k}).\]

\begin{theorem} \label{thm:eigenvectors}
    In the setting of Theorem \ref{thm:outliers}, 
    let $\lambda \in \Lambda_0$ be any element of multiplicity 1 such that
    $|\lambda-\lambda'|>\delta$ for all other $\lambda' \in \Lambda_0$,
    and $\dist(\lambda,\supp(\mu_0))>\delta$.
    Let $u \in \ker T(\lambda) \subset
    \R^{\ell_+}$ be a unit vector, and let $\hv$ be the unit eigenvector for the
    eigenvalue $\hl$ of $\hSigma$ closest to $\lambda$. Almost surely as $n \to
    \infty$, for some choice of sign of $\hv$,
\[
\Gamma\,\hv- \alpha^{-1/2} u \to 0,
\]
where $\alpha>0$ is the scalar quantity defined by
\begin{equation} \label{eq:alpha-def}
    \alpha = u^\sT \Big(-\diag_\ell(\b) \Gamma \cdot
    \partial_\lambda[(\lambda\Id + \b \cdot
    \oSigma)^{-1}] \cdot \Gamma^\sT \diag_\ell(\b) +
    \diag_\ell(\b')\Big)u.
\end{equation}
\end{theorem}

We show in the proof that $\ker T(\lambda)$ has dimension 1,
so $u \in \ker T(\lambda)$ is unique up to
sign. The above states that the inner-products of the sample eigenvector
$\hv$ with the true signal vectors $\gamma_i^{(r)}$ are approximately
a scalar multiple of the entries of this vector $u$.

We verify in Appendix \ref{appendix:isotropic} that if
$\oSigma_r=\sigma_r^2\Id$ for each $r=1,\ldots,k$, then
Theorems \ref{thm:outliers} and \ref{thm:eigenvectors} coincide with
the first-order results in \cite{fanjohnstonesun}.

\section{Implications for principal components analysis}\label{sec:PCA}

Theorems \ref{thm:outliers} and \ref{thm:eigenvectors} imply several
qualitative phenomena for the
behavior of PCA for classical MANOVA covariance estimators in
high-dimensional linear mixed models. They imply that in the asymptotic regime
of Assumption \ref{assump:asymptotic}, the naive
sample eigenvalues and eigenvectors are inconsistent for
their population counterparts, and they lead to open questions
about how to obtain improved estimates in these settings. A full exploration
of these questions is outside the scope of this work, but we provide some
discussion of these phenomena and inferential challenges in this section.

\subsection{Qualitative phenomena}\label{subsec:qualitative}

To illustrate the phenomena that are implied by Theorems
\ref{thm:outliers} and \ref{thm:eigenvectors},
we will focus our discussion on a simple example. 

Consider any mixed model (\ref{eq:mixedmodel}) with $k \geq 2$ components.
Suppose that $\Sigma_1$ and $\Sigma_2$ each have a rank-one signal,
and $\Sigma_r$
has no signals for $r \geq 3$; that is, $\ell_1=\ell_2=1$, and $\ell_r=0$
for each $r \geq 3$. Suppose further that $\hSigma \equiv \hSigma_1$ is an
unbiased MANOVA estimate of $\Sigma_1$.  Denote
\[\gamma^{(1)}_1 \equiv \sqrt{\mu_1}v_1, \qquad \gamma^{(2)}_1 \equiv
\sqrt{\mu_2}v_2\]
as the rows of $\Gamma_1$ and $\Gamma_2$, where $v_1,v_2 \in \R^p$ are unit
vectors. For simplicity of interpretation, let us assume
that $\oSigma_r v_1=\oSigma_r v_2=0$ for every $r$.
This implies by (\ref{eq:Sigmar}) that
$\mu_1,\mu_2$ are the signal eigenvalues
in $\Sigma_1,\Sigma_2$, with eigenvectors $v_1,v_2$. (Note that our results do
not require the matrices $\oSigma_r$ to be of full rank.) We set
\[\rho=\langle v_1,v_2 \rangle, \qquad \mu=\max(\mu_1,\mu_2).\]
We also define two $O(1)$ quantities $c_1,c_2$ by
\begin{equation}
c_r=\sum_{t=1}^k \Tr (U_r^\sT BU_t)(U_tBU_r^\sT) \cdot
\Tr \oSigma_t.\label{eq:cr}
\end{equation}

\emph{Eigenvalue bias.} Theorem \ref{thm:outliers} reveals that
principal eigenvalues of $\hSigma_1$ are biased upwards
for the true eigenvalues of $\Sigma_1$. Assuming that
$\mu_1$ is large and $\mu_2 \lesssim \mu_1$, we show in Appendix
\ref{appendix:qualitative} that the largest root of $0=\det T(\lambda)$ has the
large-$\mu$ expansion
\begin{equation}\label{eq:eigenvaluebias}
\lambda=\mu_1+\text{bias}, \qquad
\text{bias}=c_1+c_2\frac{\mu_2\rho^2}{\mu_1}+o_\mu(1),
\end{equation}
where $o_\mu(1) \to 0$ as $\mu_1 \to \infty$.
Thus, for large but fixed $\mu_1$ and $\mu_2 \lesssim \mu_1$,
as $n,p \to \infty$, the sample eigenvalue is upward biased by approximately
$c_1+c_2\mu_2\rho^2/\mu_1$. Here, the first term is a constant depending on the
model design and level of noise, and the second term arises as an extra bias
if $\mu_2$ is also large and the signal eigenvector of $\Sigma_2$
is aligned with that of $\Sigma_1$.\\

\emph{Eigenvalue aliasing.} Theorem \ref{thm:outliers} also reveals that
$\hSigma_1$ can have spurious ``aliased'' outlier eigenvalues that are not
caused by signal in $\Sigma_1$, but rather by signal in $\Sigma_2$.
Suppose $\mu_1=0$, but $\mu_2$ is large. We show in Appendix
\ref{appendix:qualitative} that $0=\det T(\lambda)$ has two roots given by
\begin{equation}\label{eq:eigenvaluealias}
\lambda=\pm \sqrt{c_2\mu_2}+o_\mu(1),
\end{equation}
where $o_\mu(1) \to 0$ as $\mu_2 \to \infty$.
Thus, $\hSigma_1$ has two aliased outlier eigenvalues of opposite signs.
For large but fixed $\mu_2$, as $n,p \to \infty$, these aliased eigenvalues are
of size proportional to $\sqrt{\mu_2}$.\\

\emph{Eigenvector bias.} Theorem
\ref{thm:eigenvectors} implies that the sample eigenvectors of $\hSigma_1$
may be biased in the signal direction of $\Sigma_2$.
Suppose $\mu_1 \asymp \mu_2$ are both large, and $\rho$ is bounded away
from $\pm 1$. For the sample eigenvector $\hv$ corresponding to the
eigenvalue described by (\ref{eq:eigenvaluebias}), we show in Appendix
\ref{appendix:qualitative} that the deterministic approximation for
$\Gamma\,\hv$ in Theorem \ref{thm:eigenvectors} is
\begin{equation}\label{eq:eigenvecexpansion}
\alpha^{-1/2}u=\Gamma\,v_1+O_\mu(1/\sqrt{\mu}).
\end{equation}
Here, the $O_\mu(1/\sqrt{\mu})$ term captures the error
between $\hv$ and the true eigenvector $v_1$. To better understand this
error, let us define a vector $w \in \R^2$ so that
$\Gamma^\top w$ is the unit vector parallel to the component
of $v_2$ orthogonal to $v_1$.
We show in Appendix \ref{appendix:qualitative} that the approximation
for $\langle \Gamma^\top w,\hv \rangle=w^\top (\Gamma\,\hv)$ has the
large-$\mu$ expansion
\begin{equation}\label{eq:eigenvecexpansion2}
w^\top (\alpha^{-1/2}u)=\frac{c_2\mu_2}{\mu_1^2}\rho\sqrt{1-\rho^2}
+o_\mu(1/\mu).
\end{equation}
Thus, for large but fixed $\mu_1 \asymp \mu_2$, as $n,p \to \infty$, 
$\hv$ is biased in the direction $\Gamma^\top w$ which is orthogonal to $v_1$,
of size approximately $(c_2\mu_2/\mu_1^2)\rho\sqrt{1-\rho^2}$.\\

\begin{figure}[t]
\includegraphics[width=0.4\textwidth,page=1]{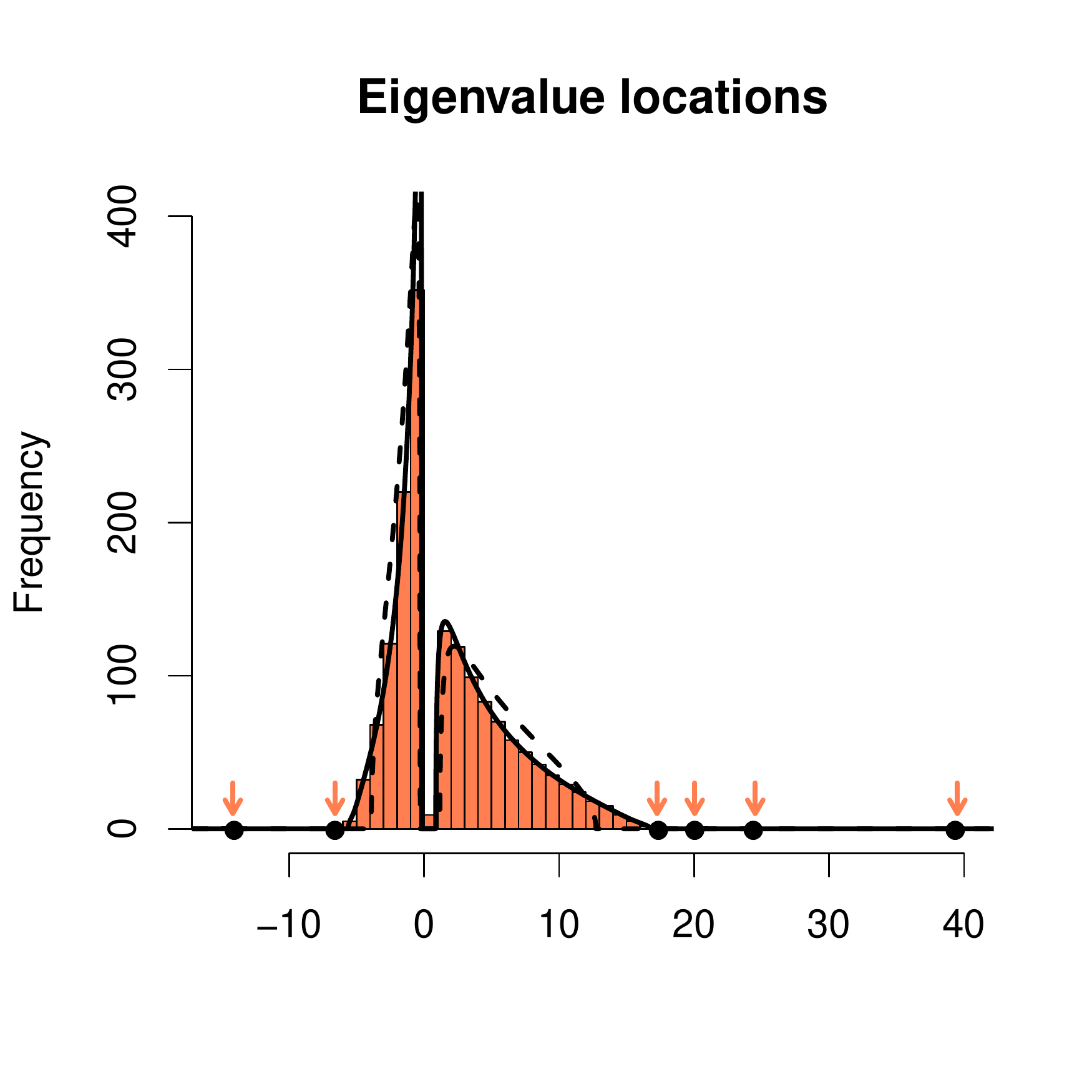}%
\includegraphics[width=0.4\textwidth,page=2]{outliers.pdf}
\caption{Left: Histogram of sample eigenvalues of the MANOVA estimate for
$\Sigma_1$ in a one-way layout design, averaged across 1000 simulations, with
the four largest and two smallest eigenvalues indicated by red arrows. Six black
dots indicate roots of $0=\det T(\lambda)$, predicting the asymptotic locations
of these eigenvalues. Overlaid are the densities of the bulk law
$\mu_0$ (solid black) and of $\mu_0$ computed from an isotropic-noise
approximation (dashed black).
Right: Inner-product of each of three principal sample eigenvectors
$(\hv_j:j=1,2,3)$ with the true population eigenvector $e_j$ (horizontal
axis) and an orthogonal direction $w_j$ partially aligned with
$\Sigma_2$ (vertical axis). Empirical
averages across 1000 simulations (red dots/arrows) are overlaid with the
predictions of Theorem \ref{thm:eigenvectors} (black dots).}\label{fig:oneway}
\end{figure}

\begin{table}[t]
\centering
\begin{tabular}{l|ccc|ccc|ccc}
\hline
& \multicolumn{3}{c|}{Predicted (true)}
& \multicolumn{3}{c|}{Predicted (isotropic)}
& \multicolumn{3}{c}{Observed} \\
  \hline
 & $\hat{\lambda}$ & $\hat{v}^\top e_j$ & $\hat{v}^\top w_j$ & $\hat{\lambda}$ & $\hat{v}^\top e_j$ & $\hat{v}^\top w_j$ & $\hat{\lambda}$ & $\hat{v}^\top e_j$ & $\hat{v}^\top w_j$ \\ 
  \hline
$\mu_1=32$ & 39.35 & 0.90 & 0.07 & 40.22 & 0.90 & 0.07 & 39.51 & 0.89 & 0.07 \\ 
  $\mu_2=16$ & 24.39 & 0.78 & 0.03 & 25.05 & 0.81 & 0.03 & 24.54 & 0.77 & 0.03 \\ 
  $\mu_3=8$ & 17.36 & 0.56 & -0.14 & 17.37 & 0.71 & -0.16 & 17.26 & 0.53 & -0.13 \\ 
   \hline
\end{tabular}
\caption{Three principal eigenvalues $\hl$ and eigenvector inner-products
$\hv_j^\top e_j$ and $\hv_j^\top w_j$
displayed in Figure \ref{fig:oneway}. Theoretical predictions are computed from Theorems \ref{thm:outliers} and \ref{thm:eigenvectors}, using the true
noise covariances (left) and an isotropic-noise approximation (middle).
Observed values (right) are averaged over 1000 simulations.}
\label{tab:oneway}
\end{table}

Figure \ref{fig:oneway} and Table \ref{tab:oneway} illustrate these phenomena
in a more complex setting for a balanced one-way layout design,
corresponding for example to a twin study with $n=1600$ individuals in $n_1=800$
twin pairs, and $p=3200$ traits. We simulate a rank-3 signal
component $32e_1e_1^\top+16e_2e_2^\top+8e_3e_3^\top$ in $\Sigma_1$
and a rank-2 signal component $32ww^\top+64e_4e_4^\top$ in $\Sigma_2$, where
$w=(e_1+e_2+e_3)/\sqrt{3}$, and we sample all remaining eigenvalues
of $\Sigma_1,\Sigma_2$ from $\Exponential(1)$.
Additional details are provided in Appendix \ref{appendix:simulations}.

Figure \ref{fig:oneway} displays sample eigenvalues of the MANOVA estimate
$\hSigma_1$, with numerically computed roots
of $0=\det T(\lambda)$. There are 4
positive and 2 negative roots. Of these, the 3rd largest and
the smallest (negative) root are attributed to aliasing from $e_4$ in
$\Sigma_2$---their sample eigenvectors are predicted by
Theorem \ref{thm:eigenvectors} to be orthogonal to $\{e_1,e_2,e_3\}$. The
1st, 2nd, and 4th largest roots correspond to the true eigenvalues
32, 16, and 8, each observed with upward bias. For each of the three
corresponding sample eigenvectors $\hv_j$, Figure \ref{fig:oneway} displays its
predicted and simulated alignment with the true direction $e_j$ and with the
orthogonal direction $w_j$ obtained by residualizing $e_j$ out of $w$.
The values of these predicted and simulated eigenvalues and
eigenvector alignments are also summarized in Table \ref{tab:oneway}.

\subsection{Improved estimation of principal components}\label{subsec:isotropic}

The preceding phenomena indicate that the sample eigenvalues and eigenvectors of
classical MANOVA estimates for $\Sigma_1,\ldots,\Sigma_k$ are inconsistent in
the regime of Assumption \ref{assump:asymptotic}.
The estimated eigenvalues in $\hSigma_r$ may have upward bias, the estimated 
eigenvectors may be biased towards eigenvectors of other
components $\Sigma_s$ for $s \neq r$, and the number of apparent signal
principal eigenvectors in $\hSigma_r$ may even be incorrect due to aliasing
effects from these other components.

The probabilistic results of Theorems \ref{thm:outliers} and
\ref{thm:eigenvectors} also suggest a possible route for improved estimation of 
the principal eigenvalues and eigenvectors:
The observed signal eigenvalues of a matrix
$\hSigma=Y^\top BY$, while inconsistent for the true signal eigenvalues of
$\Sigma_1,\ldots,\Sigma_k$, do nonetheless provide some information about these
matrices. As indicated by Theorem \ref{thm:outliers},
they correspond approximately to roots of the equation $0=\det T(\lambda)$. This
matrix $T(\lambda)$ in (\ref{eq:Tnew}) depends on:
\begin{enumerate}
\item The spectra of the $k$ noise covariances $\oSigma_1,\ldots,\oSigma_k$, and
the alignments of their eigenvectors across these $k$ different components.
\item The alignments of the rows of $\Gamma$ (the true signal
vectors) with these noise covariances $\oSigma_1,\ldots,\oSigma_k$.
\item The sizes of the true signal eigenvalues and the alignments of the
signal vectors across these $k$ components, which are related to the
magnitudes and inner-products of the rows of $\Gamma$.
\end{enumerate}

Under parametric modeling assumptions for the noise covariance matrices
$\oSigma_1,\ldots,\oSigma_k$, the observed outlier eigenvalues $\hl$
for matrices of the form $\hSigma=Y^\top BY$ can yield estimating equations
$0=\det T(\hl)$ for the true signal eigenvalues in each component
$\Sigma_1,\ldots,\Sigma_k$ (as well as for the cross-component
alignments of their corresponding signal eigenvectors).
Furthermore, if an estimation matrix $B$ for $\hSigma=Y^\top BY$ can be chosen
such that the vector $u \in \ker T(\hl)$ is proportional to $\Gamma\,v$,
where $\hl$ is the observed eigenvalue of $\hSigma$ and $v$ is a true signal
eigenvector of $\Sigma_r$, then Theorem \ref{thm:eigenvectors} indicates that
the corresponding sample eigenvector $\hv$ of $\hSigma$ approximately
satisfies $\Gamma\,\hv
\propto \Gamma\,v$, so that $\hv$ is not asymptotically biased towards the
signal direction of a different variance component $\Sigma_s$. This debiasing
can, for example, lead to asymptotically consistent estimates of linear
functionals of this true eigenvector $v$.

These ideas were implemented and analyzed in \cite{fanjohnstonesun} in the
simplest parametric setting where $\oSigma_r=\sigma_r^2\Id$, for each
$r=1,\ldots,k$ and some scalar variance parameters
$\sigma_1^2,\ldots,\sigma_k^2$. In this setting, \cite{fanjohnstonesun}
proposed a specific algorithm to solve the estimating equations 
$0=\det T(\hl)$ arising from a parametric family of matrices
$\hSigma=Y^\top B(\theta_1,\ldots,\theta_k)Y$ to yield
estimates of all sufficiently large signal eigenvalues of
$\Sigma_1,\ldots,\Sigma_k$. These estimated eigenvalues were shown to be
asymptotically consistent in the high-dimensional regime of Assumption
\ref{assump:asymptotic}, at a parametric $1/\sqrt{n}$ rate.
Furthermore, for each corresponding signal eigenvector $v$,
\cite{fanjohnstonesun} demonstrated how to obtain a specific estimation matrix
$B(\theta_1,\ldots,\theta_k)$ for which the vector $u \in \ker T(\hl)$ indeed
satisfies $u \propto \Gamma\,v$, and thus the algorithm returns a debiased
estimate of this true eigenvector $v$.
We refer readers to \cite{fanjohnstonesun} for further details.

When $\oSigma_1,\ldots,\oSigma_k$ are not isotropic, we believe that
nonparametric estimation of their spectra and eigenvector alignments may be
challenging. However, in certain more parametric contexts---for example when
$\oSigma_1,\ldots,\oSigma_k$ capture known autocovariance structure across
temporal variables or known genetic correlation structure across quantitative
traits, up to a small number of unknown parameters---it may be possible to
develop an estimation procedure similar to that of \cite{fanjohnstonesun},
which first estimates these parameters that describe the noise structure in
$\oSigma_1,\ldots,\oSigma_k$, and then estimates the principal eigenvalues and
eigenvectors of interest, using more general estimating equations that are
derived from our results in Theorems \ref{thm:outliers} and
\ref{thm:eigenvectors}. We leave a more
detailed exploration of this possibility to future work.

Isotropic noise is often assumed in practice \cite{pattersonetal}, and our
results also provide an understanding of the error that may arise in the
original method of \cite{fanjohnstonesun} due to model misspecification.
Figure \ref{fig:oneway} displays simulated eigenvalue
densities $\mu_0$ computed using
the true matrices $\oSigma_r$, which have exponentially decaying spectra,
versus using their isotropic noise approximations with
$\sigma_r^2=p^{-1}\Tr \oSigma_r$. Table \ref{tab:oneway} compares the
corresponding eigenvalue and eigenvector alignment predictions.
We observe that the predictions of Theorems \ref{thm:outliers}
and \ref{thm:eigenvectors} for large outliers are very close to those under the
isotropic noise approximation.
This may also be understood from the calculations in the preceding section
for large $\mu_1,\mu_2$, as the dependence of $c_1,c_2$ in
(\ref{eq:cr}) on $\oSigma_1,\oSigma_2$ is only through their trace.
This suggests that the estimation procedure in
\cite{fanjohnstonesun} may be reasonably accurate for the larger principal
eigenvalues and their associated eigenvectors.
For eigenvalues closer to the support of the noise spectrum, the predictions of
Theorems \ref{thm:outliers} and \ref{thm:eigenvectors} using the true noise
covariances $\oSigma_r$ are more accurate than
those assuming isotropic noise, suggesting that inference for these principal
components may be improved by better parametric modeling of the noise structure.

\section{Free probability results}\label{sec:freeprob}

Our proofs use the connection between free probability and random matrices.
Introducing representations of $U_r$, $\alpha_r$, and $B$ detailed
in Section \ref{sec:det-eq-measure}, our matrix model $\hSigma$ may be written as
\begin{equation}\label{eq:WP}
\hSigma=W+P, \qquad W=\sum_{r=1}^k\sum_{s=1}^k H_r^\sT G_r^\sT F_{rs}G_sH_s,
\end{equation}
for deterministic matrices $\{H_1,\ldots,H_k\}$ and $\{F_{11},F_{12},
\ldots,F_{kk}\}$, independent matrices $\{G_1,\ldots,G_k\}$ with i.i.d.\ Gaussian
entries, and a fixed-rank perturbation $P$ (depending on $G_1,\ldots,G_k$).
We study the spectrum of $W$ by introducing an asymptotic approximation
\[w=\sum_{r=1}^k\sum_{s=1}^k h_r^*g_r^*f_{rs}g_sh_s,\]
where $h_r,g_r,f_{rs}$ belong to a von Neumann algebra
and are conditionally free (i.e.\ free with amalgamation)
over a diagonal subalgebra \cite{benaychgeorges}.
This method was also used in \cite{fanjohnstonebulk} to derive the
fixed-point equations (\ref{eq:arecursion}--\ref{eq:m0}) in Theorem
\ref{thm:bulk}.

Our analysis develops several new tools and results in free probability theory.
In this section, we state these results independent of the specific model
(\ref{eq:mixedmodel}), as they are of general interest for analyzing structured
random matrices in other applications. We defer proofs to Appendices
\ref{app:transforms}, \ref{app:strongfreeness}, and \ref{app:resolventapprox}.

\subsection{Augmented Cauchy and
\texorpdfstring{$R$}{R}-transforms}\label{sec:augmentedtransforms}

We call $(\A,\tau)$ a \emph{von Neumann probability space}
($W^*$-probability space) if $\A$ is a von Neumann algebra and $\tau:\A \to \C$ a positive, faithful,
normal trace. For a von Neumann subalgebra $\B \subset \A$, we denote by
\[\tau^\B:\A \to \B\]
the (unique) conditional expectation satisfying $\tau(\tau^\B(a))=\tau(a)$.

We review the following definitions of $\B$-valued Cauchy- and
$\cR$-transforms: For each $l \geq 1$, let $\NC(l)$ be the space of
non-crossing partitions of $1,\ldots,l$. For $\pi \in \NC(l)$, denote by
$\kappa_\pi^\B(a_1,\ldots,a_l)$ the non-crossing cumulant corresponding to
$\pi$. These satisfy the moment-cumulant relations
\begin{equation}\label{eq:momentcumulant}
\tau^\B(a_1a_2\ldots a_l)=\sum_{\pi \in \NC(l)} \kappa_\pi^\B(a_1,a_2,\ldots,a_l).
\end{equation}
Define the $\B$-valued Cauchy- and $\cR$-transform of $a \in \A$ by
    \[G_a^\B(b)=\tau^\B((b-a)^{-1})=\sum_{l \geq 0} b^{-1}(ab^{-1})^l, \qquad
    \cR_a^\B(b)=\sum_{l \geq 1} \kappa_l^\B(a,ba,\ldots,ba),\]
the former for all invertible $b \in \B$
with $\|b^{-1}\|$ sufficiently small and the latter for all $b \in \B$
with $\|b\|$ sufficiently small.\footnote{Note that, following conventions in\
free probability, we take the opposite sign for $G_a^\B(b)$ here as for the
Stieltjes transform used in Section \ref{sec:det-eq-measure}.}
The moment-cumulant relations (\ref{eq:momentcumulant}) yield the identity
\begin{equation}\label{eq:GRrelation}
    G_a^\B(b)=(b-\cR_a^\B(G_a^\B(b)))^{-1}
\end{equation}
for invertible $b \in \B$ with $\|b^{-1}\|$ sufficiently small.
We refer the reader to \cite[Chapter 9]{mingospeicher} for additional background
and details.

For our computations in Section \ref{sec:proof}, we will make use of
the following ``left-augmented'' Cauchy- and $R$-transforms, defined for
$a_1,a \in \A$ and $b \in \B$ by the mixed moments and mixed cumulants
\begin{align} \label{eq:aug-cauchy}
G^\B_{a_1, a}(b) &= \tau^\B\Big(a_1(b - a)^{-1}\Big)
    =\sum_{l \geq 0} \tau^\B(a_1b^{-1}(ab^{-1})^l),\\  \label{eq:aug-r}
R^{\B }_{a_1, a}(b) &= \sum_{l \geq 1} \kappa_l^\B(a_1, ba, \ldots, ba).
\end{align}
The following identity is then also a consequence of (\ref{eq:momentcumulant}),
and we provide a short proof in Appendix \ref{app:transforms}.

\begin{lemma} \label{lem:mod-gr}
For $a_1,a \in \A$ and all invertible $b \in \B$ with
    $\|b^{-1}\|$ sufficiently small,
    \begin{equation}\label{eq:mod-gr}
G^{\B}_{a_1,a}(b) = R^{\B }_{a_1, a}(G^{\B}_{a}(b)) G^{\B}_{a}(b).
    \end{equation}
\end{lemma}

\subsection{Strong asymptotic freeness of GOE and deterministic matrices}

We establish a strong asymptotic freeness result for GOE and deterministic
matrices, which is the real analogue of the GUE result in \cite{male}.
The proof is provided in Appendix \ref{app:strongfreeness}.

Fix integers $p,q \geq 0$. Let
$X_1,\ldots,X_p \in \R^{N \times N}$ be independent
GOE matrices, with diagonal entries distributed as $\N(0,2/N)$ and off-diagonal
entries as $\N(0,1/N)$. Let $Y_1,\ldots,Y_q \in \C^{N \times N}$ be
deterministic matrices.  Denote
$\mathbf{X}_N=\left(X_1,\ldots,X_p\right)$ and
$\mathbf{Y}_N=\left(Y_1,\ldots,Y_q\right)$. Let $\tr_N=N^{-1}\Tr$ be the
normalized matrix trace on $\C^{N \times N}$.

Consider an $N$-dependent von Neumann probability space $(\A_N,\tau_N)$.
Suppose $\A_N$ contains $x_1,\ldots,x_p$ and $Y_1,\ldots,Y_q$, where
$x_1,\ldots,x_p$ are free semicircular elements also free
of $Y_1,\ldots,Y_q$, and $\tau_N \equiv \tr_N$ restricted to the von Neumann
subalgebra $\langle Y_1,\ldots,Y_q \rangle$. Denote
$\mathbf{x}=(x_1,\ldots,x_p)$.

\begin{theorem}\label{thm:strongfree}
Suppose $\|Y_j\|\leq C$ for all $j=1,\ldots,q$ and a constant $C>0$.
Then for any fixed non-commutative self-adjoint $*$-polynomial $Q$ in $p+q$
    variables, and any constant $\delta>0$, almost surely for all large $N$,
    \begin{equation}\label{eq:polynomial_inclusion}
        \spec(Q(\mathbf{X}_{N},\mathbf{Y}_{N}))\subset
        \spec (Q(\mathbf{x},\mathbf{Y}_{N}))_\delta.
\end{equation}
\end{theorem}

Here, $\spec(Q(\mathbf{X}_{N},\mathbf{Y}_{N}))$ are the eigenvalues of the
self-adjoint random
matrix $Q(\mathbf{X}_{N},\mathbf{Y}_{N}) \in \C^{N \times N}$, and
$\spec (Q(\mathbf{x},\mathbf{Y}_{N}))_\delta$ is the $\delta$-neighborhood
of the spectrum of the operator $Q(\mathbf{x},\mathbf{Y}_{N}) \in \A_N$.

For our application, we will apply strong asymptotic freeness directly
in the above form. However, we may also obtain as a corollary
the following more usual statement, by the arguments of \cite[Section 7]{male}.

\begin{theorem}\label{thm:strongfreeusual}
	Let $\mathbf{x}=(x_1,\ldots,x_p)$ and $\mathbf{y}=(y_1,\ldots,y_q)$ be
        elements of a fixed von Neumann probability space $(\A,\tau)$,
        such that $x_1,\ldots,x_p$ are free semicircular elements also free
        from $\mathbf{y}$. Assume that almost surely as $N \to \infty$,
        for any fixed self-adjoint $*$-polynomial $P$ in $q$ variables, 
	\[\tr_N[P(\mathbf{Y}_N)]\rightarrow\tau(P(\mathbf{y}))
	\quad\text{and}\quad
	\|P(\mathbf{Y}_N)\|\rightarrow\|P(\mathbf{y})\|.\] 
	Then, almost surely for any self-adjoint
        $*$-polynomial $Q$ in $p+q$ variables,
	\begin{equation}\label{norm_convergence}
            \tr_N[Q(\mathbf{X}_N,\mathbf{Y}_N)]\rightarrow\tau(Q(\mathbf{x},\mathbf{y}))\quad\text{and}\quad
	\|Q(\mathbf{X}_N,\mathbf{Y}_N)\|\rightarrow\|Q(\mathbf{x},\mathbf{y})\|.
	\end{equation} 
\end{theorem}

\subsection{Resolvent approximation using free deterministic
equivalents}\label{sec:resolventapprox}

We also establish a method of approximating bilinear forms in resolvents using
the free deterministic equivalent framework of \cite{speichervargas}.

Fix integers $p,q \geq 0$. We study the resolvent $R(z)$ of a random matrix
\begin{equation}\label{eq:generalW}
W=Q(H_1,\ldots,H_p,B_1,\ldots,B_q) \in \C^{N \times N},
\end{equation}
where $Q$ is any self-adjoint $*$-polynomial,
$H_1,\ldots,H_p$ are deterministic, and $B_1,\ldots,B_q$ are random matrices
orthogonally invariant in law.
For spectral arguments $z$ with constant separation from $\spec(W)$,
and any deterministic unit vectors $u,v \in \C^N$, we will show an approximation
\[u^*R(z)v \approx u^*R_0(z)v,\]
where $R_0(z)$ is a deterministic matrix defined by a free deterministic
equivalent model.

We consider a setup that will allow us to study rectangular matrices, following
\cite{benaychgeorges}: Let $\A_1=\C^{N \times N}$ and
$\tau_1=N^{-1}\Tr$. Fix $k \geq 1$, let $N=N_1+\ldots+N_k$, and consider
the associated $k \times k$ block decomposition of $\A_1$. 
Define mutually orthogonal projections $P_1,\ldots,P_k \in \A_1$ by
\[P_r=\diag(0,\ldots,0,\Id_{N_r},0,\ldots,0)\]
with $\Id_{N_r}$ in the $r$th diagonal block. Then
$(\A_1,\tau_1,P_1,\ldots,P_k)$ is a rectangular probability space in the sense
of \cite{benaychgeorges}. Define the subalgebra $\D \subset \A_1$
generated by $P_1,\ldots,P_k$, given explicitly by
\[\D=\{z_1P_1+\ldots+z_kP_k:z_1,\ldots,z_k \in \C\}.\]
Define also the space of block-diagonal orthogonal matrices
\[\mathcal{O}=\{\diag(O_1,\ldots,O_k):O_r \in \R^{N_r \times N_r},\,O_r^\sT
O_r=\Id \text{ for each } r\}.\]
Consider $H_1,\ldots,H_p,B_1,\ldots,B_q \in \A_1$, where
$H_1,\ldots,H_p$ are deterministic, and
$(B_1,\ldots,B_q)$ is random and equal in joint law to
$(OB_1O^\sT,\ldots,OB_qO^\sT)$ for all $O \in \mathcal{O}$.
For a self-adjoint $*$-polynomial $Q$ in $p+q$ arguments with coefficients
in $\D$, define $W$ by (\ref{eq:generalW}), and define its resolvent
\[R(z)=(W-z\Id)^{-1}.\]

To define the approximation $R_0(z)$, we construct a free deterministic
equivalent model: Let $(\A_2,\tau_2)$ be a second
von Neumann probability space, where $\D \subset \A_2$ and $\tau_1 \equiv
\tau_2$ restricted to $\D$.
Let $\A_2$ have elements $b_1,\ldots,b_q$ satisfying
\begin{equation}\label{eq:bapprox}
N^{-1}\Tr \Big(P(B_1,\ldots,B_q)\Big)
-\tau_2\Big(P(b_1,\ldots,b_q)\Big) \to 0
\end{equation}
almost surely as $N \to \infty$,
for any fixed $*$-polynomial $P$  with coefficients in $\D$.
Define the von Neumann amalgamated free product over $\D$,
\[(\A,\tau)=(\A_1,\tau_1) *_\D (\A_2,\tau_2),\]
so that $(H_1,\ldots,H_p)$ is free of $(b_1,\ldots,b_q)$ with amalgamation
over $\D$. Define the free deterministic equivalent approximation to $W$ by
\[w=Q(H_1,\ldots,H_p,b_1,\ldots,b_q) \in \A.\]
Finally, let $\cH=\langle H_1,\ldots,H_p,\D\rangle$
be the generated von Neumann subalgebra of $\A$, and
let $\tau^\cH:\A \to \cH$ be the
conditional expectation onto $\cH$ that satisfies $\tau(\tau^\cH(a))=\tau(a)$.
Importantly, note that for any $a \in \A$,
\[\tau^\cH(a) \in \cH \subset \A_1 \equiv \C^{N \times N},\]
so that $\tau^\cH(a)$ is an $N \times N$ matrix.
We define the free deterministic approximation $R_0(z)$ of $R(z)$ by
\begin{equation}\label{eq:R0}
R_0(z)=\tau^\cH((w-z)^{-1}).
\end{equation}
We now state our approximation result, whose proof is in Appendix
\ref{app:resolventapprox}.

\begin{theorem}[Resolvent approximation]\label{thm:resolventapprox}
For some constants $C,c>0$, suppose $c<N_r/N<C$, $\|H_i\|<C$, and
$\|B_j\|<C$ for all $r,i,j$,
almost surely for all large $N$. Fix any constant $\delta>0$ and set
\[\mathbb{D}=\{z \in \C:\dist(z,\spec(w)) \geq \delta \text{ and }
\dist(z,\spec(W)) \geq \delta\}.\]
Then for any (sequence of) deterministic unit vectors $u,v \in \C^N$,
almost surely as $N \to \infty$,
\begin{equation}\label{eq:resolventapprox}
\sup_{z \in \mathbb{D}} |u^*R(z)v-u^*R_0(z)v| \to 0.
\end{equation}
\end{theorem}

Taking $k=1$ yields a result for square orthogonally invariant matrices,
where $(\A,\tau)$ is the von Neumann free product over $\D \equiv \C$. We
consider $k \geq 2$ to encompass applications with rectangular matrices, where
each $H_i,B_j$ typically
has a single off-diagonal block which is non-zero.  We are then
interested in $*$-polynomials $Q$ that are $(1,1)$-simple, i.e.\
$W$ and $w$ satisfy
\[W=P_1WP_1, \qquad w=P_1wP_1.\]
Denote by $W_{11} \in \C^{N_1 \times N_1}$ the $(1,1)$-block of $W$.
Corresponding to $\C^{N_1 \times N_1}$ is a ``compressed algebra''
$\A^c=\{P_1aP_1:a \in \A\}$ with unit $P_1$ \cite{speichervargas}.
Denote by $w_{11} \in \A^c$ and $\spec(w_{11})$ the element $w$ and its
spectrum, viewed as a self-adjoint operator in $\A^c$. We then
have the following corollary.

\begin{corollary}\label{cor:resolventapprox}
In the setting of Theorem \ref{thm:resolventapprox}, suppose in addition that
$W=P_1WP_1$ and $w=P_1wP_1$, and let $W_{11}$ and $w_{11}$ be as above.
Let $(R_0(z))_{11} \in \C^{N_1 \times N_1}$
be the $(1,1)$-block of $R_0(z)=\tau^\cH((w-z)^{-1})$, and set
\[\mathbb{D}_1=\{z \in \C:\dist(z,\spec(w_{11})) \geq \delta \text{ and }
\dist(z,\spec(W_{11})) \geq \delta\}.\]
Then for any (sequence of) deterministic unit vectors $u_1,v_1 \in \C^{N_1}$,
almost surely as $N \to \infty$,
\begin{equation}\label{eq:resolventapprox2}
\sup_{z \in \mathbb{D}_1} |u_1^*(W_{11}-z\Id)^{-1}v_1-u_1^*(R_0(z))_{11}v_1|
\to 0.
\end{equation}
\end{corollary}

\section{Analysis of the linear mixed model}\label{sec:proof}

In this section, we give a high-level outline of the proofs of
Theorems \ref{thm:outliers} and
\ref{thm:eigenvectors}, which follow the perturbative approach of \cite{BGN}.
We present the main steps of the computations,
deferring technical details to Appendix \ref{appendix:mixedmodel}.

We assume implicitly throughout that
Assumptions \ref{assump:alpha} and \ref{assump:asymptotic} hold. We denote
by $C,c>0$ constants which may change from instance to instance.
We fix a constant $\delta > 0$, and define
\[
U_{\delta} = \{z \in \C : \dist(z, \supp(\mu_0)) > \delta\}.
\]
We denote $\|X\|_\infty=\max_{i,j} |X_{i,j}|$.
For $n$-dependent matrices $X_1(z), X_2(z)$ of the same (bounded) dimension,
we write
\[X_1(z) \sim X_2(z)\]
if almost surely as $n \to \infty$, we have
\[
\sup_{z \in U_{\delta}} \|X_1(z) - X_2(z)\|_\infty \to 0.
\]

\subsection{Model and deterministic equivalent
measure}\label{sec:det-eq-measure}

We first clarify the form of $\hSigma$ and the free probability
interpretation of the measure $\mu_0$. Introducing $\Gamma_r \in \R^{\ell_r \times p}$
as in (\ref{eq:Gamma}), and defining
\[\Xi_r=\frac{1}{\sqrt{n_r}}(\xi_{ij}^{(r)})_{i,j} \in \R^{n_r \times \ell_r},
\qquad E_r=\begin{pmatrix} - & \eps_1^{(r)} & - \\ & \vdots & \\
- & \eps_n^{(r)} & - \end{pmatrix} \in \R^{n_r \times p},\]
the random effect matrix $\alpha_r$ is written concisely as
\begin{equation}\label{eq:alphar}
    \alpha_r=\sqrt{n_r}\,\Xi_r\Gamma_r+E_r.
\end{equation}
Write further
\begin{equation}\label{eq:ErGrHr}
    E_r=\sqrt{n_r}\,G_rH_r,
\end{equation}
where $G_r \in \R^{n_r \times p}$ has i.i.d.\ $\N(0,1/n_r)$ entries and
$H_r=\oSigma^{1/2} \in \R^{p \times p}$. Then, when $\ell_r=0$ and
$\alpha_r=E_r$ for all $r$, we obtain
\[\hSigma=W \equiv \sum_{r,s=1}^k H_r^\sT G_r^\sT F_{rs}G_sH_s,\]
where $\{F_{rs}:r,s=1,\ldots,k\}$ are defined in (\ref{eq:Frs}).
More generally, we have
\begin{equation}\label{eq:hSigmarepr}
    \hSigma = \sum_{r,s=1}^k
    (\Xi_r \Gamma_r + G_r H_r)^\sT F_{rs}(\Xi_s \Gamma_s + G_s H_s) = W + P
\end{equation}
for $W$ as above, and for the low-rank perturbation
\begin{equation}
P = \sum_{r, s = 1}^k \Big(\Gamma_r^\sT \Xi_r^\sT F_{rs} G_s H_s + H_r^\sT
G_r^\sT F_{rs} \Xi_s \Gamma_s + \Gamma_r^\sT\Xi_r^\sT F_{rs}\Xi_s
\Gamma_s\Big).\label{eq:P}
\end{equation}

The proof of Theorem \ref{thm:bulk} in \cite{fanjohnstonebulk} used a
free probability approach. As the matrices $G_r$ and $F_{rs}$ in this model are
rectangular, asymptotic freeness was formally expressed using the ideas
of \cite{benaychgeorges}, by embedding these matrices in a larger
square matrix, and establishing asymptotic freeness with amalgamation over a
subalgebra generated by block-identity matrices along the diagonal.

More specifically, the proof in \cite[Section 4]{fanjohnstonebulk} 
illustrates that $\mu_0$ is a
spectral measure in the following model: Set $N = (k + 1) p + n_1+\ldots+n_k$.
Embed $\{F_{rs},G_r,H_r:r,s=1,\ldots,k\}$ into $\C^{N
\times N}$ by zero-padding, in the following blocks of the $(2k+1) \times
(2k+1)$ block decomposition for
$\C^N=\C^p \oplus \cdots \oplus \C^p \oplus \C^{n_1} \cdots \oplus \C^{n_k}$:
\begin{equation}\label{block_decomposition}
    \begin{pmatrix} & H_1^* & \cdots & H_k^* & & & \\
H_1 &&  & & G_1^* & & \\
\vdots&  & & & & \ddots & \\
H_k & & && & & G_k^* \\
& G_1 & && F_{11} & \cdots & F_{1k}\\
& & \ddots &  & \vdots &\ddots & \vdots\\
& & & G_k & F_{k1} & \cdots & F_{kk}
\end{pmatrix}
\end{equation}
Denote by $\tilde{F}_{rs}$, $\tilde{G}_r$, and $\tilde{H}_r \in \C^{N \times N}$
these embedded matrices. Consider the mutually orthogonal projections
\[P_0=\diag(\Id_p,0,\ldots,0),\quad \ldots,\quad
P_{2k}=\diag(0,\ldots,0,\Id_{n_k})\]
corresponding to the $2k+1$ diagonal blocks of $\C^{N \times N}$. Then the block
structure of this embedding induces an asymptotic freeness of
the families $\{F_{rs}\}$, $\{H_r\}$, and individual matrices $G_1,\ldots,G_k$
with amalgamation over the
subalgebra generated by these projections $\{P_0,\ldots,P_{2k}\}$.

Let $(\A,\tau)$ be a von Neumann probability space containing
mutually orthogonal projections $p_0,p_1,\ldots,p_{2k}$ which analogously
satisfy $\tau(p_0)=\ldots=\tau(p_k)=p/N$ and $\tau(p_{k+r})=n_r/N$ for each
$r=1,\ldots,k$. Let $(\A,\tau)$ also contain
$\{f_{rs},g_r,h_r:r,s=1,\ldots,k\}$ such that
\begin{enumerate}
\item \label{1} $p_{k+r}f_{rs}p_{k+s}$, $p_{k+r}g_rp_r=g_r$, and $p_rh_rp_0=h_r$.
\item \label{2} For any non-commutative $*$-polynomial $Q$ of $k$ variables,
\[\tau(Q(h_1,\ldots,h_k))=N^{-1}\Tr Q(\tilde{H}_1,\ldots,\tilde{H}_k).\]
Similarly, for any non-commutative $*$-polynomial $Q$ of $k^2$ variables,
\[\tau(Q(f_{11},f_{12},\ldots,f_{kk}))=N^{-1}\Tr
Q(\tilde{F}_{11},\tilde{F}_{12},\ldots,\tilde{F}_{kk}).\]
\item \label{3} For each $r \in \{1,\ldots,k\}$ and $l \geq 0$,
\[\frac{N}{p}\tau((g_r^*g_r)^l)=\int x^l \nu_{p/n_r}(dx)\]
where $\nu$ is the Marcenko-Pastur law with parameter $p/n_r$.
\item \label{4} The families $\{f_{rs}:r,s=1,\ldots,k\}$, $\{h_r:r=1,\ldots,k\}$, and
individual elements $g_1,\ldots,g_k$ are free with amalgamation over the
von Neumann subalgebra $\mathcal{D}=\langle p_0,\ldots,p_{2k} \rangle$.
\end{enumerate}

Define a free deterministic equivalent for $W$ by
\begin{equation}\label{def:w}
w=\sum_{r,s=1}^k h_r^*g_r^*f_{rs}g_sh_s.
\end{equation}
Only the $(0,0)$-block of $w$ is non-zero---this corresponds to $W$ belonging to
the $(0,0)$ block in the embedded space (\ref{block_decomposition}).
Thus $w$ is an element of
the compressed algebra $\A^c=\{a \in \A:a=p_0ap_0\}$, which has
unit $p_0$ and trace $\tau^c(a)=(N/p)\tau(p_0ap_0)$.
The analysis of \cite[Section 4]{fanjohnstonebulk} shows that
the law $\mu_0$ in Theorem \ref{thm:bulk} is the $\tau^c$-distribution of $w$.
This means that for any continuous function $f:\R \to \C$, we have
\[\int f(x) d\mu_0(x)=\tau^c(f(w))\]
where $f(w)$ is defined by the functional calculus on $\A^c$.
Since $\tau$ is a faithful trace, so is $\tau^c$ as a trace on $\A^c$,
and thus (cf.\ \cite[Prop.\ 3.13 and 3.15]{nicaspeicher})
\begin{equation}\label{eq:support_mu_0}
\supp(\mu_0)=\spec(w)
\end{equation}
where $\spec(w)$ is the spectrum of $w$ as an element of $\A^c$.

\subsection{Master equation}\label{subsec:master}
Following \cite{BGN}, we first establish a ``master equation'' characterizing
outlier eigenvalues of $\hSigma$.

Recall the form (\ref{eq:hSigmarepr}) for $\hSigma$.
Letting $\ell$ be the rank of $\Gamma$ (so $\ell \leq \ell_+$), write
\[
\Gamma = \wGamma Q^\sT
\]
where $Q \in \R^{p \times \ell}$ contains the right singular vectors of
$\Gamma$. We have $Q^\sT Q = \Id_{\ell}$ and $\|\wGamma\| \leq C$.
Denote the resolvent of $W$ by
\[
R(z) = (W - z \Id)^{-1}.
\]
Define the block-diagonal matrices
\[
\Xi = \left[\begin{matrix}\Xi_1 & & & \\ & \Xi_2 & & \\ & & \ddots & \\ & & &
\Xi_k \end{matrix}\right] \in \R^{n_+ \times \ell_+},
\quad G = \left[\begin{matrix} G_1 & & & \\ & G_2 & & \\ & & \ddots & \\ & & &
G_k \end{matrix}\right] \in \R^{n_+ \times kp}.
\]
Finally, define $H \in \R^{kp \times p}$ as the vertical stacking of
$\{H_r\}_{r = 1}^k$, and set
\[S(z)=\Xi^\sT FGHR(z)Q.\]

In Appendix \ref{app:approx-trace}, we write the low-rank perturbation matrix
$P$ in (\ref{eq:P}) as $P=P_1P_2$ for two rectangular matrices $P_1$ and $P_2$.
We then apply the identity $\det(\Id+R(z)P_1P_2)=\det(\Id+P_2R(z)P_1)$ to obtain
the following result.
\begin{lemma} \label{lem:master-eq}
The eigenvalues of $\hSigma$ which are not eigenvalues of $W$ are the roots of
$\det \hK(z) = 0$, where
\begin{equation} \label{eq:k-def}
    \hK(z) = \Id + \left[\begin{matrix} S(z) \cdot \wGamma^\sT & \Xi^\sT
        F G H R(z) H^\sT G^\sT F \Xi \cdot \wGamma + S(z) \cdot \wGamma^\sT
    \Xi^\sT F \Xi \wGamma  \\ Q^\sT R(z) Q \cdot \wGamma^\sT &  S(z)^\sT \cdot
    \wGamma  +  Q^\sT R(z) Q \cdot \wGamma^\sT \Xi^\sT F \Xi \wGamma  \end{matrix}\right].
\end{equation}
\end{lemma}

Denote the four blocks of this matrix as $\hK_{11},\hK_{12},\hK_{21},\hK_{22}$.
When $\hK_{11}$ is invertible, the condition
$\det \hK(z)=0$ is equivalent to $\det \hT(z)=0$ for the Schur complement
\begin{equation} \label{eq:def-sc}
\hT(z) = \hK_{22}(z) - \hK_{21}(z) \hK_{11}(z)^{-1} \hK_{12}(z).
\end{equation}
Observe that each matrix $\Xi^\top (\ldots)\Xi$ in the definition of $\hK$ has
bounded dimension $\ell_+ \times \ell_+$, each matrix $\Xi^\top (\ldots) Q$ has
bounded dimension $\ell_+ \times \ell$, and $\Xi$ is independent of
$G$ and $R(z)$. Then, conditioning on $G$ and $R(z)$ and applying
concentration inequalities for linear and bilinear forms in $\Xi$, we obtain
that $\hT(z)$ is approximated by a matrix
\begin{equation} \label{eq:def-wk}
\vT(z) = \Id + Q^\sT R(z) Q \cdot \wGamma^\sT \left(\sum_{r = 1}^k n_r^{-1}
    \Tr_r[F-FGHR(z)H^\sT G^\sT F]\Id_{\ell_r}\right) \wGamma.
\end{equation}
This is formalized in the following result, proven in
Appendix \ref{app:approx-trace}.
\begin{lemma} \label{lem:schur-comp}
We have that $S(z) \sim 0$, $\hK_{11}(z) \sim \Id_{\ell_+}$, and $\hT(z) \sim \vT(z)$.
\end{lemma}

The outlier eigenvalues of $\hSigma$ will be approximate roots of $0=\det
\vT(z)$, where this matrix $\vT(z)$ no longer depends
on the randomness in $\Xi$.

\subsection{Approximation by deterministic equivalents}

The main step of the proof is to approximate the $G$- and $R(z)$-dependent terms
appearing in (\ref{eq:def-wk}) by deterministic quantities. We do this using a
free deterministic equivalent approach. Define 
\[
    \wT(z) = \Id + Q^\sT (z \Id +\b \cdot \oSigma)^{-1} Q \cdot
    \wGamma^\sT \diag_\ell(\b)\wGamma,
\]
with notation as in Theorem \ref{thm:outliers}.
Our goal is to show the following lemma.

\begin{lemma}\label{lem:master-limit}
We have $\vT(z) \sim \wT(z)$.
\end{lemma}

This requires approximating the two terms in $\vT$ by those in $\wT$.
For approximating the first term, as perhaps can be guessed from the form
of the Stieltjes transform (\ref{eq:m0}), the matrix
$-(z\Id+\b \cdot \oSigma)^{-1}$ is a deterministic equivalent for
the resolvent $R(z)$. We verify this in the following result, using the
resolvent approximation techniques in Section \ref{sec:resolventapprox}
and Theorem \ref{thm:resolventapprox}.

\begin{proposition} \label{prop:det-res-val}
We have $Q^\sT R(z) Q \sim -Q^\sT (z \Id + \b \cdot \oSigma)^{-1} Q$.
\end{proposition}
\begin{proof}
    The von Neumann probability space $(\A,\tau)$ in Section
    \ref{sec:det-eq-measure} may be constructed as follows:
    Let $(\A_1,\tau_1)=(\C^{N \times N},N^{-1}\Tr)$, containing
    the embeddings of the matrices
    $H_1,\ldots,H_k$ and $P_0,\ldots,P_{2k}$. Denote these elements of $\A_1$
    also by $h_r$ and $p_r$. Construct a von Neumann probability space
    $(\A_2,\tau_2)$ also containing $p_0,\ldots,p_{2k}$ and
    elements $\{f_{rs},g_r:r,s=1,\ldots,k\}$ satisfying all
    required conditions on their joint law under $\tau_2$.
    Let $(\A,\tau)$ be the von Neumann amalgamated free product over
    $\langle p_0,\ldots,p_{2k} \rangle$.

    Let $w=\sum_{r,s} h_r^*g_r^*f_{rs}g_sh_s \in \A$.
    By Corollary \ref{cor:resolventapprox} applied to each pair of
    columns of $Q$, we find that 
    \[Q^\sT R(z) Q \sim Q^\sT P_0 \tau^\cH((w - z)^{-1}) P_0 Q\]
    where $P_0\tau^\cH((w - z)^{-1})P_0$ is identified with its $(0,0)$-block
    as an element of $\C^{p \times p}$. This $\tau^\cH$ trace was computed in
    \cite[Equation (4.12)]{fanjohnstonebulk} to be
    \[
    \tau^\cH\Big((w - z)^{-1}\Big) = - \left(z+\sum_{r = 1}^k h_r^* h_r
            b_r(z)\right)^{-1},
    \]
    using the identification $\beta_r(z)=-b_r(z)$
    at the conclusion of the proof of \cite[Lemma 4.4]{fanjohnstonebulk}.
    The $(0,0)$-block of this matrix is exactly
    \[-\left(z+\sum_{r=1}^k H_r^\sT H_r b_r(z)\right)^{-1}
    =-(z+\b \cdot \oSigma)^{-1}.\qquad \qedhere\]
\end{proof}

Lemma \ref{lem:master-limit} now follows by applying Proposition \ref{prop:det-res-val} and
the following approximation for the second term of $\vT(z)$.

\begin{proposition} \label{prop:det-eq}
    For each $t \in \{1,\ldots,k\}$, we have
\[
n_t^{-1} \Tr_t[F - F G H R(z) H^\sT G^\sT F] \sim -b_t(z).
\]
\end{proposition}

In the remainder of this section, we prove Proposition
\ref{prop:det-eq}.
We apply a computation using the augmented Cauchy- and $\cR$-transforms of Section \ref{sec:augmentedtransforms}.
In the von Neumann probability space $(\A, \tau)$ of Section \ref{sec:det-eq-measure}, let
$\cH = \langle h_1, \ldots, h_k\rangle$, $\cG = \langle g_1,
\ldots, g_k\rangle$, $\cF=\langle f_{11},f_{12},\ldots,f_{kk} \rangle$
and $\D=\langle p_0,\ldots,p_{2k} \rangle$
be the generated von Neumann subalgebras of $\A$. Define the elements
\begin{equation} \label{eq:wvu-def}
w = \sum_{r, s = 1}^k h_r^* g_r^* f_{rs} g_s h_s \qquad v = \sum_{r, s = 1}^k g_r^* f_{rs} g_s \qquad u = \sum_{r, s = 1}^k f_{rs}.
\end{equation}
For any $r, s, t \in \{1, \ldots, k\}$ define 
\[
a_{rts} = h_r^* g_r^* f_{rt} f_{ts} g_s h_s \qquad b_{rts} = g_r^* f_{rt} f_{ts} g_s \qquad c_{rts} = f_{rt} f_{ts}.
\]
Our goal is to compute
\[\sum_{r,s=1}^k \tau(f_{ts}g_{s}h_{s}(w-z)^{-1}h_r^*g_r^*f_{rt})
=\sum_{s,t=1}^k \tau(a_{rts}(w-z)^{-1}),\]
which is the free approximation for $\Tr_t F G H R(z) H^\sT G^\sT F$.

For $a \in \A$ and $h \in \cH$, define the $\cH$-valued conditional expectation $\tau^\cH(a)$, Cauchy-transform
$G^\cH_a(h)$, and $\cR$-transform $\cR^\cH_a(h)$, and similarly for $\cG$ and $\D$.
For each $i \in \{0,\ldots,2k\}$, denote
\[\tau_i(a)=\tau(p_i)^{-1} \tau(p_iap_i)\]
and note that $\tau^\D(a)=\sum_i \tau_i(a)p_i$.
For a sufficiently large constant $C>0$, define
\[
\mathbb{D} = \{z \in \C : |z| > C\}.
\]
We define the following analytic functions
$\{\alpha_i\}_{i = 0}^{2k}$, $\{\beta_i\}_{i = 0}^{2k}$, $\{d_i\}_{i=0}^{2k}$,
$\{\gamma_j\}_{j = 0}^{2k}$, $\{\delta_j\}_{j = 0}^{2k}$, and
$\{e_j\}_{j=0}^{2k}$ on $\mathbb{D}$,
also used in \cite{fanjohnstonebulk}: For $i=1,\ldots,k$, define
\begin{equation} \label{eq:alpha-beta-def}
\alpha_i = \tau_i(h_i G^\cH_w(z) h_i^*),\qquad \beta_i =
\tau_i\left(R^\D_v\left(\sum_{i = 1}^k \alpha_i p_i\right)\right).
\end{equation}
Set $\alpha_0 = \alpha_{k + 1} = \cdots = \alpha_{2k} = |z|^{-1}$ and $\beta_0 =
\beta_{k + 1} = \cdots = \beta_{2k} = 0$, and
\[d_i=\alpha_i^{-1} + \beta_i, \quad  d = \sum_{i = 0}^{2k} d_ip_i.\]
Now, for $j=1,\ldots,k$, define
\begin{equation} \label{eq:gamma-delta-def}
\gamma_{j + k} = \tau_{j + k}(g_j G^\cG_v(d) g_j^*), \quad \delta_{j + k} =
\tau_{j + k}\left(R^\D_u\left(\sum_{j = k + 1}^{2k} \gamma_j p_j\right)\right).
\end{equation}
Set $\gamma_0 = \gamma_{1} = \cdots = \gamma_{k} = |z|^{-1}$ and $\delta_0 =
\delta_{1} = \cdots = \delta_k = 0$, and
\[
e_j=\gamma_j^{-1} + \delta_j, \quad e = \sum_{j = 0}^{2k} e_jp_j.
\]

The following identities are shown in \cite{fanjohnstonebulk}.  

\begin{proposition} \label{prop:free-identity}
For all $z \in \mathbb{D}$,
\begin{itemize}
\item[(a)] $\sum_{i = 0}^{2k} \alpha_i p_i = G_v^\D(d)$.

\item[(b)] $\sum_{j = 0}^{2k} \gamma_j p_j = G^\D_u(e)$.

\item[(c)] The quantities $a_r = - \frac{p\alpha_r}{n_r}$ and $b_r = - \beta_r$
    satisfy the relations (\ref{eq:arecursion}--\ref{eq:brecursion}).
\item[(d)] For $r=1,\ldots,k$, we have $e_{r + k} = - a_r^{-1}$. 
\end{itemize}
\end{proposition}
\begin{proof}
(a) follows from \cite[Equation (4.15)]{fanjohnstonebulk}, (b) follows from
    \cite[Equation (4.21)]{fanjohnstonebulk}, (c) is shown at the end of the
    proof of \cite[Lemma 4.4]{fanjohnstonebulk}, and (d)
    follows from \cite[Equation (4.28)]{fanjohnstonebulk}.
\end{proof}

The following identities are similar to \cite[Lemma 4.3]{fanjohnstonebulk}.  

\begin{proposition} \label{prop:rh}
We have
\begin{align*}
R^{\cH }_{a_{rts}, w}(G^\cH_w(z)) &= h_r^* h_r \tau_r\left[R^{\D }_{b_{rts},
    v}\Big(G^\D_v(d)\Big)\right],\\
R^{\cG }_{b_{rts}, v}(G^\cG_v(d)) &= g_r^* g_r \tau_{r + k}\left[R^{\D }_{c_{rts}, u}\Big(G^\D_u(e)\Big)\right].
\end{align*}
\end{proposition}
\begin{proof}
For the first equality, notice that for $c = G^\cH_w(z)$, we have
\begin{align*}
&\kappa_l^\cH(a_{rts}, cw, \ldots, cw)\\
&= \sum_{\substack{r_2, \ldots, r_l = 1\\ s_2, \ldots, s_l = 1}}^k \kappa_l^\cH\Big(h_r^* g_r^* f_{rt} f_{ts} g_s h_s, c h_{r_2}^* g_{r_2}^* f_{r_2 s_2} g_{s_2} h_{s_2}, \ldots, c h_{r_{l}}^* g_{r_l}^* f_{r_l s_l} g_{s_l} h_{s_l}\Big)\\
&= \sum_{\substack{r_2, \ldots, r_l = 1\\ s_2, \ldots, s_l = 1}}^k h_r^*
    \kappa_l^\cH\Big( g_r^* f_{rt} f_{ts} g_s , h_s c h_{r_2}^* g_{r_2}^* f_{r_2
    s_2} g_{s_2} , \ldots, h_{s_{l-1}}c h_{r_{l}}^* g_{r_l}^* f_{r_l s_l} g_{s_l}\Big)  h_{s_l}\\
&= \sum_{\substack{r_2, \ldots, r_l = 1\\ s_2, \ldots, s_l = 1}}^k h_r^* \kappa_l^\D\Big( g_r^* f_{rt} f_{ts} g_s, \tau^\D(h_s c h_{r_2}^*) g_{r_2}^* f_{r_2 s_2} g_{s_2}, \ldots, \tau^\D(h_{s_{l - 1}} c h_{r_{l}}^*) g_{r_l}^* f_{r_l s_l} g_{s_l}\Big) h_{s_l},
\end{align*}
where we apply \cite[Theorem 3.6]{NSS} and $\D$-freeness of $\{\cF,\cG\}$ and
    $\cH$ in the last step.  Notice now that $\tau^\D(h_s c h_r^*) = 0$ unless
    $s = r$, that for any $d' \in \D$ we have $h_r^* d' h_r = h_r^* h_r \tau_r(d')$, and that 
\[
\tau^\D(h_r c h_r^*) g_r^* = \tau_r(h_r c h_r^*) p_r g_r^* = \Big(\sum_{i = 0}^{2k} \alpha_i p_i\Big) g_r^*.
\]
Therefore, applying Proposition \ref{prop:free-identity}(a) and defining $c' =
G^\D_v(d)$, the above is equal to
    \[h_r^* h_r \sum_{r_3, \ldots, r_l = 1}^k
    \tau_r\Big(\kappa_l^\D\Big( g_r^* f_{rt} f_{ts} g_s, c' g_{s}^* f_{s r_3}
    g_{r_3}, c' g_{r_3}^* f_{r_3r_4} g_{r_4}, \ldots, c' g_{r_l}^* f_{r_l r}
    g_{r}\Big)\Big).\]
On the other hand, using $g_s=g_sp_s$ and $p_sc'p_r=0$ unless $s=r$, we have
\begin{align*}
    &\kappa_l^\D(b_{rts},c'v,\ldots,c'v)\\
    &=\sum_{\substack{r_2,\ldots,r_l=1 \\ s_2,\ldots,s_l=1}}^k
\kappa_l^\D\Big(g_r^*f_{rt}f_{ts}g_s,c'g_{r_2}^*f_{r_2s_2}g_{s_2},
\ldots,c'g_{r_l}^*f_{r_ls_l}g_{s_l}\Big)\\
    &=\sum_{\substack{r_2,\ldots,r_l=1 \\ s_2,\ldots,s_l=1}}^k
    \kappa_l^\D\Big(g_r^*f_{rt}f_{ts}g_s,p_sc'p_{r_2}g_{r_2}^*f_{r_2s_2}g_{s_2},
    \ldots,p_{s_{l-1}}c'p_{r_l}g_{r_l}^*f_{r_ls_l}g_{s_l}\Big)\\
    &=\sum_{r_3,\ldots,r_l=1}^k
    \kappa_l^\D\Big( g_r^* f_{rt} f_{ts} g_s, c' g_{s}^* f_{s r_3}
    g_{r_3}, c' g_{r_3}^* f_{r_3r_4} g_{r_4}, \ldots, c' g_{r_l}^* f_{r_l r}
    g_{r}\Big).
\end{align*}
Comparing with the above,
        \[\kappa_l^\cH(a_{rts},cw,\ldots,cw)
        = h_r^* h_r \tau_r\Big(\kappa_l^\D( b_{rts}, c'v, \ldots, c'v)\Big).\]
Summing over $l \geq 1$ yields the first identity.
    The proof of the second identity is
    exactly parallel, using Proposition \ref{prop:free-identity}(b) in place of
    Proposition \ref{prop:free-identity}(a).
\end{proof}

\begin{proposition} \label{prop:free-comp}
We have
  \[
\tau(a_{rts}(z - w)^{-1}) = \tau(c_{rts}(e - u)^{-1}).
\]
\end{proposition}
\begin{proof}
    Note first that
\[
\tau(a_{rts}(z - w)^{-1}) = \tau\left(G^\cH_{a_{rts, w}}(z)\right)
    =\tau(p_0)\tau_0\left(G^\cH_{a_{rts, w}}(z)\right).
\]
Substituting the expression of Proposition \ref{prop:rh} into the identity 
\[
G^\cH_{a_{rts}, w}(z) = R^{\cH }_{a_{rts}, w}(G^\cH_w(z)) G^\cH_w(z)
\]
of Lemma \ref{lem:mod-gr}, we find that
\[
G^\cH_{a_{rts}, w}(z) = h_r^* h_r \cdot G_w^\cH(z) \tau_r\left[R^{\D }_{b_{rts}, v}\Big(G^\D_v(d)\Big)\right],
\]
from which we obtain
\[
\tau_0[G^\cH_{a_{rts}, w}(z)] = \tau_0[h_r^* h_r G_w^\cH(z)] \tau_r\left[R^{\D
    }_{b_{rts}, v}\Big(G^\D_v(d)\Big)\right].
\]
    Noting that $\tau_0[h_r^* h_r G^\cH_w(z)] = \frac{\tau(p_r)}{\tau(p_0)}\alpha_r$, we obtain
\begin{multline*}
\tau_0[G^\cH_{a_{rts}, w}(z)] = \frac{\tau(p_r)}{\tau(p_0)}\tau_r\left[
    R^{\D }_{b_{rts}, v}(G^\D_v(d)) \alpha_r\right]\\
= \frac{\tau(p_r)}{\tau(p_0)}\tau_r\left[R^{\D }_{b_{rts},
    v}(G^\D_v(d))G^\D_v(d) \right] = \frac{\tau(p_r)}{\tau(p_0)} \tau_r[G^\D_{b_{rts},
    v}(d)] = \frac{\tau(p_r)}{\tau(p_0)} \tau_r[G^\cG_{b_{rts}, v}(d)],
\end{multline*}
where in the second equality we replace $\alpha_r$ by $G^\D_v(d) = \sum_{i =
    0}^{2k} \alpha_i p_i$.  Substituting Proposition \ref{prop:rh} into the identity
\[
G^\cG_{b_{rts}, v}(d) = R^{\cG }_{b_{rts}, v}(G^\cG_v(d)) G^\cG_v(d),
\]
we find that
\[
G^\cG_{b_{rts}, v}(d) = g_r^* g_r G_v^\cG(d) \tau_{r + k}\left[R^{\D }_{c_{rts}, u}\Big(G^\D_u(e)\Big)\right].
\]
Noting that $\tau_r(g_r^* g_r G^\cG_v(d)) = \frac{\tau(p_{r + k})}{\tau(p_r)}
    \gamma_{r + k}$, we find similarly that
\begin{multline*}
\tau_r[G^\cG_{b_{rts}, v}(d)] = \frac{\tau(p_{r + k})}{\tau(p_r)} \tau_{r +
    k}\left[R^{\D }_{c_{rts}, u}\Big(G^\D_u(e)\Big)\gamma_{r + k}\right]\\
= \frac{\tau(p_{r + k})}{\tau(p_r)} \tau_{r + k}\left[R^{\D
    }_{c_{rts}, u}\Big(G^\D_u(e)\Big)G^\D_u(e) \right]
= \frac{\tau(p_{r + k})}{\tau(p_r)} \tau_{r + k}\left[G^\D_{c_{rts}, u}\Big(e\Big)\right].
\end{multline*}
Putting everything together, we conclude that
\begin{align*}
\tau(a_{rts}(z - w)^{-1}) = \tau(p_{r + k}) \tau_{r + k}\left[G^\D_{c_{rts},
    u}\Big(e\Big)\right]=\tau(c_{rts}(e - u)^{-1}).\qquad \qedhere
\end{align*}
\end{proof}

Applying the definitions of $a_{rts}$ and $c_{rts}$,
the asymptotic freeness result in \cite[Theorem 3.9]{fanjohnstonesun}, and
Proposition \ref{prop:free-identity}(d), the above implies
\[\frac{1}{n_t} \Tr_t[F G H R(z) H^\sT G^\sT F]
\sim \frac{1}{n_t} \Tr_t \Big(F (\diag_n(\a^{-1}) + F)^{-1} F \Big),\]
and Proposition \ref{prop:det-eq} now follows from the Woodbury matrix identity.
We defer these details to Appendix \ref{app:approx-trace}.

\subsection{Outlier eigenvectors and eigenvectors}

Combining Lemmas \ref{lem:schur-comp} and \ref{lem:master-limit}, we have shown
that $\hT \sim \wT$. Recalling $\Gamma=\wGamma Q^\top$ and using
$\det(\Id+AB)=\det(\Id+BA)$, we see that
the roots of $0=\det \wT(z)$ are the same as those
of $0=\det T(z)$. Then Theorem \ref{thm:outliers} follows from an application of
Hurwitz's theorem. We defer the technical details of this argument to
Appendix \ref{app:outlier-eigenvalues}.

The proof of Theorem \ref{thm:eigenvectors} uses the following two
results, whose proofs are deferred to Appendix \ref{app:outlier-eigenvectors}.

\begin{proposition} \label{prop:sec-sing}
    In the setting of Theorem \ref{thm:eigenvectors}, $\ker \wT(\lambda)$ has
    dimension exactly $1$, and
    each other singular value of $\wT(\lambda)$ is at least a constant $c \equiv c(\delta) > 0$.
\end{proposition}

\begin{proposition}\label{prop:derbounds}
Denote by $S'(z)$ and $R'(z)$ the derivatives of $S(z)$ and $R(z)$ with respect
to $z$. Then
\[S'(z) \sim 0, \qquad Q^\sT R'(z)Q \sim -Q^\sT\partial_z[(z\Id+\b \cdot
    \oSigma)^{-1}]Q,\]
\[n_t^{-1}\Tr_t[F G H R'(z) H^\sT G^\sT F] \sim b_t'(z).\]
\end{proposition}

\begin{proof}[Proof of Theorem \ref{thm:eigenvectors}]
Since $(\hl, \hv)$ is an eigenvalue-eigenvector pair, we have that $\hl \hv = \hSigma \hv = W \hv + P \hv$, which implies that
\begin{equation} \label{eq:eigenvect}
0 = (\Id + R(\hl) P) \hv.
\end{equation}
Define 
\[
\hv_1 = \Xi^\sT F G H \hv \qquad \text{ and } \qquad \hv_2 = Q^\sT \hv.
\]
Multiplying (\ref{eq:eigenvect}) on the left by
$\left[\begin{matrix} \Xi^\sT F G H \\ Q^\sT \end{matrix}\right]$
and recalling (\ref{eq:k-def}), we obtain
\begin{equation}\label{eq:hKkernel}
0 = \hK(\hl) \left[\begin{matrix} \hv_1 \\ \hv_2 \end{matrix}\right].
\end{equation}
Eliminating $\hv_1$ in this system of equations,
    we get $0 = \hT(\hl) \hv_2$
for the Schur complement $\hT$ from (\ref{eq:def-sc}).
We show in Proposition \ref{prop:T-analytic} that $\wT(z)$ is bounded over
$U_\delta$. Then so is $\wT'(z)$, by the Cauchy integral formula.
Applying a Taylor expansion and
the results $\hl-\lambda \to 0$ and $\hT \sim \wT$ from Theorem
    \ref{thm:outliers} and Lemmas \ref{lem:schur-comp} and \ref{lem:master-limit}, almost surely
$\|\hT(\hl) - \wT(\lambda)\| \to 0$. So also
\[
\|\hT(\hl)^\sT \hT(\hl) - \wT(\lambda)^\sT \wT(\lambda)\| \to 0.
\]
Applying this to $\hv_2$, we find that 
$\|\wT(\lambda)^\sT \wT(\lambda) \hv_2 \| \to 0$,
which implies by Proposition \ref{prop:sec-sing} and the Davis-Kahan theorem that
\begin{equation} \label{eq:eigenvector-dir}
\hv_2 - \|\hv_2\| v_2 \to 0,
\end{equation}
where $v_2$ is a unit vector in $\ker \wT(\lambda)$ with an appropriate choice of sign.

We now compute the limit of $\|\hv_2\|$.  By (\ref{eq:eigenvect}) and the
    definition of $P$, we see that 
    \begin{equation}\label{eq:hvrelation}
- \hv = R(\hl) \Big(Q \wGamma^\sT \hv_1 + (H^\sT G^\sT F \Xi \wGamma + Q \wGamma^\sT \Xi^\sT F \Xi \wGamma) \hv_2\Big). 
    \end{equation}
On the other hand, in the equation (\ref{eq:hKkernel}),
we may solve for $\hv_1$ to obtain
$\hv_1 = - \hK_{11}(\hl)^{-1} \hK_{12}(\hl) \hv_2$
when $\hK_{11}(\hl)$ is invertible. Substituting into (\ref{eq:hvrelation}),
    \begin{equation}\label{eq:hvrelation2}
        \hv = R(\hl)(M_1(\hl)+M_2(\hl))\hv_2
    \end{equation}
    for the matrices \[M_1(\hl)=Q\wGamma^\sT \hK_{11}(\hl)^{-1}
    \hK_{12}(\hl) - Q \wGamma^\sT \Xi^\sT F \Xi
    \wGamma, \quad M_2(\hl)=-H^\sT G^\sT F \Xi \wGamma.\]
Note that $M_1'(z),M_2'(z),R'(z)$ are also bounded over $U_\delta$, on a
high-probability event when $\spec(W) \subset \supp(\mu_0)_{\delta/2}$ and
$\|\Xi\|,\|G\|<C$. Taking the squared
norm of (\ref{eq:hvrelation2}) on both sides and applying
    $\hl-\lambda \to 0$ and a Taylor expansion,
\begin{equation} \label{eq:v-len}
    1 = \sum_{i,j=1}^2 \hv_2^\sT M_i(\hl)^\sT R(\hl)^2 M_j(\hl) \hv_2
    = \sum_{i,j=1}^2 \hv_2^\sT M_i(\lambda)^\sT R(\lambda)^2 M_j(\lambda) \hv_2
    +o(1).
\end{equation}

Applying Lemma \ref{lem:schur-comp} and Propositions \ref{prop:concentration} and \ref{prop:det-eq}, we find that 
    \[Q^\sT M_1(z) \sim \wGamma^\sT \Xi^\sT F G H R(z) H^\sT G^\sT F \Xi
    \wGamma -\wGamma^\sT \Xi^\sT F \Xi \wGamma
    \sim \wGamma^\sT \diag_\ell(\b(z)) \wGamma.\]
    Also, noting that $R(z)^2=R'(z)$ and
    applying Proposition \ref{prop:derbounds},
    \[Q^\sT R(z)^2 Q \sim Q^\sT R'(z)Q \sim -Q^\sT \partial_z[(z\Id+\b(z) \cdot
    \oSigma)^{-1}]Q.\]
Combining these, applying $\Gamma=\wGamma Q^\sT$, and setting
$\hu=\wGamma \hv_2=\Gamma \hv$, we get
\begin{align}
    &\hv_2^\sT M_1(\lambda)^\sT R(\lambda)^2M_1(\lambda)\hv_2\nonumber\\
&=-\hu^\sT \diag_\ell(\b) \Gamma \cdot
\partial_\lambda[(\lambda \Id+\b \cdot \oSigma)^{-1}] \cdot
\Gamma^\sT\diag_\ell(\b)\hu+o(1)\label{eq:M1M1}
\end{align}
where we write as shorthand $\b \equiv \b(\lambda)$.
Applying $R(z)^2=R'(z)$
and Propositions \ref{prop:concentration} and \ref{prop:derbounds}, we
also get $\Xi^\sT F G H R(z)^2 H^\sT G^\sT F \Xi \sim \diag_\ell(\b'(z))$,
and hence
\begin{equation}\label{eq:M2M2}
    \hv_2^\sT M_2(\lambda)^\sT R(\lambda)^2 M_2(\lambda)\hv_2
=\hu^\sT \diag_\ell(\b')\hu+o(1).
\end{equation}
Finally, applying $S'(z) \sim 0$ from Proposition \ref{prop:derbounds},
we get $Q^\sT R(z)^2H^\sT G^\sT F\Xi \sim 0$ and hence
\begin{equation}\label{eq:M1M2}
    \hv_2^\sT M_1(\lambda)^\sT R(\lambda)^2 M_2(\lambda)\hv_2 \to 0.
\end{equation}
Then substituting (\ref{eq:M1M1}), (\ref{eq:M2M2}), and (\ref{eq:M1M2})
into (\ref{eq:v-len}),
\begin{equation} \label{eq:first-ueq}
1=\hu^\sT \Big(-\diag_\ell(\b) \Gamma \cdot \partial_{\lambda}[(\lambda \Id + \b
    \cdot \oSigma)^{-1}]\cdot  \Gamma^\sT \diag_\ell(\b) + \diag_\ell(\b') \Big)
    \hu+o(1).
\end{equation}
Multiplying (\ref{eq:eigenvector-dir}) on the left by $\wGamma$, we find that 
\begin{equation} \label{eq:ev-dir2}
\hu - \|\hv_2\| \wGamma v_2 \to 0.
\end{equation}
Define $\tilde{u}=\wGamma v_2$, and note that $\tilde{u}$
is a non-zero vector in $\ker T(\lambda)$ because $v_2$ is a unit vector in
$\ker \wT(\lambda)$. Then $u=\tilde{u}/\|\tilde{u}\|$ is a unit vector in $\ker
T(\lambda)$, which is unique up to sign by Proposition
\ref{prop:sec-sing}. Substituting (\ref{eq:ev-dir2}) into (\ref{eq:first-ueq})
and recalling the definition of
$\alpha$ in Theorem \ref{thm:eigenvectors}, we find that
\[1=\|\hv_2\|^2\|\tilde{u}\|^2 \cdot \alpha+o(1).\]
Writing (\ref{eq:ev-dir2}) as $\hu-\|\hv_2\|\|\tilde{u}\|u \to 0$
and substituting $\alpha^{-1/2}$ for $\|\hv_2\|\|\tilde{u}\|$ concludes the
proof.
\end{proof}

\appendix

\section{Qualitative phenomena and simulations}\label{appendix:examples}

\subsection{Qualitative phenomena}\label{appendix:qualitative}
We provide the calculations for (\ref{eq:eigenvaluebias}),
(\ref{eq:eigenvaluealias}), (\ref{eq:eigenvecexpansion}), and
(\ref{eq:eigenvecexpansion2}).
Recall that $\alpha_1,\ldots,\alpha_k$ are independent, with independent rows of
mean 0 and covariances $\Sigma_1,\ldots,\Sigma_k$. 
Then in the mixed model (\ref{eq:mixedmodel}),
for any matrix $\hSigma=Y^\sT BY$, we have
\[\E[\hSigma]=\sum_{r,s=1}^k \E[\alpha_r^\sT U_r^\sT BU_s\alpha_s]
=\sum_{r=1}^k \left(\Tr U_r^\sT BU_r\right)\Sigma_r.\]
Recall also $F_{rt}=\sqrt{n_rn_t}(U_r^\sT BU_t)$.
If $\hSigma$ is an unbiased MANOVA estimate for $\Sigma_1$, this implies
\begin{equation}\label{eq:MANOVAcondition}
\frac{1}{n_1}\Tr F_{11}=\Tr U_1^\sT BU_1=1,
\quad \frac{1}{n_r}\Tr F_{rr}=\Tr U_r^\sT BU_r=0 \text{ for all } r \neq 1.
\end{equation}

In Step 4 of the proof of \cite[Lemma 4.4]{fanjohnstonebulk}, it is shown that
$z a_r(z)$ and $b_r(z)$ remain bounded as $|z| \to \infty$.
Then, linearizing the fixed-point equations
(\ref{eq:arecursion}--\ref{eq:brecursion}) for large $z=\lambda$, we obtain
\begin{align*}
a_r(\lambda)&=-\frac{1}{\lambda}\left(\frac{1}{n_r}\Tr \oSigma_r\right)
+O\left(\frac{1}{\lambda^2}\right),\\
b_r(\lambda)&=-\frac{1}{n_r}\Tr F_{rr}+
\frac{1}{n_r}\sum_{t=1}^k a_t(\lambda) \Tr F_{rt}F_{tr}
+O\left(\frac{1}{\lambda^2}\right).
\end{align*}
Substituting the first expression into the second,
applying (\ref{eq:MANOVAcondition}), and recalling the definition of
$c_r$ from (\ref{eq:cr}),
\begin{equation}\label{eq:bexpansion}
b_r(\lambda)=-\1\{r=1\}-\frac{c_r}{\lambda}+O\left(\frac{1}{\lambda^2}\right).
\end{equation}
Under Assumption \ref{assump:asymptotic}, $|c_r|$ is bounded by a constant.

Suppose that $\ell_1=\ell_2=1$ and $\ell_r=0$ for each other $r$,
and write the rows of $\Gamma_1,\Gamma_2$ as
$\sqrt{\mu_1}v_1$ and $\sqrt{\mu_2}v_2$. Assume
$\oSigma_r v_1=\oSigma_r v_2=0$ for every $r$. Then,
recalling $\rho=\langle v_1,v_2 \rangle$, we get
\begin{equation}\label{eq:quadformsimplify}
\Gamma (\lambda \Id+b \cdot \oSigma)^{-1}\Gamma^\sT
=\frac{1}{\lambda}\Gamma\Gamma^\sT
=\frac{1}{\lambda}
\begin{pmatrix} \mu_1 & \rho\sqrt{\mu_1\mu_2} \\
\rho\sqrt{\mu_1\mu_2} & \mu_2 \end{pmatrix}.
\end{equation}
Applying this and (\ref{eq:bexpansion}) to (\ref{eq:Tnew}), we obtain
\begin{equation}\label{eq:Texpansionfull}
T(\lambda)=\Id-
\begin{pmatrix} \mu_1 & \rho\sqrt{\mu_1\mu_2} \\
\rho\sqrt{\mu_1\mu_2} & \mu_2 \end{pmatrix}
\begin{pmatrix}
\frac{1}{\lambda}+\frac{c_1}{\lambda^2}
& 0 \\ 0 & \frac{c_2}{\lambda^2}
\end{pmatrix}+O\left(\frac{\mu}{\lambda^3}\right).
\end{equation}
Taking the determinant of this $2 \times 2$ matrix yields the expansion
\begin{equation}\label{eq:detTexpansion}
\det T(\lambda)= 1-\frac{\mu_1}{\lambda}
-\frac{c_1\mu_1+c_2\mu_2}{\lambda^2}\nonumber
+\frac{c_2\mu_1\mu_2(1-\rho^2)}{\lambda^3}+
O\left(\frac{\lambda \mu+\mu^2}{\lambda^4}\right).
\end{equation}
When $\mu_2 \lesssim \mu_1$ and $\mu_1$ is large,
the largest root of $0=\det T(\lambda)$ takes
the form (\ref{eq:eigenvaluebias}). When $\mu_1=0$ and $\mu_2$ is large,
$0=\det T(\lambda)$ has two roots given by (\ref{eq:eigenvaluealias}).

For (\ref{eq:eigenvecexpansion}) and (\ref{eq:eigenvecexpansion2}),
consider $\lambda$ described by (\ref{eq:eigenvaluebias}),
which satisfies $\lambda \asymp \mu_1 \asymp \mu_2$.
The expression (\ref{eq:eigenvaluebias}) yields
\[\frac{1}{\lambda}=\frac{1}{\mu_1}\left(1-\frac{\text{bias}}{\mu_1}\right)
+o_\mu\left(\frac{1}{\mu^2}\right).\]
Substituting this into the second row of (\ref{eq:Texpansionfull}),
\[T_{21}(\lambda)=-\rho\sqrt{\mu_2/\mu_1}
+c_2\rho^3\sqrt{\mu_2^3/\mu_1^5}+o_\mu(1/\mu),\quad
T_{22}(\lambda)=1-c_2\mu_2/\mu_1^2+o_\mu(1/\mu),\]
where the first terms are $O(1)$ and the second terms are $O(1/\mu)$.
The unit vector $u \in \ker T(\lambda)$ is orthogonal to $(T_{21},T_{22})$, so
it is given (up to sign) by
\begin{equation}\label{eq:ufirstorder}
u=\frac{1}{\sqrt{\mu_1+\rho^2\mu_2}}
\begin{pmatrix} \sqrt{\mu_1} \\ \rho\sqrt{\mu_2} \end{pmatrix}+O_\mu(1/\mu).
\end{equation}
To approximate $\alpha$ in (\ref{eq:alpha-def}), recall
(\ref{eq:quadformsimplify}) and (\ref{eq:bexpansion}). Then
\[\alpha=u^\top \begin{pmatrix} -1 & 0 \\ 0 & 0 \end{pmatrix}
\cdot \frac{1}{\lambda^2}\Gamma\Gamma^\top
\cdot \begin{pmatrix} -1 & 0 \\ 0 & 0 \end{pmatrix} u
+O(1/\lambda^2)
=(\mu_1+\rho^2\mu_2)^{-1}+O_\mu(1/\mu^2),\]
so
\begin{equation}\label{eq:alphaexpansion}
\alpha^{-1/2}=\sqrt{\mu_1+\rho^2\mu_2}+O_\mu(1/\sqrt{\mu}).
\end{equation}
Multiplying (\ref{eq:ufirstorder}) and (\ref{eq:alphaexpansion}) yields
(\ref{eq:eigenvecexpansion}).

For (\ref{eq:eigenvecexpansion2}), one may check that
\[w=\frac{1}{\sqrt{\mu_1\mu_2(1-\rho^2)}}
\begin{pmatrix} -\rho\sqrt{\mu_2} \\ \sqrt{\mu_1}
\end{pmatrix}\]
is the desired vector $w \in \R^2$ for which $\Gamma^\top w$ is the
unit vector parallel to the component of $v_2$ orthogonal to $v_1$.
Let us write a second-order correction to (\ref{eq:ufirstorder}) as 
\[u=\frac{1}{\sqrt{\mu_1+\rho^2\mu_2}}
\begin{pmatrix} \sqrt{\mu_1} \\ \rho\sqrt{\mu_2}
\end{pmatrix}+v+o_\mu(1/\mu)\]
where $v=O_\mu(1/\mu)$. Then $w^\top u=w^\top v+o_\mu(1/\mu^{3/2})$.
The condition $T_{21}u_1+T_{22}u_2=0$ requires
\begin{align*}
-\rho\sqrt{\frac{\mu_2}{\mu_1}}\,v_1+v_2&=
\frac{1}{\sqrt{\mu_1+\rho^2\mu_2}}
\left(-c_2\rho^3\sqrt{\mu_2^3/\mu_1^4}
+c_2\rho\sqrt{\mu_2^3/\mu_1^4}\right)+o_\mu(1/\mu)\\
&=\frac{c_2\rho(1-\rho^2)}{\mu_1^2}\sqrt{\frac{\mu_2^3}{\mu_1+\rho^2\mu_2}}
+o_\mu(1/\mu).
\end{align*}
So
\[w^\top u=w^\top v+o_\mu(1/\mu^{3/2})
=\frac{c_2\rho}{\mu_1^2}\sqrt{\frac{(1-\rho^2)\mu_2^2}{\mu_1+\rho^2\mu_2}}
+o_\mu(1/\mu^{3/2}).\]
Multiplying by (\ref{eq:alphaexpansion}) yields (\ref{eq:eigenvecexpansion2}).

\subsection{Simulation details}\label{appendix:simulations}
We provide additional details for the simulations
in Sections \ref{subsec:qualitative} and \ref{subsec:isotropic}:
We consider the
special case of (\ref{eq:mixedmodel}) corresponding to a balanced one-way layout
design,
\begin{equation}\label{eq:oneway}
Y=U_1\alpha_1+\alpha_2,
\end{equation}
where $n=2n_1=n_2=1600$, $p=3200$, and $U_2=\Id$. Each sample
$i=1,\ldots,n$ belongs to one of $n_1$ disjoint pairs, and
each column of $U_1 \in \R^{n \times n_1}$ has a
block of two 1's indicating the samples belonging to the corresponding pair.
The MANOVA estimate of $\Sigma_1$ is given by
\[\hSigma_1=Y^\sT B_1Y, \qquad B_1=\frac{1}{n}(\pi-\pi^\perp),\]
where $\pi$ and $\pi^\perp$ are the orthogonal projections onto the
column span of $U_1$ and its orthogonal complement.

We simulate $\alpha_1$ and $\alpha_2$ using the covariances
\begin{align*}
\Sigma_1&=32e_1e_1^\top+16e_2e_2^\top+8e_3e_3^\top+\oSigma_1,\\
\Sigma_2&=32ww^\top+64e_4e_4^\top+\oSigma_2,
\end{align*}
where $e_1,e_2,e_3,e_4$ are the first four standard basis vectors,
$w=(e_1+e_2+e_3)/\sqrt{3}$, and $\oSigma_1,\oSigma_2$ are diagonal matrices
whose first 4 diagonal entries are 0 and remaining entries are drawn randomly
from $\operatorname{Exponential}(1)$. We fix a single instance of
$\Sigma_1,\Sigma_2$ and generate all 1000 simulations of $\hSigma$
from this instance.

We compute $\det T(\lambda)$ over a fine grid of values $\lambda
\in \R$, by iteratively solving the fixed-point equations
(\ref{eq:arecursion}--\ref{eq:brecursion}).
For faster computation, the initial values for $b_r$
at each next point $\lambda+0.01$ are initialized by linear interpolation from
their values at $\lambda$ and $\lambda-0.01$.
Applying the Stieltjes inversion formula, we approximate the density
of $\mu_0$ at $\lambda$ by the value $\pi^{-1}\Im m_0(\lambda+10^{-8}i)$,
where $m_0$ is computed from (\ref{eq:m0}). We compute the
roots of the equation $0=\det T(\lambda)$ using grid search.

\subsection{Isotropic noise}\label{appendix:isotropic}
We verify that Theorems \ref{thm:outliers} and \ref{thm:eigenvectors}
agree with the earlier results of \cite{fanjohnstonesun} when restricted to the
setting of isotropic noise, where
$\oSigma_r=\sigma_r^2\Id$ for each
    $r=1,\ldots,k$.

Comparing (\ref{eq:m0}) with (\ref{eq:arecursion}), we have in this setting
    $a_r(z)=(p\sigma_r^2/n_r)m_0(z)$ for each $r$. Then $b_r(z)$ coincides with
    $-t_r(z)$ as defined in \cite[Eq.\ (3.2)]{fanjohnstonesun}. (Note that the
    matrix $F_{rs}$ in \cite{fanjohnstonesun} corresponds to
    $(p\sigma_r\sigma_s/\sqrt{n_rn_s})F_{rs}$ in the notation of this paper.)
    Applying $\det(\Id+XY)=\det(\Id+YX)$, our
    determinant equation $0=\det T(\lambda)$ is equivalent to
    \begin{align*}
&0=\det\Big(\Id+(\lambda \Id+\b \cdot \oSigma)^{-1} \Gamma^\sT \diag_\ell(\b)
    \Gamma\Big) \\
&\Leftrightarrow
    0=\det(\lambda \Id+\b \cdot \oSigma+\Gamma^\sT \diag_\ell(\b) \Gamma)
    =\det(\lambda \Id+\b \cdot \Sigma).
\end{align*}
    This is the same as the equation defining $\Lambda_0$
    in \cite[Eq.\ (3.4)]{fanjohnstonesun}.

    For the eigenvectors, note that $u \in \ker T(\lambda)$ corresponds to
    \[v=-M(\lambda)u \in \ker(\Id + \b \cdot \Sigma), \quad
    M(\lambda)=(\lambda \Id+\b \cdot \oSigma)^{-1}\Gamma^\sT \diag_\ell(\b).\]
    Then, in our notation, \cite[Theorem 3.3]{fanjohnstonesun} shows
    \[\Gamma\,\hv+\Big(u^\sT M(\lambda)^\sT \Big(\Id+\b'(\lambda) \cdot
    \oSigma+\Gamma^\sT \diag_\ell(\b')\Gamma\Big)
    M(\lambda)u\Big)^{-1/2}\Gamma M(\lambda)u \to 0.\]
    Since $u \in \ker T(\lambda)$, we have $\Gamma M(\lambda) u=-u$.
    Applying this and simplifying, we recover exactly
    Theorem \ref{thm:eigenvectors}.

\section{Augmented Cauchy and \texorpdfstring{$R$}{R}-transforms} \label{app:transforms}

We prove the identity between augmented Cauchy and $\cR$-transforms in
Lemma \ref{lem:mod-gr}.

\begin{proof}[Proof of Lemma \ref{lem:mod-gr}]
We apply the cumulant expansion to obtain
\begin{equation}\label{eq:cumulantexpansion}
G^{\B}_{a_1, a}(b)=\sum_{l \geq 0} \tau^{\B}(a_1 b^{-1} (a b^{-1})^l) =\sum_{l \geq 0} \sum_{\pi \in \NC(l + 1)} \kappa^{\B}_\pi(a_1 b^{-1}, a b^{-1}, \ldots, ab^{-1}).
\end{equation}
For a given non-crossing partition $\pi \in \NC(l+1)$, let $S \in \pi$ denote the element containing 1. Then the size $m$ of $S$ can range from 1 to $l+1$. Denote $S=\{j_0,j_1,\ldots,j_{m-1}\}$ where $j_0=1$. Set $c_i=j_i-j_{i-1}-1$ for $i=1,\ldots,m-1$ to be the number of elements between $j_{i-1}$ and $j_i$, and set $c_m=l+1-j_{m-1}$ as the number of elements after $j_{m-1}$. Then $c_1,\ldots,c_m$ sum to $l+1-m$, and the remaining elements of $\pi$ form non-crossing partitions of these intervals of sizes $c_1,\ldots,c_m$. Hence, applying the definition and multilinearity of $\kappa_\pi^\B$, we have
\begin{align*}
&\sum_{\pi \in \NC(l+1)} \kappa_\pi^\B(a_1,\ldots,a_{l+1})\\
&=\sum_{m=1}^{l+1} \mathop{\sum_{c_1,\ldots,c_m \geq 0}}_{\sum_i c_i=l+1-m} \sum_{\pi_1 \in \NC(c_1),\ldots,\pi_m \in \NC(c_m)}\\
&\hspace{0.5in}\kappa_m^\B(a_1 \kappa_{\pi_1}^\B(a_2,\ldots,a_{j_1-1}), a_{j_1} \kappa_{\pi_2}^\B(a_{j_1+1},\ldots,a_{j_2-1}), \ldots, a_{j_{m-1}}\kappa_{\pi_m}^\B(a_{j_{m-1}+1},\ldots,a_{l+1}))\\
&=\sum_{m=1}^{l+1}\mathop{\sum_{c_1,\ldots,c_m \geq 0}}_{\sum_i c_i=l+1-m} \kappa_m^\B(a_1\tau^\B(a_2\ldots a_{j_1-1}),a_{j_1} \tau^\B(a_{j_1+1}\ldots a_{j_2-1}),\ldots,a_{j_{m-1}} \tau^\B(a_{j_{m-1}+1}\ldots a_{l+1})).
\end{align*}
Applying this to (\ref{eq:cumulantexpansion}), exchanging orders of summations by
\[\sum_{l \geq 0}\;\;\sum_{m=1}^{l+1} \mathop{\sum_{c_1, \ldots, c_m \geq 0}}_{\sum_i c_i = l + 1 - m}
=\;\;\sum_{m \geq 1}\;\;\sum_{l \geq m-1} \mathop{\sum_{c_1, \ldots, c_m \geq 0}}_{\sum_i c_i = l + 1 - m}
=\;\;\sum_{m \geq 1}\sum_{c_1 \geq 0,\ldots,c_m \geq 0},\]
and then applying the definition of $\B$-valued Cauchy and $R$-transforms, we obtain
\begin{align*}
&G^\B_{a_1,a}(b)\\
&= \sum_{l \geq 0} \sum_{m = 1}^{l + 1} \sum_{\substack{c_1, \ldots, c_m \geq 0\\
\sum_i c_i = l + 1 - m}} \kappa^{\B}_m(a_1 b^{-1} \tau^{\B}((ab^{-1})^{c_1}), ab^{-1}\tau^\B((ab^{-1})^{c_2}),
\ldots, a b^{-1}\tau^{\B}((a b^{-1})^{c_m})) \\
&= \sum_{m \geq 1} \kappa^{\B}_m(a_1 G^{\B}_{a}(b), a G^{\B}_{a}(b), \ldots, a) G^{\B}_{a}(b)\\
&= R^{\B }_{a_1, a}(G^{\B}_{a}(b)) G^{\B}_{a}(b).
\end{align*}
For $\|b^{-1}\|$ sufficiently small, the preceding infinite series are all
    absolutely norm-convergent, and hence the preceding manipulations are
    valid as convergent series in $\B$.
\end{proof}

Recall that if $s,t \in \A$ are free with amalgamation over $\B$, we have a
subordination identity for the $\B$-valued Cauchy-transform
$G_{s+t}^\B$, given by
\begin{equation} \label{eq:gst-sub}
  G_{s+t}^\B(b)=G_t^\B(b-\cR_s^\B(G_{s+t}^\B(b))).
\end{equation}
This is a consequence of the additivity
$\cR_s^\B(b)+\cR_t^\B(b)=\cR_{s+t}^\B(b)$, and the moment-cumulant relation
(\ref{eq:GRrelation}). Lemma \ref{lem:mod-gr} yields the following analogous
subordination identity for the augmented transforms.

\begin{lemma}[Left subordination identity]\label{lem:leftsubordination}
Suppose $s,t,m \in \A$ are such that $s$ and $\{t,m\}$ are free with
    amalgamation over $\B$. Then for any invertible $b \in \B$ with 
    $\|b^{-1}\|$ sufficiently small,
\[G^{\B}_{m,s+t}(b)=G_{m,t}^\B(b-\cR^\B_s(G_{s+t}^\B(b))).\]
\end{lemma}
\begin{proof}
Denote $b'=G_{s+t}^\B(b) \in \B$. The usual subordination identity gives
    $b'=G_t^\B(b-\cR_s^\B(b'))$. Then
\begin{align*}
G_{m,s+t}^\B(b)&=\cR_{m,s+t}^\B(b')b'\\
&=\sum_{l \geq 1} \kappa_l^\B(m,b'(s+t),\ldots,b'(s+t))b'\\
&=\sum_{l \geq 1} \kappa_l^\B(m,b't,\ldots,b't)b'\\
&=\cR_{m,t}^\B(b')b'\\
&=G_{m,t}^\B(b-\cR^\B_s(b'))
\end{align*}
where the first and last equalities apply (\ref{eq:mod-gr}) with $a=s+t$
and $a=t$, the second and fourth equalities apply the definition of
$\cR_{a_1,a}^\B$, and the middle equality applies multi-linearity of
$\kappa_l$, $\B$-freeness of $s$ and $m$, and vanishing of mixed cumulants for free elements.
\end{proof}

\section{Strong asymptotic freeness}\label{app:strongfreeness}

In this section, we prove Theorem \ref{thm:strongfree}.
The proof follows an argument analogous to \cite{haagerupthorbjornsen,male},
which established such a result for GUE and GUE + deterministic
matrices, respectively. Several modifications to the argument are needed,
drawing on ideas in \cite{schultz,belinschicapitaine} which established this
type of result for GOE matrices and complex Wigner + deterministic matrices,
respectively. We note that the result of \cite{belinschicapitaine} requires the
real and imaginary parts of the complex Wigner matrices to have the same
variance, and does not directly apply to the GOE + deterministic matrix setting.

We provide here a brief outline of the proof and its relation to these previous
works:

\begin{enumerate}
    \item By the linearization trick of \cite[Section 2]{haagerupthorbjornsen},
        we first study linear polynomials $L$ with $k \times k$ Hermitian 
        matrix-valued coefficients, for arbitrary fixed dimension $k$. We aim 
        to show the spectral inclusion (\ref{eq:polynomial_inclusion}) for such 
        $L$, see Lemma \ref{inclusion_linear}.
    \item For this, it suffices to show that the difference between the Cauchy
        transform of $L(\mathbf{X}_N,\mathbf{Y}_N)$ and that of a deterministic
        measure $\mu_\A$ with the same spectrum as $L(\mathbf{x},\mathbf{Y}_N)$
        is at most $\operatorname{poly}((\Im \lambda)^{-1})/N^{1+\kappa}$,
        for some $\kappa>0$ and any spectral argument $\lambda \in \C^+$.
        For simplicity, we drop the $\lambda$-dependence here and denote
        this as $O(1/N^{1+\kappa})$.
        As in \cite{haagerupthorbjornsen,male}, we bound the expected
        difference by $O(1/N^2)$ and 
        the variance by $O(1/N^4)$. The latter bound uses the same Gaussian
        Poincar\'e argument as in these works.
    \item To bound the expected difference, we work with the expected
        $M_k$-valued Cauchy transform $G_{S_N+T_N}$ of
        $L(\mathbf{X}_N,\mathbf{Y}_N)$, and
        the $M_k$-valued Cauchy transform $G_{s+T_N}$ of
        $L(\mathbf{x},\mathbf{Y}_N)$. The latter satisfies the
        operator-valued subordination equation for the free additive convolution,
        \[G_{s+T_N}(\Lambda)=G_{T_N}(\Lambda-\cR_s(G_{s+T_N}(\Lambda))).\]
        Applying a similar Gaussian
        integration-by-parts argument as in \cite{male}, we show
        \[G_{S_N+T_N}(\Lambda)=G_{T_N}(\Lambda-\cR_s(G_{S_N+T_N}(\Lambda)))
        +O(1/N),\]
        see Lemma \ref{lemma:matrixsubordination}. In contrast to the GUE
        setting of \cite{male}, this is a \emph{first-order} remainder of size
        $O(1/N)$, not $O(1/N^2)$ as required. The $O(1/N)$ term vanishes for
        the GUE by a cancellation due to the real and imaginary parts having the
        same variance, but does not vanish for the GOE.
        A similar difficulty occurred also in \cite{schultz}.
    \item The bulk of the additional work in our argument lies in obtaining the
        second-order $O(1/N^2)$ approximation. In Proposition
        \ref{prop:secondordersubordination} below,
        applying the stability property of the subordination equation
        established in \cite[Proposition 4.3]{male} together with
        a Taylor expansion of $G_{T_N}$, we obtain
        \[\|G_{S_N+T_N}(\Lambda)-G_{s+T_N}(\Lambda)
        -\mathcal{L}_\Lambda(R_N(\Lambda))\| \leq O(1/N^2),\]
        where $\|\mathcal{L}_\Lambda(R_N(\Lambda))\| \leq O(1/N)$.
        We approximate the random quantity $R_N(\Lambda)$ by a deterministic
        approximation $R_\A(\Lambda)$, and show that $G_{s+T_N}(\Lambda)
        +\mathcal{L}_\Lambda(R_\A(\Lambda))$ is the Cauchy transform of a
        deterministic measure $\mu_\A$ as above. For the approximation
        $R_N(\Lambda) \approx R_\A(\Lambda)$, we follow an approach inspired
        by \cite{schultz}, and we identify the key term of
        $R_N(\Lambda)-R_\A(\Lambda)$ as the derivative of the difference of
        certain ``left-augmented'' $M_{2k}$-valued Cauchy transforms of
        $L(\mathbf{X}_N,\mathbf{Y}_N)$ and $L(\mathbf{x},\mathbf{Y}_N)$ in an
        expanded $2k \times 2k$ coefficient space, see
        (\ref{eq:aug-cauchy}) below. We bound this difference using
        a left-augmented subordination identity for $R_\A(\Lambda)$,
        an approximate such identity for $R_N(\Lambda)$, and a second
        application of \cite[Proposition 4.3]{male}.

    \item Finally, having established (\ref{eq:polynomial_inclusion}) for all
        such linear polynomials $L$, we may directly establish
        (\ref{eq:polynomial_inclusion}) for all $Q$ by applying the
        linearization and ultraproduct argument of \cite[Section
        7]{haagerupthorbjornsen} in a subsequential form. This concludes the
        proof of Theorem \ref{thm:strongfree}.
\end{enumerate}

In the remainder of this section, we carry out these steps to prove Theorem
\ref{thm:strongfree}. Its corollary
Theorem \ref{thm:strongfreeusual} then follows from results in \cite{male},
and is discussed at the end of the section. 

\subsection{Linearization and first-order approximation}
Replacing $Y_j$ and $Y_j^*$
by $(Y_j+Y_j^*)/2$ and $(Y_j-Y_j^*)/(2i)$, we will assume without loss
of generality that $Y_1,\ldots,Y_q$ are Hermitian.
We write as shorthand $M_N=\C^{N \times N}$, and denote by
$\tr_N=N^{-1}\Tr$ the normalized matrix trace on $M_N$. 

We first consider linear polynomials with matrix-valued coefficients.
Fix any $k \geq 1$ and Hermitian matrices $a_{0},...,a_{p},b_{1},...,b_{q} \in
M_{k}$. Set
    \begin{equation}\label{eq:LN}
        L_{N}=a_{0}\otimes \Id_{N}+S_{N}+T_{N}, \qquad
    S_{N}=\sum_{j=1}^{p}a_{j}\otimes X_{j}, \qquad
    T_{N}=\sum_{j=1}^{q}b_{j}\otimes Y_{j}.
    \end{equation}
Define correspondingly
\begin{equation}\label{eq:LA}
    L_\A=a_{0}\otimes \Id_{N}+s+T_{N}, \qquad
s=\sum_{j=1}^{p}a_{j}\otimes x_{j}.
\end{equation}
These belong to von Neumann probability spaces
$(M_k \otimes M_N,\tr_k \otimes \tr_N)$ and $(M_k \otimes \A_N,\tr_k \otimes
\tau_N)$. We denote by $\Id_k$ the identity in $M_k$, and by $\Id_N$ both the identity in $M_N$ and
the unit in $\A_N$. The space $M_k$ is identified as a subalgebra of both $M_k
\otimes M_N$ and $M_k \otimes \A_N$ via the inclusion map $x \mapsto x \otimes
\Id_N$, with the partial traces $\Id_k \otimes
\tr_N$ and $\Id_k \otimes \tau_N$ being the conditional expectations onto this
subalgebra.
Throughout, we let $C,C_1,C_2,\ldots$ be arbitrary constants depending on $k,p,q,\|a_0\|,\ldots,\|a_p\|,\|b_1\|,\ldots,\|b_q\|$.

For any element $x$ of a von Neumann algebra $\A$, define the self-adjoint
element
\[\Im x=\frac{x-x^*}{2i} \in \A.\]
We will use repeatedly the fact that for any self-adjoint element $y \in \A$,
\[\|(x+y)^{-1}\| \leq \|(\Im (x+y))^{-1}\|=\|(\Im x)^{-1}\|,\]
see \cite[Lemma 3.1]{haagerupthorbjornsen}. Let
\[M_{k}^{+}=\{X\in M_{k}:\Im X \succ 0\},
\quad M_k^-=\{X \in M_k:\Im X \prec 0\}\]
where $\succ$ and $\prec$ denote the positive-definite partial ordering for Hermitian matrices.

For $\Lambda,\Gamma \in M_k^+$, define the resolvents
    \begin{align*}
h_{S_{N}+T_{N}}(\Lambda)&=\left(\Lambda\otimes
        \Id_{N}-S_{N}-T_{N}\right)^{-1},\\
        g_{S_{N}+T_{N}}(\Lambda)&=\E[h_{S_N+T_N}(\Lambda)],\\
    g_{T_{N}}(\Gamma)&=\left(\Gamma\otimes \Id_{N}-T_{N}\right)^{-1}.
\end{align*}
    Define the $M_k$-valued Cauchy transforms
    \begin{align*}
H_{S_{N}+T_{N}}(\Lambda)&=(\Id_{k}
        \otimes\tr_{N})[h_{S_{N}+T_{N}}(\Lambda)],\\
    G_{S_{N}+T_{N}}(\Lambda)&=\E[H_{S_N+T_N}(\Lambda)],\\
    G_{T_{N}}(\Gamma)&=(\Id_{k}\otimes\tr_{N})[g_{ T_{N}}(\Gamma)].
\end{align*}
    We will eventually apply these with $\Lambda=\lambda \Id_k-a_0$
    to obtain the scalar-valued Cauchy transform of $L_N$.
Since $S_N$ and $T_N$ are Hermitian, we have the operator-norm bounds
\[\|H_{S_N+T_N}(\Lambda)\| \leq \|h_{S_N+T_N}(\Lambda)\|
\leq \|(\Im \Lambda)^{-1}\|,\]
and similarly for $G_{T_N}$ and $g_{T_N}$. One may verify that $H_{S_N+T_N}$, $G_{S_N+T_N}$, and $G_{T_N}$ are analytic maps from $M_k^+$ to $M_k^-$.

Correspondingly,
define the resolvent and $M_k$-valued Stieltjes transform of $s+T_N$, for
$\Lambda \in M_k^+$, by
\begin{align*}
g_{s+T_N}(\Lambda)&=(\Lambda \otimes \Id_N-s-T_N)^{-1},\\
G_{s+T_N}(\Lambda)&=(\Id_k \otimes \tau_N)[g_{s+T_N}(\Lambda)].
\end{align*}
Then by (\ref{eq:gst-sub}), $G_{s+T_N}$ satisfies the subordination identity
\begin{equation}\label{eq:subordination}
G_{s+T_N}(\Lambda)=G_{T_N}(\Lambda-\cR_s(G_{s+T_N}(\Lambda))),
\end{equation}
where
\begin{equation} \label{eq:rs-def}
  \cR_s(x)=\sum_{j=1}^p a_jxa_j
  \end{equation}
is the $M_k$-valued $\cR$-transform of $s$, see \cite[Proposition 4.2]{male}.
This identity holds for all $\Lambda \in M_k^+$, as both sides are analytic
over $M_k^+$. Since the $a_j$'s are Hermitian, $x \in M_k^+$ implies
$\Im \cR_s(x) \succeq 0$, and $x \in M_k^-$ implies $\Im \cR_s(x) \preceq 0$.

The subordination property (\ref{eq:subordination}) arises from freeness of $s$
and $T_N$ over $M_k$. In this subsection, we establish the following matrix analogue of this identity, which arises from the asymptotic freeness of $S_N$ and $T_N$.

\begin{lemma}[Matrix subordination identity]\label{lemma:matrixsubordination}
Fix any $\Lambda \in M_k^+$, and set
    $\Gamma=\Lambda-\cR_s(G_{S_N+T_N}(\Lambda))$. Then
\begin{equation}\label{eq:matrixsubordination}
G_{S_N+T_N}(\Lambda)=G_{T_N}(\Gamma)+R_N(\Lambda,\Gamma,\Id_k \otimes \Id_N)
+\Theta_N(\Lambda,\Gamma,\Id_k \otimes \Id_N)
\end{equation}
where
    \[\|R_N(\Lambda,\Gamma,\Id_k \otimes \Id_N)\| \leq \frac{C}{N}\|(\Im
    \Lambda)^{-1}\|^3,\qquad
\|\Theta_N(\Lambda,\Gamma,\Id_k \otimes \Id_N)\| \leq \frac{C}{N^2}
    \|(\Im \Lambda)^{-1}\|^5.\]
\end{lemma}

Comparing with (\ref{eq:subordination}), there is a ``first-order'' remainder
term $R_N$ and ``second-order'' remainder term $\Theta_N$, whose exact
forms are below. We will further approximate $R_N$ in the next subsection.

We show Lemma \ref{lemma:matrixsubordination} by specializing the following
proposition to $M=\Id_k \otimes \Id_N$.

\begin{proposition}\label{master1}
For any deterministic
$\Lambda,\Gamma\in M_{k}^{+}$ and $M \in M_k \otimes M_N$, we have
\begin{equation}
\Id_k \otimes \tr_N \Big(\E[Mh_{S_N+T_N}(\Lambda)]-Mg_{T_N}(\Gamma)\Big)
=R_N(\Lambda,\Gamma,M)+\Theta_N(\Lambda,\Gamma,M)
\end{equation}
where
\begin{align*}
\Theta_N(\Lambda,\Gamma,M)&=\E\left[\Id_k\otimes\tr_{N}\left[Mg_{T_{N} }(\Gamma)\Big(\big(\mathcal{R}_{s}(H_{S_{N}+T_{N}}(\Lambda))-\Lambda+\Gamma\big)\otimes \Id_{N}\Big)h_{S_{N}+T_{N}}(\Lambda)\right]\right],\\
    R_{N}(\Lambda,\Gamma,M)&=-\frac{1}{N}\sum_{j=1}^{p}\sum_{s,l=1}^{k}\mathbb{E}\Big[\Id_{k}\otimes
    \tr_{N}\Big[Mg_{T_{N}}(\Gamma)(a_{j} e_{sl}^{(k)} \otimes
    \Id_{N}) \\
&\hspace{2in}h_{S_{N}+T_{N}}^{\sT}(\Lambda)(e_{sl}^{(k)}a_{j}\otimes
    \Id_{N})h_{S_{N}+T_{N}}(\Lambda)\Big]\Big].
\end{align*}
    Here, $e_{sl}^{(k)} \in M_k$ is the matrix with $(s,l)$ coordinate equal to
    1 and remaining coordinates 0, and
    $h_{S_N+T_N}^{\sT}(\Lambda)=(\Lambda^\sT \otimes \Id_N-S_N^\sT-T_N^\sT)^{-1}$ is the (non-conjugated) matrix transpose, where
\[S_N^\sT=\sum_{j=1}^p a_j^\sT \otimes X_j, \qquad
T_N^\sT=\sum_{j=1}^q b_j^\sT \otimes Y_j^\sT.\]
 \end{proposition}
\begin{proof}
    The argument follows \cite[Proposition 5.2]{male}, with
    modifications similar to \cite[Theorem 2.1]{schultz} which produce the extra term $R_N$ in the setting of the GOE.

We represent $X_{j}=\frac{1}{\sqrt{2}}(Z_{j}+Z_{j}^\sT)$ where $Z_j \in \R^{N \times N}$ has i.i.d.\ $\N(0,1/N)$ entries. The Gaussian integration-by-parts identity $\E[\xi f(\xi)]=N^{-1}\E[f'(\xi)]$ for $\xi \sim \N(0,1/N)$ gives
\[\mathbb{E}\left[(Z_{j})_{sl}h_{S_{N}+T_{N}}(\Lambda)\right]=
\frac{1}{N\sqrt{2}}\mathbb{E}\left[\frac{d}{dt}\Big|_{t=0}(\Lambda\otimes \Id_{N}-S_{N}-T_{N}-t(a_{j}\otimes e_{sl}^{(N)}+a_{j}\otimes e_{ls}^{(N)}))^{-1}\right],\]
where $e_{sl}^{(N)} \in \R^{N \times N}$ is the matrix with the single entry $(s,l)$ equal to 1. Applying
\begin{equation}\label{eq:matrixderivative}
\frac{d}{dt}A(t)^{-1}=-A(t)^{-1}(t)A'(t)A^{-1}(t),
\end{equation}
we get
\[\mathbb{E}\left[(Z_{j})_{sl}h_{S_{N}+T_{N}}(\Lambda)\right]
=\frac{1}{N\sqrt{2}}
\mathbb{E}\left[h_{S_{N}+T_{N}}(\Lambda)(a_{j}\otimes e_{sl}^{(N)}+a_{j}\otimes e_{ls}^{(N)})h_{S_{N}+T_{N}}(\Lambda)\right].\] Then writing $Z_{j}=\sum_{s,l=1}^{N}(Z_{j})_{sl}e_{sl}^{(N)}$, 
\begin{align}
    &\mathbb{E}\left[\left(\frac{a_{j}}{\sqrt{2}}\otimes
Z_{j}\right)h_{S_{N}+T_{N}}(\Lambda)\right]\nonumber\\
&\hspace{0.5in}=\frac{1}{2N}\sum_{s,l=1}^{N}(a_{j}\otimes e_{sl}^{(N)})\mathbb{E}\left[h_{S_{N}+T_{N}}(\Lambda)(a_{j}\otimes e_{sl}^{(N)}+a_{j}\otimes e_{ls}^{(N)})h_{S_{N}+T_{N}}(\Lambda)\right].\label{e4}
\end{align}

For any $a,b \in M_k$ and any elementary tensor $x \otimes Y \in M_k \otimes M_N$,
\[\sum_{s,l=1}^{N}(a \otimes e_{sl}^{(N)})(x \otimes Y)(b\otimes e_{ls}^{(N)})=N\tr_N(Y) \cdot axb \otimes \Id_{N},\]
and
\begin{align*}
&\sum_{s,l=1}^{N}(\Id_{k}\otimes e_{sl}^{(N)})(x\otimes Y)(\Id_{k}\otimes
e_{sl}^{(N)})=x\otimes Y^{\sT}\\
&\hspace{1in}=(x^{\sT}\otimes Y)^{\sT}=\sum_{s,l=1}^{k}(e^{(k)}_{sl}\otimes \Id_{N})(x\otimes Y)^{\sT}(e^{(k)}_{sl}\otimes \Id_{N}).
\end{align*}
Then by linearity, for any $M \in M_k \otimes M_N$,
\[\sum_{s,l=1}^{N}(a \otimes e_{sl}^{(N)})M(b\otimes e_{ls}^{(N)})=N \cdot \Big(a((\Id_k \otimes \tr_N)M)b\Big) \otimes \Id_{N},\]
and
\[\sum_{s,l=1}^{N}(\Id_{k}\otimes e_{sl}^{(N)})M(\Id_{k}\otimes e_{sl}^{(N)})=\sum_{s,l=1}^{k}(e^{(k)}_{sl}\otimes \Id_{N})M^\sT(e^{(k)}_{sl}\otimes \Id_{N}).\]
So the right side of (\ref{e4}) is
\begin{align*}
&\frac{1}{2}\mathbb{E}\left[(a_{j}H_{S_{N}+T_{N}}(\Lambda)a_{j}\otimes
\Id_{N})h_{S_{N}+T_{N}}(\Lambda)\right]\\
&+\frac{1}{2N}\sum_{s,l=1}^{k} \mathbb{E}\left[(a_{j} e_{sl}^{(k)} \otimes \Id_{N}) h_{S_{N}+T_{N}}^{\sT}(\Lambda)(e_{sl}^{(k)}a_{j}\otimes \Id_{N})h_{S_{N}+T_{N}}(\Lambda)\right].
\end{align*}

Applying the same identity as (\ref{e4}) for $Z_j^\sT$, summing over $j$, recalling $S_N=\sum_j a_j \otimes (Z_j+Z_j^\sT)/\sqrt{2}$, multiplying on the left by $Mg_{T_N}(\Gamma)$, and recalling the definition of $\cR_s$ from (\ref{eq:rs-def}) we get
\begin{align*}
&\mathbb{E}\left[Mg_{T_{N}}(\Gamma)S_{N}h_{S_{N}+T_{N}}(\Lambda)\right]\\
        &=\mathbb{E}\left[Mg_{T_{N}}(\Gamma)(\cR_s(H_{S_{N}+T_{N}}(\Lambda))\otimes \Id_{N})h_{S_{N}+T_{N}}(\Lambda)\right]\\
        &\hspace{0.1in}+\frac{1}{ N}\sum_{j=1}^{p}\sum_{s,l=1}^{k}\mathbb{E}\left[Mg_{T_{N}}(\Gamma)(a_{j} e_{sl}^{(k)} \otimes \Id_{N})h_{S_{N}+T_{N}}^{\sT}(\Lambda)(e_{sl}^{(k)}a_{j}\otimes \Id_{N})h_{S_{N}+T_{N}}(\Lambda)\right].
\end{align*}
Writing $S_{N}=(\Lambda-\Gamma)\otimes \Id_{N}+(\Gamma\otimes \Id_{N}-T_{N})-(\Lambda\otimes \Id_{N}-S_{N}-T_{N})$, rearranging,
and applying the partial trace $\Id_k \otimes \tr_N$ to both sides yields the result.
\end{proof}

\begin{remark*}
Proposition \ref{master1} shows the difference between GOE and GUE matrices.
Applying integration by parts for the $N^2$ independent Gaussian random
variables in the GUE setting, we would obtain $N^2$ terms on the right of
(\ref{e4}), see \cite[eqs.\ (3.7--3.9)]{haagerupthorbjornsen}. However, in
the GOE setting, there are $2N^2$ terms in (\ref{e4}), and the terms in
(\ref{e4}) which do not appear in the GUE case lead to the first order
remainder $R_N$.
\end{remark*}

\begin{proposition}\label{prop:Rbound}
For any $\Lambda,\Gamma \in M_k^+$ and $M \in M_k \otimes M_N$,
\begin{equation}
    \|R_{N}(\Lambda,\Gamma,M)\|\leq \frac{C}{N}\|M\|\|(\Im\Lambda)^{-1}\|^{2}\|(\Im\Gamma)^{-1}\|.
\end{equation}
\end{proposition}
\begin{proof}
This follows from the definition of $R_N$, and the bounds $\|g_{T_N}(\Gamma)\| \leq \|(\Im \Gamma)^{-1}\|$ and $\|h_{S_N+T_N}(\Lambda)\| \leq \|(\Im \Lambda)^{-1}\|$.
\end{proof}

\begin{proposition}\label{prop:Thetabound}
For any $\Lambda \in M_k^+$, $M \in M_k \otimes M_N$, and for $\Gamma=\Lambda-\cR_s(G_{S_N+T_N}(\Lambda))$,
\[\|\Theta_N(\Lambda,\Gamma,M)\|\leq \frac{C}{N^2}\|M\|\|(\Im \Lambda)^{-1}\|^5.\]
\end{proposition}
\begin{proof}
    The proof is similar to that of \cite[Proposition 5.3]{male}, and we
    will omit some details. Introduce
\[K_{S_{N}+T_{N}}(\Lambda)=H_{S_{N}+T_{N}}(\Lambda)-G_{S_{N}+T_{N}}(\Lambda)
=H_{S_{N}+T_{N}}(\Lambda)-\E[H_{S_{N}+T_{N}}(\Lambda)].\]
Then, as $\cR_s$ is a linear map, for the given value of $\Gamma$
\[\Theta_N(\Lambda,\Gamma,M)
=\E\Big[\Id_k \otimes \tr_N\Big[
Mg_{T_N}(\Gamma)(\cR_s(K_{S_N+T_N}(\Lambda)) \otimes \Id_N)
h_{S_N+T_N}(\Lambda)\Big]\Big].\]
Further introduce
\[k_{S_N+T_N}(\Lambda)=h_{S_N+T_N}(\Lambda)-g_{S_N+T_N}(\Lambda)
=h_{S_N+T_N}(\Lambda)-\E[h_{S_N+T_N}(\Lambda)].\]
Then, applying $\E[K_{S_N+T_N}(\Lambda)]=0$, the above implies
\[\Theta_N(\Lambda,\Gamma,M)
=\E\Big[\Id_k \otimes \tr_N\Big[
Mg_{T_N}(\Gamma)(\cR_s(K_{S_N+T_N}(\Lambda)) \otimes \Id_N)
k_{S_N+T_N}(\Lambda)\Big]\Big].\]

Denote
\[\|M\|_\infty=\max_{i,j} |M_{ij}|, \quad \|M\|_\HS^2=\sum_{i,j} |M_{ij}|^2.\]
For $X \in M_k \otimes M_N$ and $e_s^{(N)}$ the $s^\text{th}$ standard basis vector in $\C^N$, define
\[(\Id_k \otimes e_s^{(N)})^\sT X (\Id_k \otimes e_l^{(N)})
=X^{(s,l)} \in M_k\]
and
\[(e_s^{(k)} \otimes \Id_N)^\sT X (e_l^{(k)} \otimes \Id_N)=X_{(s,l)} \in M_N.\]
Note in particular that
\[X=\sum_{s,l=1}^N X^{(s,l)} \otimes e_{sl}^{(N)}.\]
Applying this decomposition to $Mg_{T_N}(\Gamma)$ and to $h_{S_N+T_N}(\Lambda)$, we bound
    \begin{align*}
\|\Theta(\Lambda,\Gamma,M)\| &\leq \sqrt{k}\|\Theta(\Lambda,\Gamma,M)\|_\infty\\
&=\frac{\sqrt{k}}{N}\left\|\sum_{s,l=1}^N
    \E\Big[(Mg_{T_N}(\Gamma))^{(s,l)}\cR_s(K_{S_N+T_N}(\Lambda))(k_{S_N+T_N}(\Lambda))^{(l,s)}\Big]\right\|_\infty.
\end{align*}
Then applying $|\sum_{i=1}^m y_i| \leq m \max_i |y_i|$, we obtain
\begin{align*}
&\|\Theta(\Lambda,\Gamma,M)\|\\
&\leq \frac{k^{5/2}}{N}
\max_{u,v,u',v' \in \{1,\ldots,k\}}
\left|\E\left[\cR_s(K_{S_N+T_N}(\Lambda))_{u',v'}
\sum_{s,l=1}^N (Mg_{T_N}(\Gamma))^{(s,l)}_{u,u'} \cdot
(k_{S_N+T_N}(\Lambda))_{v',v}^{(l,s)}\right]\right|\\
&\leq k^{5/2}\max_{u,v,u',v' \in \{1,\ldots,k\}}
\E\left[|\cR_s(K_{S_N+T_N}(\Lambda))_{u',v'}| \cdot \left|\tr_N
\Big(Mg_{T_N}(\Gamma)_{(u,u')}k_{S_N+T_N}(\Lambda)_{(v',v)}\Big)\right|
\right]\\
&\leq k^{5/2}\max_{u,v,u',v' \in \{1,\ldots,k\}}
\Var[\cR_s(H_{S_N+T_N}(\Lambda))_{u',v'}]^{1/2}\\
&\hspace{2in}\cdot \Var\left[\tr_N
\Big(Mg_{T_N}(\Gamma)_{(u,u')}h_{S_N+T_N}(\Lambda)_{(v',v)}\Big)\right]^{1/2},
\end{align*}
where the last line applies Cauchy-Schwarz and $\Var$ denotes the complex variance.

Fix any $u,v,u',v'$, and define the scalar-valued functions
\begin{align*}
F_1(S_N)&=\cR_s(H_{S_N+T_N}(\Lambda))_{u',v'},\\
F_2(S_N)&=\tr_N \left(Mg_{T_N}(\Gamma)_{(u,u')}h_{S_N+T_N}(\Lambda)_{(v',v)}\right).
\end{align*}
    Following the same arguments as in \cite[Proposition 5.3]{male},
and applying $\|(\Im \Gamma)^{-1}\| \leq \|(\Im \Lambda)^{-1}\|$ because $\Lambda \in M_k^+$ and $\Im \cR_s(G_{S_N+T_N}(\Lambda)) \preceq 0$, we may verify that
    \[\|\nabla F_1(S_N)\|^2 \leq \frac{C}{N}\|(\Im \Lambda)^{-1}\|^4, \qquad
    \|\nabla F_2(S_N)\|^2 \leq \frac{C}{N}\|M\|^2\|(\Im \Lambda)^{-1}\|^6.\]
Then, as the entries of $S_N$ are $C/\sqrt{N}$-Lipschitz in the independent standard Gaussian variables which define $X_1,\ldots,X_p$, the Gaussian Poincar\'e inequality yields
\[\Var[F_1(S_N)] \leq \frac{C}{N^2}\|(\Im \Lambda)^{-1}\|^4,
\quad \Var[F_2(S_N)] \leq \frac{C}{N^2}\|M\|^2\|(\Im \Lambda)^{-1}\|^6.\]
Substituting above concludes the proof.
\end{proof}

Combining Propositions \ref{master1}, \ref{prop:Rbound} and
\ref{prop:Thetabound}, and specializing to $M=\Id_k \otimes \Id_N$,
we obtain Lemma \ref{lemma:matrixsubordination}. The following is then a
consequence of the stability property for the subordination equation
(\ref{eq:subordination}), established in \cite{male}:
For a parameter $\eta>0$, define the simply connected open set
\begin{equation}\label{omega}
    \Omega_{\eta}^{(N)}=\left\{\Lambda\in M_{k}^{+}:\|(\Im\Lambda)^{-1}\|<N^{\eta}\right\},
\end{equation}
\begin{lemma}[First-order Cauchy transform approximation]\label{lemma:firstorderapprox}
Let $\eta<1/3$. Then there exists $N_{0}>0$ such that for all $N \geq N_{0}$ and $\Lambda \in \Omega_\eta^{(N)}$,
\begin{align*}
    \|G_{s+T_{N}}(\Lambda)-G_{S_{N}+T_{N}}(\Lambda)\|
    &\leq \frac{C}{N}\|(\Im \Lambda)^{-1}\|^3\left(1+\|(\Im \Lambda)^{-1}\|^2\right).
\end{align*}
\end{lemma}
\begin{proof}
For $\eta<1/3$, $N \geq N_{0}$, and $\Lambda\in\Omega_{\eta}^{(N)}$, Lemma
    \ref{lemma:matrixsubordination} implies
\[\|G_{S_N+T_N}(\Lambda)-G_{T_N}(\Lambda-\cR_s(G_{S_N+T_N}(\Lambda)))\|
\leq \frac{C}{N}\|(\Im \Lambda)^{-3}\| \leq 1/2.\]
The result then follows from \cite[Proposition 4.3]{male}.
\end{proof} 

\subsection{Second-order approximation}
For $\Lambda \in M_k^+$, denote
the first-order remainder in Lemma \ref{lemma:matrixsubordination} as
\[R_{N}(\Lambda)=R_{N}(\Lambda,\Gamma_N, \Id_k \otimes \Id_N), \qquad
\Gamma_N \equiv \Gamma_N(\Lambda)=\Lambda-\cR_s(G_{S_N+T_N}(\Lambda)).\]
Define the approximation to $\Gamma_N$, which appears in
(\ref{eq:subordination}), by
\[\Gamma_\A \equiv \Gamma_\A(\Lambda)=\Lambda-\mathcal{R}_{s}(G_{s+T_{N}}(\Lambda))).\]
Note that if $\Lambda \in M_k^+$, then $\Gamma_N,\Gamma_\A \in M_k^+$ also.
Then define an approximation to $R_N(\Lambda)$ by
\begin{equation}\label{eq:RA}
    R_\A(\Lambda)=-\frac{1}{N}\sum_{j=1}^{p}\sum_{m,l=1}^{k}(\Id_{k}\otimes \tau_N)\left(g_{T_{N}}(\Gamma_\A)(a_{j} e_{ml}^{(k)} \otimes \Id_{N})g_{s+T_{N}}^{\sT}(\Lambda)(e_{ml}^{(k)}a_{j}\otimes \Id_{N})g_{s+T_{N}}(\Lambda)\right).
\end{equation}
Here, $g_{s+T_N}^\sT=(\Lambda^\sT \otimes \Id_N-s^\sT-T_N^\sT)^{-1}$ where
\[s^\sT=\sum_{j=1}^p a_j^\sT \otimes x_j\]
and $T_N^\sT$ is as before. In this section, we extend Lemma
\ref{lemma:firstorderapprox} to the following second-order approximation.

\begin{lemma}[Second-order Cauchy-transform approximation]\label{GsT-GST+R/N}
For $\eta<1/3$, a constant $N_0>0$, all $N \geq N_0$, and
    $\Lambda \in \Omega_\eta^{(N)}$,
\[\|G_{s+T_{N}}(\Lambda)-G_{S_{N}+T_{N}}(\Lambda)+\mathcal{L}_\Lambda(R_\A(\Lambda))\|
    \leq \frac{C}{N^{2}} \|(\Im\Lambda)^{-1}\|^5(1+\|(\Im
    \Lambda)^{-1}\|^{10})\]
    where $\mathcal{L}_\Lambda:M_k \to M_k$ is the linear map
\[\mathcal{L}_\Lambda(x)=x-G'_{s+T_N}(\Lambda)[\cR_s(x)],\]
    and $G_{s+T_N}'(\Lambda)$ is the derivative of $G_{s+T_N}$.
\end{lemma}

The map $\mathcal{L}_\Lambda$ above appeared also in the analysis of
\cite[Theorem 5.7]{belinschicapitaine}.
The proof will reveal that $\|\mathcal{L}_\Lambda(R_\A(\Lambda))\|$ is of
size $O(1/N)$.

We first show that the above result holds with $R_N$ in place of $R_\A$. 

\begin{proposition}\label{prop:secondordersubordination}
For any fixed constant $\eta<1/3$, there exists $N_0>0$ such that for all $N \geq N_0$ and $\Lambda \in \Omega_\eta^{(N)}$,
\[\|G_{s+T_N}(\Lambda)-G_{S_N+T_N}(\Lambda)+\mathcal{L}_\Lambda(R_N(\Lambda))\|
\leq \frac{C}{N^2}\|(\Im \Lambda)^{-1}\|^5(1+\|(\Im \Lambda)^{-1}\|^{10}).\]
Furthermore, defining the operator norm
$\|\mathcal{L}_\Lambda\|=\sup_{x \in M_k:\|x\|=1} \|\mathcal{L}_\Lambda(x)\|$,
\[\|\mathcal{L}_\Lambda\| \leq C(1+\|(\Im \Lambda)^{-1}\|).\]
\end{proposition}
\begin{proof}
Let us write
\[\Delta_N(\Lambda)=G_{s+T_N}(\Lambda)-G_{S_N+T_N}(\Lambda).\]
Subtracting (\ref{eq:matrixsubordination}) from (\ref{eq:subordination}), we get
\begin{equation}\label{eq:deltaapprox1}
\|\Delta_N(\Lambda)-G_{T_N}(\Gamma_\A)+G_{T_N}(\Gamma_N)+R_N(\Lambda)\|
\leq \frac{C}{N^2}\|(\Im \Lambda)^{-1}\|^5.
\end{equation}
Lemma \ref{lemma:firstorderapprox} provides a bound for $\|\Delta_N(\Lambda)\|$, from which we obtain also
\begin{equation}\label{eq:Gammadiff}
\|\Gamma_{N}-\Gamma_{\A}\|=\|\cR_s(\Delta_N(\Lambda))\|
    \leq \frac{C}{N}\|(\Im \Lambda)^{-1}\|^3\left(1+\|(\Im \Lambda)^{-1}\|^2\right).
\end{equation}

We apply a Taylor expansion to approximate $G_{T_N}(\Gamma_N)-G_{T_N}(\Gamma_\A)$: Fix $v,w \in \C^k$ with $\|v\|=\|w\|=1$ and define
\[\Gamma_t=(1-t)\Gamma_\A+t\Gamma_N, \qquad f(t)=v^*G_{T_N}(\Gamma_t)w.\]
Then
\begin{align*}
f'(t)&=v^*\Big[(\Id \otimes \tr_N)\Big(g_{T_N}(\Gamma_t)((\Gamma_\A-\Gamma_N) \otimes \Id_N)g_{T_N}(\Gamma_t)\Big)\Big]w,\\
f''(t)&=2v^*\Big[(\Id \otimes \tr_N)\Big(g_{T_N}(\Gamma_t)
((\Gamma_\A-\Gamma_N) \otimes \Id_N)g_{T_N}(\Gamma_t)
((\Gamma_\A-\Gamma_N) \otimes \Id_N)g_{T_N}(\Gamma_t)\Big)\Big]w.
\end{align*}
In particular, for all $t \in [0,1]$, by Lemma \ref{lemma:firstorderapprox} and the bounds $\|g_{T_N}(\Gamma_t)\| \leq C \|(\Im \Lambda)^{-1}\|$, we find
\[|f''(t)| \leq C\|(\Im \Lambda)^{-1}\|^3 \|\Gamma_\A-\Gamma_N\|^2
\leq \frac{C}{N^2}\|(\Im \Lambda)^{-1}\|^9(1+\|(\Im \Lambda)^{-1}\|^4).\]
So
\begin{align*}
\Big|v^*\Big(G_{T_N}(\Gamma_N)-G_{T_N}(\Gamma_\A)-G_{T_N}'(\Gamma_\A)[\Gamma_N-\Gamma_\A]\Big)w\Big|&=|f(1)-f(0)-f'(0)|\\
&\leq \frac{C}{N^2}\|(\Im \Lambda)^{-1}\|^9(1+\|(\Im \Lambda)^{-1}\|^4).
\end{align*}
Applying this and $\Gamma_N-\Gamma_\A=\cR_s(\Delta_N(\Lambda))$ to (\ref{eq:deltaapprox1}), we obtain
\begin{equation} \label{eq:delta-gr-bound}
  \|\Delta_N(\Lambda)+G_{T_N}'(\Gamma_\A)[\cR_s(\Delta_N(\Lambda))]+R_N(\Lambda)\|
  \leq \frac{C}{N^2}\|(\Im \Lambda)^{-1}\|^5(1+\|(\Im \Lambda)^{-1}\|^8).
  \end{equation}

We now claim that the linear map
\[F_\Lambda(x)=x+G_{T_N}'(\Gamma_\A(\Lambda))[\cR_s(x)]\]
is invertible, with inverse given by $\mathcal{L}_\Lambda$.
Indeed, differentiating the subordination identity (\ref{eq:subordination}) in $\Lambda$, for any $\Lambda \in M_k^+$ and $x \in M_k$,
\[G_{s+T_N}'(\Lambda)[x]=G_{T_N}'(\Gamma_\A(\Lambda))\Big[x-\cR_s(G_{s+T_N}'(\Lambda)[x])\Big].\]
Then for any $z \in M_k$, setting $x=\cR_s(z)$ and $y=z-G_{s+T_N}'(\Lambda)[x]=\mathcal{L}_\Lambda(z)$,
we obtain
\[z-y=G_{T_N}'(\Gamma_\A)[\cR_s(y)].\]
Hence $z=F_\Lambda(y)$, so $F_\Lambda$ is onto and invertible,
with inverse $\mathcal{L}_\Lambda$. Then noting that,
\[
F_\Lambda(\Delta_N(\Lambda))=\Delta_N(\Lambda)+G_{T_N}'(\Gamma_\A(\Lambda))[\cR_s(\Delta_N(\Lambda))],
\]
we have by (\ref{eq:delta-gr-bound}) that
\begin{align*}
\|\Delta_N(\Lambda)+\mathcal{L}_\Lambda(R_N(\Lambda))\|
&\leq \|\mathcal{L}_\Lambda\| \cdot \|F_\Lambda(\Delta_N(\Lambda))
+R_N(\Lambda)\|\\
&\leq \|\mathcal{L}_\Lambda\| \cdot \frac{C}{N^2} \|(\Im \Lambda)^{-1}\|^5(1+\|(\Im \Lambda)^{-1}\|^8).
\end{align*}
Finally, writing
\begin{align*}
\mathcal{L}_\Lambda(z)&=z-(\Id \otimes \tau_N)\Big[
\frac{d}{dt}\Big|_{t=0} g_{s+T_N}(\Lambda+t\cR_s(z))\Big]\\
&=z+(\Id \otimes \tau_N)\Big[g_{s+T_N}(\Lambda)(\cR_s(z) \otimes \Id_N)g_{s+T_N}(\Lambda)\Big],
\end{align*}
we verify $\|\mathcal{L}_\Lambda\| \leq C(1+\|(\Im \Lambda)^{-1}\|^2)$, and hence also the desired bound.
\end{proof}

To complete the proof of Lemma \ref{GsT-GST+R/N}, we will show that
\begin{equation}\label{eq:Rdiffbound}
\|R_\A(\Lambda)-R_{N}(\Lambda)\|\leq
\frac{C}{N^2}\|(\Im \Lambda)^{-1}\|^5(1+\|(\Im \Lambda)^{-1}\|^6).
\end{equation}
Let us write
\[R_N(\Lambda)-R_\A(\Lambda)=\frac{1}{N}\sum_{j=1}^p \sum_{m,l=1}^k
\left(A_1(j,m,l)+A_2(j,m,l)\right),\]
where
\begin{align*}
A_{1}(j,m,l)&=\mathbb{E}\Big[\Id_{k}\otimes
\tr_{N}\Big((g_{T_{N}}(\Gamma_\A)-g_{T_N}(\Gamma_N))(a_{j} e_{ml}^{(k)} \otimes
\Id_{N})\\
&\hspace{2in}h_{S_{N}+T_{N}}^{\sT}(\Lambda)(e_{ml}^{(k)}a_{j}\otimes
\Id_{N})h_{S_{N}+T_{N}}(\Lambda)\Big)\Big]
\end{align*}
and
\begin{align}
    A_{2}(j,m,l)&=\Id_{k}\otimes \tau_N\left(g_{T_{N}}(\Gamma_\A)(a_{j}
e_{ml}^{(k)} \otimes \Id_{N})g_{s+T_{N}}^{\sT}(\Lambda)(e_{ml}^{(k)}a_{j}\otimes
    \Id_{N})g_{s+T_{N}}(\Lambda)\right)\nonumber\\
    &\hspace{0.2in}-\mathbb{E}\left[(\Id_{k}\otimes
\tr_{N})\left(g_{T_{N}}(\Gamma_\A)(a_{j} e_{ml}^{(k)} \otimes
    \Id_{N})h_{S_N+T_{N}}^{\sT}(\Lambda)(e_{ml}^{(k)}a_{j}\otimes
    \Id_{N})h_{S_N+T_{N}}(\Lambda)\right)\right].\label{eq:A2}
\end{align}
We bound separately $A_1$ and $A_2$.

\begin{proposition}\label{prop:A1bound}
Let $\eta<1/3$. Then for a constant $N_0>0$, all $N \geq N_0$, all $\Lambda\in \Omega_{\eta}^{(N)}$, and all $j \in \{1,\ldots,p\}$ and $m,l \in \{1,\ldots,k\}$,
\[\|A_1(j,m,l)\| \leq \frac{C}{N}\|(\Im \Lambda)^{-1}\|^7
(1+\|(\Im \Lambda)\|^2).\]
\end{proposition}
\begin{proof}
This follows from (\ref{eq:Gammadiff}),
$\|h_{S_N+T_N}(\Lambda)\| \leq \|(\Im \Lambda)^{-1}\|$
and $g_{T_N}(\Gamma_*) \leq \|(\Im \Lambda)^{-1}\|$ for $\Gamma_* \in \{\Gamma_\A,\Gamma_N\}$, and the resolvent identity
\[g_{T_{N}}(\Gamma_\A)-g_{T_{N}}(\Gamma_{N})=g_{T_{N}}(\Gamma_\A)(\Gamma_N-\Gamma_\A)g_{T_{N}}(\Gamma_N). \qedhere\]
\end{proof}

To bound $A_2(j,m,l)$, denote by
\[\mathcal{Y}_N=\langle Y_1,\ldots,Y_q\rangle\]
the von Neumann subalgebra generated by $Y_1,\ldots,Y_q$, both as a subalgebra
of $M_N$ and of $\A_N$. For $M \in M_k \otimes \mathcal{Y}_N$, denote
\[G_{M,s+T_{N}}(\Lambda)=(\Id_k \otimes \tau_N)\Big(Mg_{s+T_N}(\Lambda)\Big),
\quad
G_{M,S_{N}+T_{N}}(\Lambda)=(\Id_k \otimes
\tr_N)\E\Big[Mh_{S_N+T_N}(\Lambda)\Big].\]
Note that these are ``left'' $M_k$-valued Cauchy transforms in the sense of
Lemma \ref{lem:leftsubordination}. We combine the left subordination identity of
that lemma with Proposition \ref{master1}, now applied with a general
matrix $M \in M_k \otimes \mathcal{Y}_N$ to obtain the following.

\begin{proposition}\label{prop:derQ}
    Let $\eta<1/3$. Then there exists $N_0>0$ such that for all $N \geq N_0$,
    $\Lambda \in \Omega_\eta^{(N)}$, and $M \in M_k \otimes \mathcal{Y}_N$,
    \[\|G_{M,s+T_N}(\Lambda)-G_{M,S_N+T_N}(\Lambda)\|
    \leq \frac{C}{N}\|M\|\|(\Im \Lambda)^{-1}\|^3(1+\|(\Im
    \Lambda)^{-1}\|^6).\]
    Furthermore, let $G'$ be the derivative in $\Lambda$ and
    $\|G'(\Lambda)\|=\sup_{x \in M_k:\|x\|=1} \|G'(\Lambda)[x]\|$. Then
    \[\|G_{M,s+T_N}'(\Lambda)-G_{M,S_N+T_N}'(\Lambda)\|
    \leq \frac{C}{N}\|M\|\|(\Im \Lambda)^{-1}\|^4(1+\|(\Im
    \Lambda)^{-1}\|^6).\]
\end{proposition}
\begin{proof}
    Applying Lemma \ref{lem:leftsubordination}
    with $\B=M_k$, $\tau^\B=\Id_k \otimes \tau_N$, $m=M$, $t=T_N$, and
    $b=\Lambda \otimes \Id_N$, we get
    \[G_{M,s+T_N}(\Lambda)=\Id_k \otimes \tr_N\Big(M g_{T_N}(\Gamma_\A)\Big)\]
    for $\Gamma_\A \equiv \Gamma_\A(\Lambda)=
    \Lambda-\cR_s(G_{s+T_N}(\Lambda))$ and
    $\|\Lambda^{-1}\|$ sufficiently small. Since both sides are analytic
    functions of $\Lambda \in M_k^+$, this must then hold for
    all $\Lambda \in M_k^+$.

Then applying Proposition \ref{master1} with this matrix $M$,
    \[\|G_{M,s+T_N}(\Lambda)-G_{M,S_N+T_N}(\Lambda)\|
    \leq \|\Theta_N(\Lambda,\Gamma_\A,M)\|+\|R_N(\Lambda,\Gamma_\A,M)\|.\]
By Proposition \ref{prop:Rbound}, for the first term we have
    $\|R_N(\Lambda,\Gamma_\A,M)\|
    \leq CN^{-1}\|M\|\|(\Im \Lambda)^{-1}\|^3$.
By Proposition \ref{prop:Thetabound}, for the second term we have
    $\|\Theta_N(\Lambda,\Gamma_N,M)\|
    \leq CN^{-2}\|M\|\|(\Im \Lambda)^{-1}\|^5$.
Recalling the definition of $\Theta_N$ and setting $K_{S_N+T_N}(\Lambda)=H_{S_N+T_N}(\Lambda)-G_{S_N+T_N}(\Lambda)$,
\begin{align*}
&\|\Theta_N(\Lambda,\Gamma_N,M)-\Theta_N(\Lambda,\Gamma_\A,M)\|\\
&\leq \E\left[\left\|M(g_{T_N}(\Gamma_N)-g_{T_N}(\Gamma_\A))\Big(\cR_s(K_{S_N+T_N}(\Lambda)) \otimes \Id_N\Big)h_{S_N+T_N}(\Lambda)\right\|\right]\\
&\hspace{0.2in}+\E\left[\left\|Mg_{T_N}(\Gamma_\A)\Big((\Gamma_N-\Gamma_\A) \otimes \Id_N \Big)h_{S_N+T_N}(\Lambda)\right\|\right].
\end{align*}
Applying again (\ref{eq:Gammadiff}) and the resolvent identity,
\[\|\Theta_N(\Lambda,\Gamma_N,M)-\Theta_N(\Lambda,\Gamma_\A,M)\|
\leq \frac{C}{N}\|M\|\|(\Im \Lambda)^{-1}\|^5\left(1+\|(\Im \Lambda)^{-1}\|^2\right)^2.\]
    Combining the above yields the desired bound on $G_{M,s+T_N}-G_{M,S_N+T_N}$.

For the difference of the derivatives, we apply the Cauchy integral formula.
Let $x \in M_{k}$ with $\|x\|=1$. Fix $\eta' \in (\eta,1/3)$. For
$r=\|(\Im\Lambda)^{-1}\|^{-1}/2$ and any $z \in \C$ with $|z|<r$,
note that $\Lambda+zx\in \Omega_{\eta'}^{(N)}$ because 
\[\Im(\Lambda+zx) \succeq \Im\Lambda-|r|\| x\|\Id_k \succeq
\Big(\|(\Im\Lambda)^{-1}\|^{-1}-r\Big)\Id_k \succ N^{-\eta'}\Id_k.\]
Define a path $\gamma$ by $\gamma(t)=re^{it}.$
Then by the Cauchy integral formula applied entrywise to the matrix-valued
analytic function $z \mapsto G_*(\Lambda+zx)$,
\begin{align*}
        &\|(G'_{M,s+T_{N}}(\Lambda)-G'_{M,S_{N}+T_{N}}(\Lambda))[x]\|\\
        &=\Big\|\frac{d}{dz}\Big|_{z=0}(M,G_{s+T_{N}}(\Lambda+zx)-G_{M,S_{N}+T_{N}}(\Lambda+zx))\Big\|\\
        &\leq \frac{1}{r}\max_{t\in[0,2\pi]}\Big\{\|G_{M,s+T_{N}}(\Lambda+\gamma(t)x)-G_{M,S_{N}+T_{N}}(\Lambda+\gamma(t)x)\|\Big\}\\
        &\leq \frac{C}{Nr}\|M\|\|(\Im\Lambda)^{-1}\|^{3}(1+\|(\Im \Lambda)^{-1}\|^6),
\end{align*}
where the last inequality comes the first part of the proposition
    applied to $\Omega_{\eta'}^{(N)}$. As $\eta<\eta'<1/3$ are arbitrary,
    replacing $\eta'$ by $\eta$ and applying $r^{-1}=2\|(\Im \Lambda)^{-1}\|$,
   the derivative bound follows.
\end{proof}

We now bound $A_2(j,m,l)$ following an argument similar
to \cite[Lemma 4.1]{schultz}.

\begin{proposition}\label{prop:A2bound}
Let $\eta<1/3$. Then for a constant $N_0>0$, all $N \geq N_0$, $\Lambda \in \Omega_\eta^{(N)}$, and $j \in \{1,\ldots,p\}$ and $m,l \in \{1,\ldots,k\}$,
\[\|A_{2}(j,m,l)\| \leq \frac{C}{N}\|(\Im \Lambda)^{-1}\|^5 (1+\|(\Im \Lambda)^{-1}\|^6).\]
\end{proposition}
\begin{proof}
For $x \in M_k$, consider the embeddings into $M_{2k}$ given by
    \[x^{11}=\begin{pmatrix}
    x&0\\
    0&0
    \end{pmatrix},\quad 
    x^{12}=\begin{pmatrix}
    0&x\\
    0&0
    \end{pmatrix},\quad x^{21}=\begin{pmatrix}
    0&0\\
    x&0
    \end{pmatrix},\quad
    x^{22}=\begin{pmatrix}
    0&0\\
    0&x\end{pmatrix}.\]
In the block decomposition with respect to $M_{2k} \otimes M_N=
    (M_k \otimes M_N) \oplus (M_k \otimes M_N)$, set
    \begin{align*}
        \tilde{S}_N
        &=\begin{pmatrix} S_N^\sT & 0 \\ 0 & S_N \end{pmatrix}
            =\sum_{j=1}^p ((a_j^\sT)^{11}+a_j^{22})\otimes X_j,\\
        \tilde{T}_N&=\begin{pmatrix} T_N^\sT & 0 \\ 0 & T_N \end{pmatrix}
            =\sum_{j=1}^q (b_j^\sT)^{11} \otimes Y_j^\sT+b_j^{22} \otimes Y_j.
    \end{align*}
    Define $g_{\tilde{T}_N},h_{\tilde{S}_N+\tilde{T}_N}:M_{2k} \to M_{2k}
    \otimes M_N$ analogously to $g_{T_N}$ and $h_{S_N+T_N}$.
    
    Define also $\tilde{\Lambda}=(\Lambda^\sT)^{11}+\Lambda^{22}$ and
    $\tilde{\Gamma}=(\Gamma^\sT)^{11}+\Gamma^{22}$. 
    Note that if $\Lambda,\Gamma \in M_k^+$,
    then $\tilde{\Lambda},\tilde{\Gamma} \in M_{2k}^+$. Furthermore, if $\|(\Im
    \Lambda)^{-1}\|<N^\eta$, then $\|(\Im \tilde{\Lambda})^{-1}\|<N^\eta$ also.
    For any $x,y \in M_k$ and $\Lambda,\Gamma \in M_k^+$, we have
    \begin{align*}
        &\begin{pmatrix}
    0&0\\
    0&g_{T_N}(\Gamma)(y \otimes \Id_N)h_{S_{N}+T_{N}}^{\sT}(\Lambda)(x\otimes \Id_{N})h_{S_{N}+T_{N}}(\Lambda)
    \end{pmatrix}\\
        &=g_{\tilde{T}_N}(\tilde{\Gamma})(y^{21}\otimes
        \Id_{N})h_{\tilde{S}_N+\tilde{T}_N}(\tilde{\Lambda})
        (x^{12} \otimes \Id_N)h_{\tilde{S}_N+\tilde{T}_N}(\tilde{\Lambda})\\
        &=\frac{d}{dt}\Big|_{t=0}g_{\tilde{T}_N}(\tilde{\Gamma})(y^{21}
        \otimes \Id_{N})h_{\tilde{S}_N+\tilde{T}_N}(\tilde{\Lambda}-tx^{12}).
    \end{align*}
Therefore, 
\begin{align*}
    &(\Id_k \otimes \tr_N)\Big[g_{T_N}(\Gamma)(y \otimes \Id_N)h_{S_N+T_N}^\sT(\Lambda)(x \otimes \Id_N)h_{S_N+T_N}(\Lambda)\Big]\\
    &=(\Tr \otimes \Id_{k})\Bigg[\frac{d}{dt}\Big|_{t=0}(\Id_{2k} \otimes
    \tr_N)\Big(g_{\tilde{T}_N}(\tilde\Gamma)(y^{21} \otimes
    \Id_N)h_{\tilde{S}_N+\tilde{T}_N}(\tilde{\Lambda}-tx^{12})\Big)
    \Bigg].
\end{align*}
We specialize this identity to $\Gamma=\Gamma_\A$, $y=a_je_{ml}^{(k)}$, and
    $x=e_{ml}^{(k)}a_j$. Set
    $\tilde{M}=g_{\tilde{T}_N}(\tilde{\Gamma}_\A)((a_je_{ml}^{(k)})^{21} \otimes
    \Id_N)$, and
    define for $w \in M_{2k}^+$ the left Cauchy transform
    \[\tilde{G}_{\tilde{M},\tilde{S}_N+\tilde{T}_N}(w)=(\Id_{2k} \otimes \tr_N)
    \E[\tilde{M}h_{\tilde{S}_N+\tilde{T}_N}(w)].\]
    Then we obtain that the second term defining $A_2(j,m,l)$ in (\ref{eq:A2})
    is equal to
    \[-\Tr \otimes \Id_{k}\Big[
        \tilde{G}'_{\tilde{M},\tilde{S}_N+\tilde{T}_{N}}
        (\tilde{\Lambda})[(e_{ml}^{(k)}a_j)^{12}]
    \Big].\]
Similar arguments in the space $M_{2k} \otimes \A_N$ yield that the first term
defining $A_2(j,m,l)$ is equal to
\[-\Tr \otimes \Id_{k}\Big[
    \tilde{G}'_{\tilde{M},\tilde{s}+\tilde{T}_{N}}(\tilde{\Lambda})
    [(e_{ml}^{(k)}a_j)^{12}] \Big],\]
where
\[\tilde{s}=\sum_{j=1}^p ((a_j^\sT)^{11}+a_j^{22}) \otimes x_j.\]

Taking the difference, we apply Proposition \ref{prop:derQ} with $2k$,
$2q$, and $Y_1^\sT,\ldots,Y_q^\sT,Y_1,\ldots,Y_q$ 
in place of $k$, $q$, and $Y_1,\ldots,Y_q$. Finally, using the bound
$\|\tilde{M}\| \leq C\|g_{T_N}(\Gamma_\A)\| \leq \|(\Im \Lambda)^{-1}\|$, we
get the desired bound for $A_2(j,m,l)$.
\end{proof}

Combining Propositions \ref{prop:A1bound} and \ref{prop:A2bound} for $A_1$ and
$A_2$, we get (\ref{eq:Rdiffbound}). Lemma \ref{GsT-GST+R/N} then follows from
this and Proposition \ref{prop:secondordersubordination}.

\subsection{The spectrum of \texorpdfstring{$L_N$}{LN}}

Recall the linear polynomials $L_N$ and $L_\A$ from (\ref{eq:LN}) and
(\ref{eq:LA}). We now apply Lemma \ref{GsT-GST+R/N} to obtain the following
spectral inclusion.

\begin{lemma}\label{inclusion_linear}
    In the setting of Theorem \ref{thm:strongfree},
    for any $k \geq 1$, self-adjoint linear $*$-polynomial $L$ with
    coefficients in $M_k(\C)$, and $\delta>0$, almost surely for all large $N$
    \begin{equation}\label{eq:inclusion_linear}
    \spec(L(\mathbf{X}_{N},\mathbf{Y}_{N}))\subseteq \spec
        (L(\mathbf{x},\mathbf{Y}_{N}))_\delta.
\end{equation}
\end{lemma}

For this, we specialize Lemma \ref{GsT-GST+R/N} to the scalar-valued Stieltjes
transforms of $L_N$ and $L_\A$. For $\lambda \in \C^+$, define
\begin{align*}
    g_N(\lambda)&=\E[(\tr_k\otimes\tr_N)(\lambda \Id_k\otimes
    \Id_{N}-L_N)^{-1}]=\tr_k (G_{S_N+T_N}(\lambda \Id_k-a_0)),\\
    g_\A(\lambda)&=(\tr_k\otimes\tau_N)(\lambda \Id_k\otimes \Id_N-L_\A)^{-1}
    =\tr_k (G_{s+T_N}(\lambda \Id_k-a_0)),\\
    r_\A(\lambda)&=\tr_k [\mathcal{L}_{(\lambda \Id_k-a_0)}
    (R_\A(\lambda \Id_k-a_0))]
\end{align*}
Then Lemma \ref{GsT-GST+R/N} applied with $\Lambda=\lambda \Id_k - a_0$ yields
\begin{equation}\label{eq:scalarbound}
    \Big|g_\A(\lambda)-g_N(\lambda)+r_\A(\lambda)\Big|\leq
\frac{C}{N^2}(\Im \lambda)^{-5}(1+(\Im \lambda)^{-10})
\end{equation}
for any $\eta \in (0,1/3)$, a constant $C \equiv C(\eta)>0$,
and all $\lambda \in \C^+$ such that $\Im \lambda \geq N^{-\eta}$.

As in \cite{schultz}, we first show the following.

\begin{proposition}\label{prop:rA}
    The function $r_\A(\lambda)$ is the Stieltjes transform of a distribution on
    $\R$ with support contained in $\spec(L_\A)$.
\end{proposition}
\begin{proof}
    By \cite[Theorem 5.4]{schultz}, it suffices to check that
    \begin{itemize}
    \item $r_\A(\lambda)$ is analytic on $\C \setminus \spec(L_\A)$,
    \item $r_\A(\lambda)\rightarrow 0$ as $|\lambda|\rightarrow\infty$, and
    \item There exists a constant $C>0$ and a compact set $K \subset \R$
        containing $\spec(L_\A)$ such that
        $|r_\A(\lambda)| \leq C\cdot \max\{\dist(\lambda,K)^{-3},1\}$
    for all $\lambda\in\C\setminus\R$.
    \end{itemize}

    The matrix $\Gamma_\A$ in (\ref{eq:RA}) is given by
    $\Gamma_\A(\lambda)=\lambda \Id_k-a_0
    -\mathcal{R}_{s}(G_{s+T_{N}}(\lambda \Id_k-a_0))$.
    For the first claim, if $\lambda \notin \spec(L_\A)$, then
    $G_{s+T_N}(\lambda \Id_k-a_0)$ exists and is analytic at $\lambda$. The
    subordination identity (\ref{eq:subordination})
    implies $G_{s+T_N}(\lambda \Id_k-a_0)=G_{T_N}(\Gamma_\A(\lambda))$ for
    all $\lambda \in \C^+$, and
    hence also for all $\lambda \notin \spec(L_\A)$ by analytic continuation.
    Then $g_{T_N}(\Gamma_\A(\lambda))$ also
    exists and is analytic at $\lambda$. Recalling the definition of $r_\A$
    above and of $R_\A$ from (\ref{eq:RA}), we see that
    $r_\A(\lambda)$ is analytic on $\C \setminus \spec(L_\A)$.
    
    For the second claim, note that for some constant $M>0$, 
    uniformly over $\lambda \in \C$ where $|\lambda| \geq M$, we have
    \[\|G_{s+T_{N}}(\lambda \Id_k-a_0)\|\leq
    \|(\lambda \Id_k\otimes \Id_N-L_\A)^{-1}\| \leq C/|\lambda|\]
    and similarly $\|G_{s+T_{N}}'(\lambda \Id_k-a_0)\| \leq C/|\lambda|^2$.
    Then also
    \[\|G_{T_N}(\Gamma_A(\lambda))\| \leq \frac{1}{|\lambda|-\|a_0\|
    -\|\mathcal{R}_{s}(G_{s+T_{N}}(\lambda \Id_k-a_0))\|-\|T_N\|}
    \leq C/|\lambda|.\]
    Thus $\|R_\A(\lambda \Id_k-a_0)\| \leq C|\lambda|^{-3}$, and
    $|r_\A(\lambda)| \leq C|\lambda|^{-3}(1+|\lambda|^{-2})$. In particular,
    $r_\A(\lambda) \to 0$ as $|\lambda| \to \infty$

    For the third claim, let $K=[-M,M]$. Over the region $\Re \lambda \in K$
    and $\Im \lambda \in [-M,M] \setminus \{0\}$, we apply the similar bound
    \[|r_\A(\lambda)| \leq C|\Im \lambda|^{-3}(1+|\Im \lambda|^{-2})\]
    to get $|r_\A(\lambda)| \leq C\dist(\lambda,K)^{-3}$. For $\lambda$ outside
    this region, the preceding argument implies $|r_\A(\lambda)|$ is uniformly
    bounded. The third claim follows.
\end{proof}

Combining this with (\ref{eq:scalarbound}), we get the following result.

\begin{lemma}\label{almostsurelemma}
    Fix any $M,\delta>0$ such that $\spec(L_\A)_\delta \subset [-M,M]$ for all
    large $N$. Consider any (sequence of) non-negative smooth functions
    $f_N:\R \to [0,1]$ such that
    \[f_N(x)=\begin{cases} 0 & x \in \spec(L_\A)_{\delta/2} \text{ or }
        x \notin [-M-\delta,M+\delta]\\
        1 & x \in [-M,M] \setminus \spec(L_\A)_\delta
    \end{cases}\]
    and $|f_N^{(k)}(x)| \leq C_k\delta^{-k}$ for each $k \geq 1$, some constants
    $C_k>0$, and all $x \in \R$.
    Then for any fixed $\kappa \in (0,1/2)$, almost surely as $N \to \infty$,
    \[N^{1+\kappa}(\tr_k\otimes\tr_N)(f_N(L_N))\rightarrow 0.\]
\end{lemma}
\begin{proof}
    The argument is similar to \cite{haagerupthorbjornsen} and \cite{male}, and we will
    omit most of the details. Since $f_N \equiv 0$ on $\spec(L_\A)$, we have from 
    Proposition \ref{prop:rA} and the Stieltjes inversion formula that
    \begin{align*}
        \E[(\tr_k\otimes\tr_N)f_N(L_N)]
        &=\lim_{y\rightarrow 0^+}-\frac{1}{\pi}
        \Im\left[\int_{\R} f_N(x)g_N(x+iy)]dx\right]\\
        &=\lim_{y\rightarrow 0^+}\frac{1}{\pi}
    \Im\left[\int_{\R}f_N(x)[g_\A(x+iy)+r_\A(x+iy)-g_N(x+iy)]dx\right].
    \end{align*}
    Then applying (\ref{eq:scalarbound}) and
    following the same arguments as \cite[Theorem 6.2]{haagerupthorbjornsen}, we get
    \[\E[(\tr_k\otimes\tr_N)f_N(L_N)] \leq C/N^2\]
    for a constant $C \equiv C(\delta)>0$.

    As in the proof of Proposition \ref{master1}, we write
    $X_{j}=\frac{1}{\sqrt{2}}(Z_{j}+Z_{j}^\sT)$ where $Z_j \in \R^{N \times N}$
    has i.i.d.\ $\N(0,1/N)$ entries. Defining
    \[F_N(Z_1,...,Z_p)=f_N\left(a_0\otimes \Id_N+\frac{1}{\sqrt{2}}
    \sum_{j=1}^p a_j\otimes (Z_j+Z_j^\sT)+\sum_{j=1}^q b_j\otimes Y_j\right),\]
    the Gaussian Poincar\'e inequality yields
    \[\Var[(\tr_k\otimes\tr_N)f_N(L_N)]\leq \frac{1}{N}\E\left[\|\nabla
    F_N(Z_1,...,Z_p)\|_2^2\right].\]
    The same argument as \cite[Proposition 4.7]{haagerupthorbjornsen} yields
    \[\|\nabla F_N(Z_1,...,Z_p)\|_2^2
    \leq \frac{C}{N}(\tr_k\otimes\tr_N)((f_N')^2(L_N)),\]
    where $(f_N')^2$ denotes the function $z \mapsto (f_N'(z))^2$. So
    \[\Var[(\tr_k\otimes\tr_N)f(L_N)] \leq \frac{C}{N^2}
    \E[(\tr_k\otimes\tr_N)((f_N')^2(L_N))].\]
    Applying the same argument as above,
    \begin{align*}
&\E[(\tr_k\otimes\tr_N)(f_N')^2(L_N)]\\
    &=\lim_{y\rightarrow 0^+}\frac{1}{\pi}
    \Im\left[\int_{\R}(f_N'(x))^2[g_\A(x+iy)+r_\A(x+iy)-g_N(x+iy)]dx\right]
    \leq C/N^2,
\end{align*}
    so $\Var[(\tr_k\otimes\tr_N)f_N(L_N)] \leq C/N^4$. Then by Markov's
    inequality,
    \[\P[(\tr_k\otimes\tr_N)f_N(L_N) \geq N^{-1-\kappa}]
    \leq N^{2+2\kappa}\E[((\tr_k\otimes\tr_N)f_N(L_N))^2]
    \leq CN^{2+2\kappa} \cdot N^{-4}.\]
    Taking $0<\kappa<1/2$, the result follows from Borel-Cantelli.
\end{proof}

Taking a constant $M>0$ large enough ensures $\spec(L_N) \subset [-M,M]$
almost surely for all large $N$. Then defining $f_N$ as in
Lemma \ref{almostsurelemma}, if there exists an eigenvalue of $L_N$
outside $\spec(L_\A)_\delta$, we must have
\[(\tau_k\otimes\tau_N)f_N(L_N) \geq N^{-1}.\]
Lemma \ref{almostsurelemma} guarantees this does not happen, almost
surely for all large $N$. This concludes the proof of
Lemma \ref{inclusion_linear}.

\subsection{Linearization trick and ultraproduct argument}

We conclude the proof of Theorem \ref{thm:strongfree} from
Lemma \ref{inclusion_linear} by applying the linearization trick and
ultraproduct argument of \cite{haagerupthorbjornsen}. As our algebra $\A$ is $N$-dependent,
we apply this argument in a subsequence form.

Let $M_k(\mathbb{Q}+i\mathbb{Q})_{sa}$ be the set of $k\times k$ Hermitian
matrices whose entries have rational real and imaginary parts. Define the
countable set
\[\mathcal{L}=\bigcup_{k=1}^\infty
\{\text{all linear} *\text{-polynomials of } p+q \text{ variables
with coefficients in }
M_k(\mathbb{Q}+i\mathbb{Q})_{sa}\}.\]
Let $\Omega$ denote the sample space. 
Let $\mathbf{Z}_N(\omega)=(\mathbf{X}_{N}(\omega),\mathbf{Y}_{N})$, $\mathbf{z}_N=(\mathbf{x},\mathbf{Y}_{N})$, for all $\omega\in\Omega$.
\begin{proof}[Proof of Theorem \ref{thm:strongfree}] 
Let $\Omega_0 \subset \Omega$
be the event where
\[\sup_{N \geq 1} \max_{i=1}^p
\left\|X_i^{(N)}(\omega)\right\|<\infty,\] 
and also where for each $L \in \mathcal{L}$ and (rational)
$\delta>0$, there exists $N_0(L,\delta,\omega)>0$ such that
\begin{equation}\label{step1:coro_linear_inclusion}
     \spec(L(\mathbf{Z}_N(\omega)))\subseteq \spec(L(\mathbf{z}_N))_\delta
\end{equation}
for all $N \geq N_0(L,\delta,\omega)$. By Lemma \ref{inclusion_linear},
$\Omega_0$ has probability 1.

We claim that (\ref{eq:polynomial_inclusion}) holds on $\Omega_0$.
Suppose by contradiction that this is false for some
non-commutative $*$-polynomial $Q$ (with coefficients in $\C$),
$\delta>0$, and $\omega \in \Omega_0$. Then at this $\omega$,
there is a subsequence $\{N_j\}$ and values
$\{\lambda_{N_{j}}\} \in \R$ such that for all $j$,
\begin{equation}\label{step3:assumption}
      \lambda_{N_{j}}\in \spec(Q(\mathbf{Z}_{N_{j}}(\omega))) \quad\text{but}\quad
      \lambda_{N_{j}}\not\in \spec(Q(\mathbf{z}_{N_{j}}))_\delta.
\end{equation}
Since $\spec(Q(\mathbf{Z}_N(\omega)))$ is uniformly bounded in $N$,
there is a further subsequence $\{N_{j_m}\}$
such that (\ref{step3:assumption}) still holds and
\begin{equation}\label{step3:converge}
      \lambda_{N_{j_m}}\to \lambda_0 \quad\text{as}\quad N_{j_m}\rightarrow\infty,
\end{equation}
for some $\lambda_0\in \R$. To ease notation, let us denote $\{N_{j_m}\}$ in the
following argument simply as $\{N\}$.

We introduce the quotient map defined in \cite[Proposition
7.3]{haagerupthorbjornsen}. Define the product and sum of the sequence of
algebras $\{\A_N\}_{N=1}^{\infty}$ by 
    \[\prod_N \A_N=\left\{\left(a_N\right)_{N=1}^{\infty}:a_N\in
    \A_N,\,\sup_N \left\|a_N\right\|<\infty\right\}\]
and
    \[\sum_N \A_N=\left\{\left(a_N\right)_{N=1}^{\infty}:a_N\in
    \A_N,\,\lim_{N\rightarrow\infty}\left\|a_N\right\|=0\right\}.\]
Then $\prod_N \A_N$ is a $C^*$-algebra (under coordinate-wise addition and
multiplication), and $\sum_N \A_N$ is a two-sided ideal. Thus, we can define a
quotient map by
    \[\pi_\A:\prod_N \A_N \longrightarrow \Big(\prod_N
    \A_N\Big)\Big/\Big(\sum_N \A_N \Big) \equiv \mathcal{C}_\A.\]
We identify $M_k \otimes \mathcal{C}_\A$ with
    \[\Big(\prod_N M_k\otimes\A_N \Big)\Big/\Big(\sum_N M_k\otimes\A_N\Big).\]
Similarly, define the product and sum of the matrix spaces
$\{M_N\}_{N=1}^{\infty}$, and a quotient map
    \[\pi:\prod_N M_N \longrightarrow \Big(\prod_N
    M_N\Big)\Big/\Big(\sum_N M_N \Big) \equiv \mathcal{C}.\]
Denote $\mathbf{Z}_N(\omega)=(\mathbf{X}_{N}(\omega),\mathbf{Y}_{N})$
and $\mathbf{z}_N=(\mathbf{x},\mathbf{Y}_{N})$.
Denote their images under the above quotient maps as
\begin{align*}
\mathbf{Z}'(\omega)&=(Z'_{i}(\omega))_{i=1}^{p+q}\\
&=\left(\pi\left(\left\{X_1^{(N)}(\omega)\right\}\right),\ldots,\pi\left(\left\{X_p^{(N)}(\omega)\right\}\right),\pi\left(\left\{Y_1^{(N)}\right\}\right),\ldots,\pi\left(\left\{Y_q^{(N)}\right\}\right)\right),\\
\mathbf{z}'&=(z'_{i})_{i=1}^{p+q}\\
&=\left(\pi_\A\left(\left\{x_1\right\}\right),\ldots,\pi_\A\left(\left\{x_p\right\}\right),\pi_\A\left(\left\{Y_1^{(N)}\right\}\right),\ldots,\pi_\A\left(\left\{Y_q^{(N)}\right\}\right)\right).
\end{align*}

We first claim that for every $L \in \mathcal{L}$,
\begin{equation}\label{step1:conclusion}
     \spec(L(\mathbf{Z}'(\omega)))\subseteq \spec(L(\mathbf{z}')).
 \end{equation}
 Indeed, fixing $L\in\mathcal{L}$, for any
 $\lambda\not\in \spec(L(\mathbf{z}'))$ there exists an element
 $w'\in M_k \otimes \mathcal{C}_\A$ such that
 $w'\left(\lambda-L(\mathbf{z}')\right)=1$.
 Letting $\left(w_N\right)_{N=1}^{\infty}\in  \prod_{N } M_{k }\otimes\A_N$ be
 such that $\pi_\A\left(\left\{w_N\right\}\right)=w',$ and noting that
 $\Id_k\otimes \pi_\A(\{L(\mathbf{z}_N )\})=L(\mathbf{z}')$, there must
 exist $\left(v_N\right)_{N=1}^{\infty}\in  \sum_{N } M_{k }\otimes\A_N$ such
 that for every $N$,
 \[w_{N}\left(\lambda \Id_k\otimes \Id_{N}-
 L(\mathbf{z}_N)\right)=\Id_k\otimes \Id_N+v_{N}.\]
 For $N$ large enough such that $\|v_N\|<1/2$, we get
    \[\lambda \not\in \spec( L(\mathbf{z}_{N }
     ))\quad\text{and}\quad\left\|\left(\lambda \Id_k\otimes\Id_{N }-
     L(\mathbf{z}_{N })\right)^{-1}\right\| \leq 2\sup_N \|w_N\|.\]
 Then $\dist(\lambda,\spec( L(\mathbf{z}_{N }))) \geq (2\sup_N \|w_N\|)^{-1}$.
 Applying (\ref{step1:coro_linear_inclusion}) with
 $\delta=(4\sup_N \|w_N\|)^{-1}$, we conclude that
 $\dist(\lambda,\spec( L(\mathbf{Z}_{N }(\omega)))) \geq (4\sup_N \|w_N\|)^{-1}$,
 so $\lambda \not\in \spec( L(\mathbf{Z}_{N }(\omega)))$ for all large $N$.
 Then defining $\{W_N\}_{N=1}^\infty$ by
 $(\lambda \Id_k\otimes \Id_N-L(\mathbf{Z}_{N }(\omega)))^{-1}$ for large $N$,
 we obtain that $\pi(\{W_N\})$ is the inverse of
 $\lambda-L(\mathbf{Z}'(\omega))$ in $M_k \otimes \mathcal{C}$.
 Thus, $\lambda\not\in \spec(L(\mathbf{Z}'(\omega)))$, so (\ref{step1:conclusion})
 holds.

 Then for this fixed $\omega$,
\cite[Theorem 2.4]{haagerupthorbjornsen} establishes the existence of a
unital $*$-homomorphism
\[\phi:\langle 1,z_1',\ldots,z_{p+q}'
\rangle \rightarrow  \langle 1,Z_1'(\omega),\ldots,Z_{p+q}'(\omega) \rangle\]
such that $\phi(z'_i)=Z'_i(\omega)$ for each $i \in \{1,\ldots,p+q\}$.
Note that if $x \in \langle 1,z_1',\ldots,z_{p+q}' \rangle$
is invertible with inverse $x^{-1}$,
then $\phi(x)$ is also invertible with inverse $\phi(x^{-1})$.
The assumption (\ref{step3:assumption}) implies
that $\dist(\lambda_N,\spec(Q(\mathbf{z}_N))) \geq \delta$ for all $N$. Then
\[\pi_\A\left(\{\lambda_N \Id_N-Q(\mathbf{z}_{N})\}_{N=1}^{\infty}\right)
=\pi_\A(\{\lambda_N \Id_N\})-Q(\mathbf{z}')\]
is invertible. Applying $\phi(Q(\mathbf{z}'))=Q(\mathbf{Z}'(\omega))$,
we get that
$\phi(\pi_\A(\{\lambda_N  \Id_N\}))-Q(\mathbf{Z}'(\omega))$ is also
invertible. From (\ref{step3:converge}), we obtain
\begin{equation}
    \pi(\{\lambda_N \Id_N\})=\pi(\{ \lambda_0 \Id_N\})=\lambda_0
    \1_{\mathcal{C}}=\phi(\lambda_0 \1_{\mathcal{C}_\A})=\phi(\pi_\A(\{\lambda_0
    \Id_N\}))=\phi(\pi_\A(\{\lambda_N \Id_N\})).
 \end{equation}
 Then $\pi(\{\lambda_N \Id_N-Q(\mathbf{Z}_{N}(\omega))\}_{N=1}^{\infty})$ is
 invertible. So
 $W_N(\lambda_N \Id_N-Q(\mathbf{Z}_{N}(\omega)))=\Id_N+V_N$ for some
 matrices $W_N,V_N$ with $\sup_N \|W_N\|<\infty$ and $\|V_N\| \to 0$.
 For large enough $N$, this contradicts the first statement of
 (\ref{step3:assumption}), that $\lambda_N \in \spec(Q(\mathbf{Z}_N))$,
 concluding the proof.
 \end{proof}

\begin{proof}[Proof of Theorem \ref{thm:strongfreeusual}]
The convergence in trace in (\ref{norm_convergence}) is known, see e.g.\
    \cite[Theorem 5.4.5]{AGZbook}. To verify the convergence in norm, it is
    sufficient to show almost surely
 \[\liminf_{N\rightarrow\infty}\|Q(\mathbf{X}_N,\mathbf{Y}_N)\|\geq \|Q(\mathbf{x},\mathbf{y})\|,\]
 \[\limsup_{N\rightarrow\infty}\|Q(\mathbf{X}_N,\mathbf{Y}_N)\|\leq \|Q(\mathbf{x},\mathbf{y})\|.\]
 The first inequality can be verified from the trace convergence and \cite[Lemma
    7.2]{haagerupthorbjornsen}. For the second inequality, because of the
    linearization trick, it suffices to prove that for any linear polynomial $L$
    with coefficients in $M_k$, any $\delta>0$, and all large $N$,
 \[\spec(L(\mathbf{X}_N,\mathbf{Y}_N))\subset\spec(L(\mathbf{x},\mathbf{y}))_\delta.\] Based on Lemma \ref{inclusion_linear}, it remains to show \[\spec(L(\mathbf{x},\mathbf{Y}_N))\subset\spec(L(\mathbf{x},\mathbf{y}))_\delta,\] 
 which is the main result in \cite[Section 7 and Appendix A]{male}.
\end{proof}

\section{Anisotropic resolvent approximation}\label{app:resolventapprox}

In this section, we prove Theorem \ref{thm:resolventapprox}.
We note that
for specific matrix models, stronger forms of Theorem \ref{thm:resolventapprox}
known as \emph{anisotropic local laws} were obtained in \cite{knowlesyin}, which
allow for $z \to \spec(w)$ at a near-optimal rate.
Our result is global, in that it considers only $z$ with constant
separation from $\spec(w)$, but it encompasses more complicated models
than those studied in \cite{knowlesyin} and provides a general recipe for how to
derive the approximation using free probability techniques.

When the rank-one matrix $vu^*$ is ``infinitesimally free'' of $W$ 
(for example, if $W$ is rotationally invariant with respect to $u,v$),
Theorem \ref{thm:resolventapprox} is also related to the work
of \cite{shlyakhtenko,collinsetal}, and the resolvent approximation is given by
\[u^*R(z)v \approx
\langle u,v \rangle \cdot N^{-1}\Tr R(z) \approx \langle u,v \rangle
\cdot m_0(z),\]
where $m_0$ is the Stieltjes transform of $w$. Our analysis extends to
anisotropic approximations, where $R_0(z) \neq m_0(z)\Id$. We require this
in our application, because signal eigenvectors in one covariance $\Sigma_r$
of the mixed model can have a non-random orientation with respect to the bulk
eigenvectors of a different covariance matrix $\Sigma_s$.

Our proof will proceed by first showing convergence of moments, and
then converting this information into convergence of the resolvents.

\subsection{Convergence for moments}\label{sec:moment-conv}

We first show the following result on convergence of moments.

\begin{theorem}\label{thm:anisotropicpoly}
Under the assumptions of Theorem \ref{thm:resolventapprox},
let $Q$ be any fixed $*$-polynomial of $p+q$ arguments,
with coefficients in $\D$, and
$v,w \in \C^N$ any deterministic vectors such that $\|v\|,\|w\| \leq C$.
Then almost surely as $N \to \infty$,
\[v^*Q(H_1,\ldots,H_p,B_1,\ldots,B_q)w-v^*\tau^\cH(Q(H_1,\ldots,H_p,b_1,\ldots,b_q))w \to 0.\]
\end{theorem}

Call a matrix $A \in \C^{N \times N}$ (or element $a \in \A$) simple if
$P_rAP_s=A$ (resp.\ $P_raP_s=a$) for some $r,s \in \{1,\ldots,k\}$. By
linearity, we may reduce Theorem \ref{thm:anisotropicpoly} to the following
setting.

\begin{lemma}\label{lemma:anisotropicpoly}
Fix the constants $C,c>0$.
Suppose, in addition to the assumptions of Theorem \ref{thm:anisotropicpoly},
that each $H_i$, $B_j$, and $b_j$ is simple for $i=1,\ldots,p$ and
$j=1,\ldots,q$. Then for any $m \geq 0$, any $j_1,\ldots,j_m \in \{1,\ldots,q\}$
and $\{i_1,\ldots,i_{m-1}\} \in \{1,\ldots,p\}$, and any deterministic
$v,w \in \C^N$ with $\|v\|,\|w\| \leq C$, almost surely as $N \to \infty$,
\begin{equation}\label{eq:inductiveclaim}
v^*B_{j_1}H_{i_1}\ldots B_{j_{m-1}}H_{i_{m-1}}B_{j_m}w
-v^*\tau^\cH(b_{j_1}H_{i_1}\ldots b_{j_{m-1}}H_{i_{m-1}}b_{j_m})w \to 0.
\end{equation}
\end{lemma}

We first explain why Theorem \ref{thm:anisotropicpoly} follows,
and then prove the lemma by induction on $m$.

\begin{proof}[Proof of Theorem \ref{thm:anisotropicpoly}]
Any $A \in \C^{N \times N}$ or $a \in \A$ is decomposed into simple elements as
\[A=\sum_{r=1}^k\sum_{s=1}^k P_rAP_s, \qquad a=\sum_{r=1}^k\sum_{s=1}^k
p_rap_s.\]
Then by linearity, it suffices to establish Theorem \ref{thm:anisotropicpoly}
for all $*$-monomials $Q$,
when each $H_1,\ldots,H_p$, $B_1,\ldots,B_q$, and $b_1,\ldots,b_q$ is simple.
Combining adjacent $x$'s and $y$'s in $Q$, and extending the families
$\{H_1,\ldots,H_p\}$ and $\{B_1,\ldots,B_q\}$ to include products and Hermitian
conjugates of these matrices as necessary, we may assume that
$Q$ is an alternating word in $x_i$'s and $y_i$'s. If $Q$ begins with $x_i$ or
ends with $x_j$, let us use $\tau^\cH(H_iaH_j)=H_i\tau^\cH(a)H_j$
and replace $v$ by $H_i^*v$ and $w$ by $H_jw$. Then the result follows from
Lemma \ref{lemma:anisotropicpoly}.
\end{proof}

\begin{proof}[Proof of Lemma \ref{lemma:anisotropicpoly}]
We induct on $m$. The result is clear for $m=0$, as $\tau^\cH(1)=1$ and
the left side of (\ref{eq:inductiveclaim}) is simply $v^*w-v^*w$. Suppose by
induction that the lemma holds up to $m-1$, and consider the case of $m$.
Introduce the centered elements
\[\mathring{H}_i=H_i-\tau^\D(H_i), \qquad \mathring{B}_j=B_j-\tau^\D(b_j),
\qquad \mathring{b}_j=b_j-\tau^\D(b_j).\]
(Note that here, we first center $B_j$ by $\tau^\D(b_j)$, not a normalized trace
of $B_j$.) On the left side of (\ref{eq:inductiveclaim}), let us write
$H_{i_r}=\mathring{H}_{i_r}+\tau^\D(H_{i_r})$ for each $i_r$,
and similarly for each $B_{j_r}$ and $b_{j_r}$. Expanding the resulting product,
we obtain that the left side of (\ref{eq:inductiveclaim}) is equal to
\begin{equation}\label{eq:centeredterm}
v^*\mathring{B}_{j_1}\mathring{H}_{i_1}\ldots \mathring{B}_{j_{m-1}}\
\mathring{H}_{i_{m-1}}\mathring{B}_{j_m}w
-v^*\tau^\cH(\mathring{b}_{j_1}\mathring{H}_{i_1}\ldots
\mathring{b}_{j_{m-1}}\mathring{H}_{i_{m-1}}\mathring{b}_{j_m})w
\end{equation}
plus a (constant) number of remainder terms which include at least one factor
$\tau^\D(H_i)$ or $\tau^\D(b_j)$. Since $H_i$ is simple, we have
either $\tau^\D(H_i)=0$ or $\tau^\D(H_i)=z(H_i) \cdot P_{r_i}$
for some $r_i \in \{1,\ldots,k\}$ and for $z(H_i)=\tau(H_i)/\tau(P_{r_i})
\in \C$, and similarly for $\tau^\D(b_j)$. Then,
absorbing $P_{r_i}$ into the adjacent factor and applying the arguments of the
proof of Theorem \ref{thm:anisotropicpoly} above,
each such remainder term may be written as a sum of differences of the form
(\ref{eq:inductiveclaim}) for a value $m' \leq m-1$, multiplied by an
$N$-dependent coefficient $z_N$
which is a product of a subset of the coefficients
$z(H_1),\ldots,z(H_p),z(b_1),\ldots,z(b_q)$. Since $\|\tau^\D(H_i)\| \leq
\|H_i\| \leq C$ and similarly for $b_j$, we have that $|z_N| \leq C$ for a
constant $C>0$ and all $N$. Then the remainder terms converge to 0 by the
inductive hypothesis.

It remains to show that the difference (\ref{eq:centeredterm}) converges to
0. We claim that
\begin{equation}\label{eq:freetracezero}
\tau^\cH(\mathring{b}_{j_1}\mathring{H}_{i_1}\ldots
\mathring{b}_{j_{m-1}}\mathring{H}_{i_{m-1}}\mathring{b}_{j_m})=0.
\end{equation}
Indeed, letting $\operatorname{NC}(m)$ be the set of non-crossing partitions of
$\{1,\ldots,m\}$ and introducing the $\cH$-valued non-crossing cumulants
$\kappa_\pi^\cH$, we have
\[\tau^\cH(\mathring{b}_{j_1}\mathring{H}_{i_1}\ldots
\mathring{b}_{j_{m-1}}\mathring{H}_{i_{m-1}}\mathring{b}_{j_m})
=\sum_{\pi \in \operatorname{NC}(m)}
\kappa_\pi^\cH(\mathring{b}_{j_1}\mathring{H}_{i_1},\;
\ldots,\mathring{b}_{j_{m-1}}\mathring{H}_{i_{m-1}},\;\mathring{b}_{j_m}).\]
Each partition $\pi$ has an element which is an interval $\{r,\ldots,r+\ell-1\}$
of consecutive indices, for some $\ell \geq 1$. Letting $\tau^\D$ be the
$\tau$-invariant projection onto $\D$, we apply \cite[Theorem 3.5]{NSS} and
freeness of $\cH$ and $\B$ over $\D$ to obtain
\[\kappa_\ell^\cH(b_1H_1,\ldots,b_{\ell-1}H_{\ell-1},b_\ell)
=\kappa_\ell^\D(b_1\tau^\D(H_1),\ldots,b_{\ell-1}\tau^\D(H_{\ell-1}),b_\ell)
=0\]
for any elements $b_1,\ldots,b_\ell \in \B$ and $H_1,\ldots,H_{\ell-1}
\in \cH$ which are zero-centered with respect to $\tau^\D$.
(In the case $\ell=1$, the
latter equality holds because $\kappa_1^\D(b_1)=\tau^\D(b_1)=0$.) Applying
this to the cumulant $\kappa_\ell^\cH$ of the terms corresponding to this
interval $\{r,\ldots,r+\ell-1\}$ of $\pi$, we obtain
$\kappa_\pi^\cH(\mathring{b}_{j_1}\mathring{H}_{i_1},\;
\ldots,\mathring{b}_{j_{m-1}}\mathring{H}_{i_{m-1}},\;\mathring{b}_{j_m})=0$
for each $\pi \in \operatorname{NC}(m)$, and hence (\ref{eq:freetracezero}).

Thus, to show that (\ref{eq:centeredterm}) converges to 0,
we must show that correspondingly
\begin{equation}\label{eq:matrixtracezero}
v^*\mathring{B}_{j_1}\mathring{H}_{i_1}\ldots \mathring{B}_{j_{m-1}}\
\mathring{H}_{i_{m-1}}\mathring{B}_{j_m}w \to 0.
\end{equation}
Since $H_i$ and $B_j$ are simple, some $(r_i,s_i)$ block of each
$\mathring{H}_i$ is non-zero and the remaining blocks are 0,
and some $(t_j,u_j)$ block of each $\mathring{B}_j$ is non-zero and
the remaining blocks are 0. We may suppose $u_{j_1}=r_{i_1}$,
$s_{i_1}=t_{j_2}$, $u_{j_2}=r_{i_2}$, etc., for otherwise the left side of
(\ref{eq:matrixtracezero}) is automatically 0. Denote by
\[\check{H}_i \in \C^{N_{r_i} \times N_{s_i}}\]
the non-zero block of $\mathring{H}_i$.
If $r_i \neq s_i$, then $\check{H}_i$ is just the corresponding
block $(H_i)_{r_is_i}$ of $H_i$. If $r_i=s_i$, then by the fact that $\tau$
coincides with $N^{-1}\Tr$ on $\A_1$,
$\check{H}_i=(H_i)_{r_ir_i}-N_{r_i}^{-1} \Tr H_i$
is the centered version of this block. Define also
\[\check{B}_j \in \C^{N_{t_j} \times N_{u_j}}\]
to be the non-zero block of $\mathring{B}_j$ if $t_j \neq u_j$,
or $\check{B}_j=(B_j)_{t_jt_j}-N_{t_j}^{-1} \Tr B_j$ if $t_j=u_j$. In
the latter case, note that $\check{B}_j$ differs from the nonzero block of
$\mathring{B}_j$ by the quantity
\begin{equation}\label{eq:tauTrBapprox}
\left(\frac{N}{N_{t_j}}\tau(B_j)-N_{t_j}^{-1} \Tr B_j\right)\Id_{N_{t_j}}
\to 0,
\end{equation}
where the convergence is in operator norm as $N \to \infty$ by (\ref{eq:bapprox}).
Finally, define $\check{v} \in \C^{r_{j_1}}$ to be the $r_{j_1}$ block of $v$, and
$\check{w} \in \C^{s_{j_m}}$ to be the $s_{j_m}$ block of $w$. Then 
\[\Big|v^*\mathring{B}_{j_1}\mathring{H}_{i_1}\ldots \mathring{B}_{j_{m-1}}
\mathring{H}_{i_{m-1}}\mathring{B}_{j_m}w
-\check{v}^*\check{B}_{j_1}\check{H}_{i_1}\ldots \check{B}_{j_{m-1}}
\check{H}_{i_{m-1}}\check{B}_{j_m}\check{w}\Big| \to 0,\]
almost surely as $N \to \infty$, by the observation (\ref{eq:tauTrBapprox})
and the operator norm bound on each $H_i$ and $B_j$. So it suffices to show
\[\check{v}^*\check{B}_{j_1}\check{H}_{i_1}\ldots \check{B}_{j_{m-1}}
\check{H}_{i_{m-1}}\check{B}_{j_m}\check{w} \to 0.\]

Let us introduce a random orthogonal matrix
\[O=\diag(O_1,\ldots,O_k) \in \mathcal{O}\]
where each $O_r \in \R^{N_r \times N_r}$ is independently Haar-distributed on
the orthogonal group and
also independent of $B_1,\ldots,B_q$. By the assumed conjugation invariance of
$(B_1,\ldots,B_q)$, we have the equality in law
\[(\check{B}_1,\ldots,\check{B}_q)\overset{L}{=}
(O_{t_1}\check{B}_1O_{u_1}^{-1},\ldots,
O_{t_q}\check{B}_qO_{u_q}^{-1}),\]
and thus we may equivalently show (almost surely as $N \to \infty$)
\begin{equation}\label{eq:almostsureconvergence}
\check{v}^*O_{t_{j_1}}\check{B}_{j_1}O_{u_{j_1}}^{-1}\check{H}_{i_1}
\ldots O_{t_{j_m}}\check{B}_{j_m}O_{u_{j_m}}^{-1}\check{w} \to 0.
\end{equation}
We then condition on $\check{B}_1,\ldots,\check{B}_q$, and write $\E$
for the expectation over $O_1,\ldots,O_k$. Defining
\[\mathcal{E}=\E\Big[|\check{v}^*O_{t_{j_1}}\check{B}_{j_1}O_{u_{j_1}}^{-1}
\check{H}_{i_1}\ldots
O_{t_{j_m}}\check{B}_{j_m}O_{u_{j_m}}^{-1}\check{w}|^4\Big],\]
we observe that this may be written in the form
\[\mathcal{E}=\E[\Tr O_{r_1}^{e_1}D_1O_{r_2}^{e_2}D_2\ldots
O_{r_{8m}}^{e_{8m}}D_{8m}]\]
where
\begin{itemize}
\item Each $r_i \in \{1,\ldots,k\}$ and each $e_i \in \{-1,1\}$.
\item Each $D_i$ is one of $\check{H}_1,\ldots,\check{H}_p$,
$\check{B}_1,\ldots,\check{B}_q$, $\check{w}\check{v}^*$,
$\check{w}\check{v}^\sT$ or their Hermitian conjugates.
\item If $r_i=r_{i+1}$ and $D_i$ is not of the form
$\check{w}\check{v}^*$, $\check{w}\check{v}^\sT$ or their conjugates,
then the centering of $\check{H}$ and $\check{B}$ implies $\Tr D_i=0$.
\item At least four of the matrices $D_1,\ldots,D_{8m}$ are of rank 1.
\end{itemize}
Then Lemma \ref{lemma:combinatorics} below implies (conditional on
$\check{B}_1,\ldots,\check{B}_q$ for all $N$, and on the event of probability
1 where $\|\check{B}_1\|,\ldots,\|\check{B}_q\|<C'$ for a constant $C'>0$ and
all large $N$) that $\mathcal{E} \leq CN^{-2}$. Then
(\ref{eq:almostsureconvergence}) holds almost surely as $N \to \infty$ by
Markov's inequality and Borel-Cantelli, as desired.
\end{proof}

\begin{lemma}\label{lemma:combinatorics}
Fix constants $B,C,c>0$ and suppose $c<N_r/N<C$ for each $r=1,\ldots,k$.
Let $O_1,\ldots,O_k$ be independent matrices, with each
$O_r \in \R^{N_r \times N_r}$ Haar-distributed on the orthogonal group.

Fix $M \geq 1$, $r_1,\ldots,r_M \in \{1,\ldots,k\}$, $e_1,\ldots,e_M \in
\{-1,1\}$, and cyclically identify $r_{M+1} \equiv r_1$. For each
$m=1,\ldots,M$, let $D_m \in \C^{N_{r_m} \times N_{r_{m+1}}}$
be a deterministic matrix with $\|D_m\|<B$. For each $m$,
suppose at least one of the following holds:
\begin{itemize}
\item $r_m \neq r_{m+1}$, or
\item $D_m$ is of rank 1, or
\item $r_m=r_{m+1}$ and $\Tr D_m=0$.
\end{itemize}
Finally, suppose that at least $K$ of $D_1,\ldots,D_M$ have rank 1. 
Then for a constant $C' \equiv C'(k,K,M,B)>0$,
\[\E[\Tr O_{r_1}^{e_1}D_1O_{r_2}^{e_2}D_2\ldots O_{r_M}^{e_M}D_M]
\leq C'N^{-K/2}.\]
\end{lemma}
\begin{proof}
The proof of this lemma is similar to that of \cite[Lemma B.2]{fanjohnstonebulk}, which
established a version of this result for $K=0$.
We extend the combinatorial argument here to handle the case of general $K$.
To ease subscript notation, we write $v[i]$ and $A[i,j]$ for entry
$i$ of $v$ and entry $(i,j)$ of $A$. We denote by $C>0$ a constant which may
depend on $k,K,M,B$ and change from instance to instance.

We may write
\[\mathcal{E} \equiv
\E[\Tr O_{r_1}^{e_1}D_1O_{r_2}^{e_2}D_2\ldots O_{r_M}^{e_M}D_M]
=\sum_{\i,\j} D(\i,\j)\E[V(\i,\j)],\]
where the sum is over all tuples
$(\i,\j)=(i_1,\ldots,i_M,j_1,\ldots,j_M)$ satisfying
\[1 \leq i_k,j_k \leq N_{r_k}\]
and where
\[V(\i,\j)=\prod_{m=1}^M O_{r_m}^{e_m}[i_m,j_m],
\qquad D(\i,\j)=\prod_{m=1}^M D_m[j_m,i_{m+1}]\]
with the cyclic identification $i_{M+1} \equiv i_1$. Define the set partition
\[
\bigsqcup_{r = 1}^k \cI(r) = \{1, \ldots, M\}
\]
by $\cI(r) = \{m : r_m = r\}$.  Consider now set partitions of the set $\{1,
    \ldots, M\} \sqcup \{1, \ldots, M\}$ of cardinality $2M$,
    where we denote elements of the first copy of $\{1, \ldots, M\}$ with a
    subscript $i$ and the second with a subscript $j$. A set in this partition
    can have elements of either or both copies of $\{1,\ldots,M\}$;
    for example, $\{1_i, 2_j\}$ or $\{2_j, 3_j\}$ might be sets in the set partition.  We say that $\i, \j$ induces $\Q$, denoted $\i, \j \mid \Q$, if
\[
\Q = \bigsqcup_{r = 1}^k \bigsqcup_{s = 1}^{N_r} (\Q^1(r, s) \sqcup \Q^2(r, s)),
\]
for
\begin{align*}
  \Q^1(r, s) &= \{m_i : m \in \cI(r), i_m = s, e_m = 1\} \cup \{m_j : m \in \cI(r), j_m = s, e_m = -1\}\\
  \Q^2(r, s) &= \{m_i : m \in \cI(r), i_m = s, e_m = -1\} \cup \{m_j : m \in \cI(r), j_m = s, e_m = 1\}.
\end{align*}
Denote $\Q(r) := \bigsqcup_{s = 1}^{N_r} (\Q^1(r, s) \sqcup \Q^2(r, s))$, and let $|\Q|$ be the total number of non-empty sets in $\Q$.

Notice that the quantity
\[\E[V(\i,\j)] \equiv E(\Q)\]
depends on $(\i,\j)$ only via its induced partition $\Q$. By
\cite[Lemma B.3(a)]{fanjohnstonebulk} we have $|E(\Q)| \leq CN^{-M/2}$
for any partition $\Q$. Thus we find
\begin{equation}\label{eq:Ebound}
\mathcal{E} \leq CN^{-M/2} \sum_{\Q:E(\Q) \neq 0} |D(\Q)|, \qquad
D(\Q) \equiv \sum_{\i,\j \mid \Q} D(\i,\j),
\end{equation}
so our main task is to bound $|D(\Q)|$ when $E(\Q) \neq 0$.
By \cite[Lemma B.3(b)]{fanjohnstonebulk}, if $\i, \j \mid \Q$ and $E(\Q) \neq 0$,
then for each $r \in \{1, \ldots, k\}$ and each $s \in \{1,\ldots,N_r\}$, the cardinality of $|\Q^1(r, s)|$ and $|\Q^2(r, s)|$
must be even. That is, each set $S \in \Q$ has even cardinality. To motivate the combinatorial idea,
note that the bound $|D_m[j_m,i_{m+1}]| \leq B$ implies that $D(\i, \j) \leq B^M$ for all $(\i, \j)$, while
\[
\#\{(\i, \j) : \i, \j \mid \Q\} \leq C N^{|\Q|},
\]
since for any fixed $\Q$ choosing $\i, \j$ which induce $\Q$ involves choosing for each set in $\Q$
a distinct index from $\{1, \ldots, N_r\}$ for some $r$.  Together, these yield the naive bound $|D(\Q)| \leq CN^{|\Q|}$. Since each set in $\Q$ has
cardinality at least 2, and the sum of all cardinalities is $2M$, we have
$|\Q| \leq M$. Combining with (\ref{eq:Ebound}) would yield
\[\mathcal{E} \leq CN^{-M/2} \cdot N^M,\]
but the exponent is too large in $M$ and does not depend on the
number of rank 1 matrices $K$.

This motivates the definitions of the following counts associated to $\Q$.
For $m \in \{1, \ldots, M\}$, call the index $m_i$ single if $D_{m - 1}$ is of rank $1$ and the index $m_j$ single if $D_m$ is of rank $1$---that is,
an index is single if it corresponds to some rank 1 matrix in the product $D(\i,\j)$.
For a fixed set partition $\Q$, define the following quantities.
\begin{itemize}
\item $T_0$: number of sets in $\Q$ of cardinality 2,
which contain no single indices.
\item $T_1$: number of sets in $\Q$ of cardinality 2,
which contain 1 or 2 single indices.
\item $R_0$: number of sets in $\Q$ of cardinality $\geq 4$,
which contain no single indices.
\item $R_1$: number of sets in $\Q$ of cardinality $\geq 4$,
which contain (exactly) 1 single index.
\end{itemize}
We establish the following claim by induction on $T_0 + T_1$.\\

\noindent {\bf Inductive claim:} For any $M \geq 1$, any
$r_1,\ldots,r_M,e_1,\ldots,e_M,D_1,\ldots,D_M$ which satisfy the conditions
of the lemma, and any such partition $\Q$ of $\{1, \ldots, M\} \sqcup \{1, \ldots, M\}$ with $T_0,T_1,R_0,R_1$ as defined above,
\begin{equation}\label{eq:DQbound}
|D(\Q)| \leq C_0N^{R_0+T_0/2+R_1/2}
\end{equation}
for a constant $C_0 \equiv C_0(k,M,T_0,T_1,R_0,R_1,B)>0$.\\

Assuming that this claim holds, note that the number of non-single indices
is $2(M-K)$, where $K$ is the number of rank 1 matrices.
Then $2(M-K) \geq 4R_0+2T_0+3R_1$. Dividing this by 4 gives the improved bound
\[|D(\Q)| \leq C_0N^{(M-K)/2}.\]
Combining with (\ref{eq:Ebound}) yields
$\mathcal{E} \leq CN^{-K/2}$, as desired.

To establish (\ref{eq:DQbound}), we induct on the total number of elements of
$\Q$ of cardinality 2, which is $T_0+T_1$. For the base case $T_0+T_1=0$,
let us assume for notational convenience that $D_1,\ldots,D_K$ are of rank 1.
For $m=1,\ldots,K$, we write $D_m=v_mw_m^*$ for bounded length vectors $v_m$ and $w_m$, and apply
$|D_m[i,j]| \leq B$ for $m=K+1,\ldots,M$. This gives
\begin{equation} \label{eq:d-base-bound}
  |D(\Q)| \leq C\sum_{\i,\j \mid \Q} |v_1[j_1]w_1[i_2]\cdots v_K[j_K]w_K[i_{K+1}]|.
\end{equation}
Let $R_2$ be the number of elements of $\Q$ containing two or more single
indices. Since $\Q$ has no elements of cardinality 2, all elements of $\Q$ are counted
by $R_0$, $R_1$, or $R_2$. We now view the sum in (\ref{eq:d-base-bound}) as a product of sums over
distinct indices for the elements of $\Q$ counted by $R_0, R_1, R_2$.  We bound the sum over distinct indices counted by $R_0$ simply by
$CN^{R_0}$. For the sum over distinct indices counted by $R_1$, note by Cauchy-Schwartz that
\[
\sum_i |u[i]| \leq \sqrt{\|u\|} \sqrt{N},
\]
yielding a combined bound of $CN^{R_1/2}$ for these indices because $\|u\|$ is bounded for the relevant vectors. For distinct
indices counted by $R_2$, we apply a bound of the form
\[\sum_i |u_1[i]\ldots u_m[i]| \leq C\sum_i |u_1[i]u_2[i]|
\leq C\|u_1\|\cdot \|u_2\|\]
for any $m \geq 2$ and any bounded vectors $u_1, \ldots, u_m$, yielding a constant bound for the combined sum over such
indices. Thus, we get
\[|D(\Q)| \leq CN^{R_0+R_1/2},\]
which concludes the proof of (\ref{eq:DQbound}) in this base case.

Assume inductively that (\ref{eq:DQbound}) holds for $T_0+T_1 \leq t-1$,
and consider now $T_0+T_1=t \geq 1$. Then there is some set $S \in \Q$
with cardinality $|S|=2$. We consider three cases.

{\bf Case 1:} $S=\{m_j, (m + 1)_i\}$, and $D_m$ is \emph{not} of rank 1. (So $S$ is
counted by $T_0$.) Suppose for notational convenience that $S=\{1_j, 2_i\}$.
This implies in particular that $r_1=r_2$ and $D_1$ is square.
Then the assumption of the lemma implies
\[\Tr D_1=0.\]
Denote by $\sum_{\i,\j|\Q \setminus S}$ the sum over indices in the tuple $(\i,\j)$ excluding
$j_1$ and $i_2$ which induce $\Q \setminus S$, and by
$\sum_{j \notin \Q(r_1) \setminus S}$
the remaining sum over the value of $j_1 \equiv i_2$,
restricted to be distinct from the $|\Q(r_1)| - 1$ preceding values
in $\{1,\ldots,N_{r_1}\}$ assumed by sets in $\Q(r_1) \setminus S$. Then
\[D(\Q)=\sum_{\i,\j|\Q \setminus S}
\prod_{m=2}^M D_m[j_m,i_{m+1}]
\cdot \sum_{j \notin \Q(r_1) \setminus S} D_1[j,j].\]
Let $\Pi$ be the set of new partitions $\Q'$ which merge $S=\{1_j, 2_i\}$ with
some other set in $\Q(r_1) \setminus S$. Then applying $\Tr D_1=0$ yields
\[D(\Q)=-\sum_{\Q' \in \Pi} \sum_{\i,\j|\Q'}
\prod_{m=1}^M D_m[j_m,i_{m+1}],\]
and hence
\[|D(\Q)| \leq \sum_{\Q' \in \Pi} |D(\Q')|.\]
As $D_m$ is not of rank 1, the indices $1_j, 2_i$ are not single.
If $\{1_j, 2_i\}$ was merged into a set in $\Q$
of cardinality $\geq 4$, then $\Q'$ has the counts $(T_0-1,T_1,R_0,R_1)$.
If $\{1_j,2_i\}$ was merged into a set in $\Q$ counted by $T_1$, then
$\Q'$ has either the counts $(T_0-1,T_1-1,R_0,R_1)$ or
$(T_0-1,T_1-1,R_0,R_1+1)$.
If $\{1_j,2_i\}$ was merged into another set in $\Q$ counted by $T_0$, then
$\Q'$ has the counts $(T_0-2,T_1,R_0+1,R_1)$. In all cases, $T_0+T_1$ has
reduced by at least 1, and the exponent $R_0+T_0/2+R_1/2$ in (\ref{eq:DQbound})
has not increased. Then applying the inductive hypothesis for each $\Q'$ and
noting that the cardinality of $\Pi$ is a constant independent of $N$, we get
(\ref{eq:DQbound}) for $\Q$.

{\bf Case 2:} $S=\{m_j,(m+1)_i\}$, and $D_m$ \emph{is} of rank 1. (So $S$ is
counted by $T_1$.) Suppose for notational convenience $S=\{1_j,2_i\}$. Then with
the same notation as defined in Case 1, we get
\[|D(\Q)| \leq |\Tr D_1| \cdot \left|\sum_{\i,\j|\Q \setminus S} \prod_{m=2}^M
D_m[j_m,i_{m+1}]\right|+
\sum_{\Q' \in \Pi} \left|\sum_{\i,\j|\Q'}
\prod_{m=1}^M D_m[j_m,i_{m+1}]\right|,\]
where the first term arises because we no longer have $\Tr D_1=0$. (If $M=1$,
the first term is understood to just be $|\Tr D_1|$.) Note that $|\Tr D_1| \leq
C$, as $D_1$ has bounded operator norm and is of rank 1. The partition $\Q
\setminus S$ in the first term must have the counts $(T_0,T_1-1,R_0,R_1)$, and
we may apply the inductive hypothesis to this term. For each $\Q'$ in the second
term, the argument is a bit different from Case 1 as $1_j,2_i$ are single.
If $\{1_j,2_i\}$ was merged into a set in $\Q$ counted by $T_0$,
$T_1$, $R_0$, $R_1$, or none of these four,
then $\Q'$ has the counts $(T_0-1,T_1-1,R_0,R_1)$,
$(T_0,T_1-2,R_0,R_1)$, $(T_0,T_1-1,R_0-1,R_1)$, $(T_0,T_1-1,R_0,R_1-1)$,
or $(T_0,T_1-1,R_0,R_1)$ respectively. Applying the inductive hypothesis
in all cases, we get (\ref{eq:DQbound}) for $\Q$.

{\bf Case 3:} The two indices in $S$ do not index the same matrix $D_m$.
Suppose for notational convenience $S=\{2_i,2_j\}$, so that they index $D_1$ and
$D_2$; other cases are analogous. Then with similar notation as in Case 1, we have
\begin{align*}
D(\Q)&=\sum_{\i,\j|\Q \setminus S}
\prod_{m=3}^M D_m[j_m,i_{m+1}]
\cdot \sum_{i \notin \Q(r_2) \setminus S} D_1[j_1,i]D_2[i,i_3].
\end{align*}
Let us introduce the matrix $\tilde{D}=D_1D_2$. Then applying the triangle
inequality as in Cases 1 and 2,
\begin{align*}
|D(\Q)|& \leq \left|\sum_{\i,\j|\Q \setminus S} \tilde{D}[j_1,i_3]
\prod_{m=3}^M D_m[j_m,i_{m+1}]\right|
+\sum_{\Q' \in \Pi} \left|\prod_{m=1}^M D_m[j_m,i_{m+1}]\right|,
\end{align*}
where $\Pi$ is the set of partitions
$\Q'$ which merge $\{2_i,2_j\}$ with another set in 
of $\Q(r_2) \setminus S$. (The product in the first term is understood to be $1$
if $M=2$.)

For the first term involving $\Q \setminus S$, note that if $\tilde{D}$ is not
of rank 1, then both $D_1$ and $D_2$ are also not of rank 1.
So $2_i,2_j,1_j,3_i$ were
not single in $\Q$, and $1_j,3_i$ remain non-single in $\Q \setminus S$
(with respect to
$\tilde{D},D_3,\ldots,D_M$). Then
$\Q \setminus S$ must have the counts $(T_0-1,T_1,R_0,R_1)$. If $\tilde{D}$ is
of rank 1, then the removal of $\{2_i,2_j\}$ reduces either $T_0$ or $T_1$ by 1,
but it is possible that $1_j$ and/or $3_i$ may have been converted
from a non-single index in $\Q$ to a single index in $\Q \setminus S$. One such
conversion may induce the count mapping $(T_0,T_1) \mapsto (T_0,T_1-1)$,
$(T_0,T_1) \mapsto (T_0-1,T_1+1)$, $(R_0,R_1) \mapsto (R_0,R_1-1)$, or
$(R_0,R_1) \mapsto (R_0-1,R_1+1)$. Note that each of these mappings does not
increase $T_0+T_1$, nor increase the exponent $R_0+T_0/2+R_1/2$ of $N$
in (\ref{eq:DQbound}). Then we may apply the induction hypothesis in every case
to obtain $|D(\Q \setminus S)| \leq CN^{R_0+T_0/2+R_1/2}$ for the first term.

For each $\Q' \in \Pi$ of the second term, we perform some casework,
depending on whether $2_i,2_j$ are both non-single (so $D_1$ and $D_2$
both have rank more than 1), and also
whether $\{2_i,2_j\}$ was merged into a set in $\Q$
counted by $T_0,T_1,R_0,R_1$ or none of these four. The possible
resulting counts for $\Q'$ are summarized in Table \ref{tab:count}.
In each setting, $T_0+T_1$ has reduced by at least 1, the exponent
$R_0+T_0/2+R_1/2$ has not increased, and we may thus
apply the induction hypothesis for $\Q'$ to obtain (\ref{eq:DQbound}) for $\Q$.

\renewcommand{\arraystretch}{1.25}
\begin{table}[h!]
\begin{tabular}{lll} \hline 
Merged into & $2_j,2_i$ not single & one or both of $2_j,2_i$ single\\ \hline
$T_0$ & $T_0-2,T_1,R_0+1,R_1$  & \makecell[cl]{$T_0-1,T_1-1,R_0,R_1$ or\\$T_0-1,T_1-1,R_0,R_1+1$}\\
$T_1$ & \makecell[cl]{$T_0-1,T_1-1,R_0,R_1$ or\\$T_0-1,T_1-1,R_0,R_1+1$} & $T_0,T_1-2,R_0,R_1$\\
$R_0$ & $T_0-1,T_1,R_0,R_1$&\makecell[cl]{$T_0,T_1-1,R_0-1,R_1$ or\\$T_0,T_1-1,R_0-1,R_1+1$} \\
$R_1$ & $T_0-1,T_1,R_0,R_1$& $T_0,T_1-1,R_0,R_1-1$ \\
None of above & $T_0-1,T_1,R_0,R_1$ & $T_0,T_1-1,R_0,R_1$ \\ \hline 
\end{tabular}
\caption{Possible counts for $\Q'$}
\label{tab:count}
\end{table}

This establishes that (\ref{eq:DQbound}) holds when $T_0+T_1=t$, in all three of
the above Cases. This completes the induction and the proof of the lemma.
\end{proof}

\subsection{Convergence for resolvent} \label{sec:resolvent}

Finally, we use Theorem \ref{thm:anisotropicpoly} to complete the proofs of
Theorem \ref{thm:resolventapprox} and Corollary \ref{cor:resolventapprox}. This will
depend on the following lemma, which allows us to work with a series expansion of the Stieltjes transform.

\begin{lemma} \label{lem:analytic-extend}
Let $C > 0$ be such that $\|W\| \leq C$ and $\|w\| \leq C$ for large $N$, and
    suppose that $f_N$ is an analytic function on $\C \setminus \spec(W)$ and
    $f$ an analytic function on $\C \setminus \spec(w)$ such that almost surely
    as $N \to \infty$, we have $f_N-f \to 0$ uniformly on $\mathbb{D}'=\{z \in
    \C:|z|>2C\}$.  Then for any fixed constant $\delta > 0$, almost surely,
    $f_N-f \to 0$ uniformly on $\mathbb{D}_N = \{z \in \C : \dist(z, \spec(w)) \geq \delta \text{ and } \dist(z, \spec(W)) \geq \delta\}$. 
\end{lemma}
\begin{proof}
Let $\Omega_0$ be the event of probability 1 where $\spec(W)$ (and
also $\spec(w)$) are uniformly bounded in $[-C,C]$ for all large $N$, and
\[
\lim_{N \to \infty} \sup_{z \in \mathbb{D}'} |f_N(z) - f(z)|=0.
\]
Suppose by contradiction that for some $\omega \in \Omega_0$ and
$\eps>0$, we have
\begin{equation}\label{eq:Dconvergence}
\limsup_{N \to \infty}
\sup_{z \in \mathbb{D}_N} |f_N(z) - f(z)|>\eps.
\end{equation}
Then there is a subsequence $\{N_k\}_{k=1}^\infty$ and points $z_{N_k} \in
\mathbb{D}_{N_k}$ for which $|f_{N_k}(z_{N_k}) - f(z_{N_k})|>\eps$ for all
$k$. Since $\spec(W)$ and $\spec(w)$ are uniformly bounded compact
subsets of $\R$, by sequential compactness under Hausdorff distance,
there must be a further subsequence of $\{N_k\}_{k=1}^\infty$ along which these 
sets converge in Hausdorff distance to fixed limits $S_1 \equiv S_1(\omega)$ and
$S_2 \equiv S_2(\omega)$. Define $\mathbb{D}_\infty(\omega)=
\{z \in \C:\dist(z,S_1) \geq \delta/2,\,\dist(z,S_2) \geq
\delta/2\}$. Then $\mathbb{D}_\infty(\omega)$ is a fixed ($N$-independent)
connected domain of $\C$. As $f_N(z) - f(z)$ is analytic on
$\mathbb{D}_\infty(\omega)$ for all large $N$, we then have
\[
\lim_{N \to \infty} \sup_{z \in \mathbb{D}_\infty(\omega)} |f_N(z) - f(z)|=0,
\]
by the convergence over $z \in \mathbb{D}'$. This implies $z_{N_k} \notin
\mathbb{D}_\infty(\omega)$ for all large $k$. But then
\[
\limsup_{k \to \infty}\;\min(\dist(z_{N_k},S_1),\dist(z_{N_k},S_2)) \leq \delta/2,
\]
which implies by the definition of Hausdorff distance that
\[
\limsup_{k \to \infty}\;\min(\dist(z_{N_k},\spec(w)),\dist(z_{N_k},\spec(W)))
\leq \delta/2,
\]
contradicting that $z_{N_k} \in \mathbb{D}_{N_k}$. Thus
(\ref{eq:Dconvergence}) cannot hold for any $\omega \in \Omega_0$.
\end{proof}

\begin{proof}[Proof of Theorem \ref{thm:resolventapprox}]
The given assumptions imply that there is a constant $C>0$ such that $\|W\| \leq C$ and $\|w\| \leq C$ almost surely for all large $N$. Let
$\mathbb{D}'=\{z \in \C:|z|>2C\}$.
Fix $\eps>0$. Applying the contractive property $\|\tau^\cH(a)\| \leq \|a\|$
of conditional expectations, there is $K>0$ such that
\[\sup_{z \in \mathbb{D}'}
\left\|\sum_{k=K+1}^\infty z^{-(k+1)}W^k\right\|<\eps,
\qquad \sup_{z \in \mathbb{D}'}
\left\|\sum_{k=K+1}^\infty z^{-(k+1)}\tau^\cH(w^k)\right\|<\eps\]
for all large $N$. For each $k \in \{0,\ldots,K\}$, Theorem
\ref{thm:anisotropicpoly} implies $u^*W^kv-u^*\tau^\cH(w^k)v \to 0$ almost
surely. Then
applying the series expansions for $(w-z)^{-1}$ and $(W-z\Id)^{-1}$, convergent
for $|z|>2C$, we get
\[\limsup_{N \to \infty} \sup_{z \in \mathbb{D}'}
|u^*(W-z\Id)^{-1}v-u^*R_0(z)v|<2\eps.\]
As $\eps>0$ is arbitrary, we obtain almost surely
\begin{equation}\label{eq:Dprimeconvergence}
\lim_{N \to \infty} \sup_{z \in \mathbb{D}'}
|u^*(W-z\Id)^{-1}v-u^*R_0(z)v|=0.
\end{equation}
Applying Lemma \ref{lem:analytic-extend} for $f_N(z) = u^*(W-z\Id)^{-1}v$ and $f(z) = u^*R_0(z)v$ concludes the proof.
\end{proof}

\begin{proof}[Proof of Corollary \ref{cor:resolventapprox}]
Let $W'=W+P_2+\ldots+P_k$ and $w'=w+P_2+\ldots+P_k$.
Note that $W',w'$ define the same submatrices
$W_{11} \in \C^{N_1 \times N_1}$ and $(R_0(z))_{11} \in \C^{N_1 \times N_1}$,
the latter because
\[P_1\tau^\cH((w'-z)^{-1})P_1=\tau^\cH(P_1(w'-z)^{-1}P_1)
=\tau^\cH(P_1(w-z)^{-1}P_1)=P_1\tau^\cH((w-z)^{-1})P_1.\]
On the other hand, for $k \geq 2$, their spectra satisfy
\[\spec(W)=\spec(W_{11}) \cup \{0\}, \quad \spec(w)=\spec(w_{11}) \cup \{0\},\]
\[\spec(W')=\spec(W_{11}) \cup \{1\}, \quad \spec(w)=\spec(w_{11}) \cup \{1\}.\]
Then for any $\delta \leq 1/2$, setting $\mathbb{D}$ and $\mathbb{D}'$ as the
sets (\ref{eq:R0}) with $(W,w)$ and $(W',w')$, we have
\[\mathbb{D}_1=\mathbb{D} \cup \mathbb{D}'.\]
Then the result follows from
applying Theorem \ref{thm:resolventapprox} with $u=(u_1,0,\ldots,0)$ and
$v=(v_1,0,\ldots,0)$, for both $(W,w)$ and $(W',w')$.
\end{proof}

\section{Analysis of the mixed effects model}\label{appendix:mixedmodel}

In this appendix, we present the details of the proofs
of Theorems \ref{thm:outliers} and \ref{thm:eigenvectors}, which were omitted
from Section \ref{sec:proof}.

\subsection{Preliminary results} \label{sec:prelim-results}

First, we prove Theorem \ref{thm:sticktobulk}, which guarantees that no bulk
eigenvalues separate from the support.

\begin{proof}[Proof of Theorem \ref{thm:sticktobulk}]
	Recall the block decomposition (\ref{block_decomposition}) in
        $\C^{N\times N}$, the orthogonal projections $P_0, \ldots, P_{2k}$,
and the embedded matrices
        $\tilde{F}_{rs},\tilde{G}_r,\tilde{H}_r\in\C^{N\times N}$.
The only non-zero block of the matrix
	\[\tilde{W}=\sum_{r,s=1}^{k}\tilde{H}_r^*\tilde{G}_r^*\tilde{F}_{rs}\tilde{G}_s
        \tilde{H}_s\in\R^{N\times N}\]
        is the $(0,0)$-block,
        which is equal to $\hSigma$. Consider the two matrices $\tilde{W}$ and
        $\check{W}=\tilde{W}+P_1+\ldots+P_{2k}$. Then
        $\spec(\tilde{W})=\spec(\hSigma) \cup \{0\}$ and
        $\spec(\check{W})=\spec(\hSigma) \cup \{1\}$, so
        \[\spec(\hSigma)=\spec(\tilde{W}) \cap \spec(\check{W}).\]

        Let $X \in \R^{N \times N}$ be a GOE matrix.
Then $\tilde{G}_r$ can be realized as
        $\tilde{G}_r=\sqrt{\frac{N}{n_r}}P_{r+k}XP_r$. Hence, 
	\[\tilde{W}=\sum_{r,s=1}^{k}\frac{N}{\sqrt{n_rn_s}}\tilde{H}_r^*P_rXP_{r+k}\tilde{F}_{rs}P_{s+k}XP_s\tilde{H}_s.\]
        We construct a free deterministic equivalent in the following way:
        Let $\D=\langle P_0,\ldots,P_{2k} \rangle$, and let
        $(\A_1,\tau_1)$ be the von Neumann free product of
        $(\D,N^{-1}\Tr)$ and a von Neumann probability space
        containing a semicircular element $x$. Set
        $(\A_2,\tau_2) \equiv (\C^{N \times N},N^{-1}\Tr)$, which
        contains $\{\tilde{F}_{rs},\tilde{H}_r:r,s=1,\ldots,k\}$ and also
        $\D$. Let $(\A,\tau)=(\A_1,\tau_1) \times_\D (\A_2,\tau_2)$ be the
        amalgamated free product over $\D$. In $\A$, identify
        $f_{rs} \equiv \tilde{F}_{rs}$, $h_r \equiv \tilde{H}_r$, $p_r \equiv
        P_r$, and define $g_r=\sqrt{N/n_r}p_{r+k}xp_r$.
        By this construction,
        $x$ is free of $\D$ (over $\C$) and also free of $\A_2$ over $\D$.
        Then \cite[Proposition 3.7]{NSS} implies that $x$ is free of $\A_2$
        (over $\C$). We may then apply Theorem \ref{thm:strongfree} and
        Assumption \ref{assump:asymptotic} to conclude
	\begin{equation}\label{middle_inclusion_in_proof}
            \spec(\hSigma)\subset\spec(\tilde{w})_\delta \cap
            \spec(\check{w})_\delta
	\end{equation}
        for all large $N$, where
        \[\tilde{w}=\sum_{r,s=1}^{k} \frac{N}{\sqrt{n_rn_s}}
        h_r^*p_rxp_{r+k}f_{rs}p_{s+k}xp_sh_s
        =\sum_{r,s=1}^{k} h_r^*g_r^*f_{rs}g_sh_s,
        \qquad \check{w}=\tilde{w}+p_1+\ldots+p_{2k}.\]

        To finish this proof, we verify that these elements
        $\{f_{rs},g_r,h_r,p_r\}$ have the same joint law as described by
        conditions (1--4) in Section \ref{sec:det-eq-measure}. Conditions
        (1--2) are evident by construction. For condition (3), denoting by
        $\NC_2(2l)$ the non-crossing pairings of $(1,\ldots,2l)$ and $K(\pi)$
        the Kreweras complement of $\pi$,
   	   	\begin{align*}
                    &\frac{N}{p}\tau\Big((g_r^*g_r)^l\Big)\\
&=\frac{N}{p}\left(\frac{N}{n_r}\right)^l\tau\big((p_rxp_{r+k}xp_r)^l\big)\\
                    &=\frac{N}{p}\left(\frac{N}{n_r}\right)^l\tau\big(xp_{r+k}xp_r\cdots xp_{r+k}xp_r\big)\\
                    &=\frac{N}{p}\left(\frac{N}{n_r}\right)^l\sum_{\pi\in\NC_2(2l)}\tau_{K(\pi)}\left[p_{r+k},p_r,...,p_{r+k},p_r\right]\\
   		&=\frac{N}{p}\left(\frac{N}{n_r}\right)^l\sum_{m=1}^l
                    \tau(p_r)^{m}\tau(p_{r+k})^{l+1-m}\cdot
                    |\{\pi\in\NC_2(2l):K(\pi) \text{ has } m \text{ blocks of }
                    p_r\}|\\
   		&=\sum_{m=1}^{l}\left(\frac{p}{n_r}\right)^{m-1}\frac{1}{l}{l
                    \choose m } {l  \choose
                    m-1}\\
&=\sum_{m=1}^{l}\frac{1}{m}\left(\frac{p}{n_r}\right)^{m-1}
                    {l \choose m-1} {l-1 \choose m-1}\\
&=\int x^l\nu_{\frac{p}{n_r}}(x)dx.
   	\end{align*}
        Here, the second line applies \cite[Theorem 14.4]{nicaspeicher},
        freeness of $\{p_r,p_{r+k}\}$ and $x$, and vanishing of all but the
        second non-crossing cumulant of $x$. The third line applies $p_r^l=p_r$
        and $p_{r+k}^l=p_{r+k}$ for $l \geq 1$, and also that
        $|K(\pi)|+|\pi|=2l+1$ so that $|K(\pi)|=l+1$. The fourth line applies
        \[|\{\pi\in\NC_2(2l):K(\pi) \text{ has } m \text{ blocks of }
        p_r\}|=|\{\gamma\in\NC(l):\gamma \text{ has }m \text{ blocks}\}|,\]
        which are defined by the Narayana numbers. For more details, see
        \cite[Lectures 9, 11, 14]{nicaspeicher}. The last equality is the
        formula for the $l^\text{th}$ moment of the Marcenko-Pastur distribution
        (see \cite[Exercise 2.11]{mingospeicher}).

        For condition (4), first consider $a_1,\ldots,a_m \in \A_2$ where
        $a_1,\ldots,a_m$ alternate between the algebras
        $\langle \{f_{rs}\},\D\rangle$ and $\langle \{h_r\},\D \rangle$, and we
        have $\tau^\D(a_i)=0$ for each $i$. The latter condition implies that
        each $a_i$ belonging to $\langle \{f_{rs}\},\D\rangle$ in fact
        satisfies $(p_{k+1}+\ldots+p_{2k})a_i(p_{k+1}+\ldots+p_{2k})=a_i$,
        and each $a_i$ belonging to $\langle \{h_r\},\D\rangle$ in fact
        satisfies $(p_0+\ldots+p_k)a_i(p_0+\ldots+p_k)=a_i$.
        Then we get $\tau^\D(a_1\ldots a_m)=0$. This establishes
        that $\{f_{rs}\}$ and $\{h_r\}$ are free over $\D$. A similar argument
        shows that $g_1,\ldots,g_k$ are free over $\D$, since each $a_i \in
        \langle g_r,\D \rangle$ with $\tau^\D(a_i)=0$ must satisfy
        $(p_r+p_{k+r})a_i(p_r+p_{k+r})=a_i$. By construction of the space $\A$,
        we have that
        $\{g_1,\ldots,g_k\} \in \A_1$ and $\{f_{rs},h_r:r,s=1,\ldots,k\}
        \in \A_2$ are free over $\D$. Thus condition (4) holds.

        Having verified these conditions (1--4), we obtain that
$\mu_0$ is the $\tau^c$-law of $\tilde{w}$ in the compressed algebra
$\A^c=\{a \in
\A:p_0ap_0=a\}$ with trace $\tau^c(a)=\tau(p_0)^{-1}\tau(p_0ap_0)$. Since
$\tau^c$ is faithful, $\supp(\mu_0)$ is the spectrum of $\tilde{w}$
as an operator in
$\A^c$. Then $\spec(\tilde{w})=\supp(\mu_0) \cup \{0\}$ and
$\spec(\check{w})=\supp(\mu_0) \cup \{1\}$, where $\spec(\cdot)$ here
denotes the spectra as operators in $\A$.
So $\supp(\mu_0)_\delta=\spec(\tilde{w})_\delta \cap
\spec(\check{w})_\delta$ for any $\delta<1/2$. Combining this with
(\ref{middle_inclusion_in_proof}) concludes the proof.
\end{proof}

Next, we establish the analytic extension of the functions $a_r,b_r$.
\begin{proposition} \label{prop:analyticextension}
For any positive semidefinite $\oSigma_1,\ldots,\oSigma_k \in \R^{p \times p}$ and symmetric $F \in \R^{M \times M}$,
let $\mu_0$ be the measure defined by (\ref{eq:arecursion}--\ref{eq:m0}). Then
the functions $a_1,\ldots,a_k,b_1,\ldots,b_k,m_0$ which solve
(\ref{eq:arecursion}--\ref{eq:m0}) extend analytically to $\C \setminus
\supp(\mu_0)$. The matrices $z\Id+\b \cdot \oSigma$ and $\Id+F \diag_n(\a)$ are invertible on all of $\C \setminus \supp(\mu_0)$, and these extensions
satisfy (\ref{eq:arecursion}--\ref{eq:m0}) on all of $\C \setminus \supp(\mu_0)$.
\end{proposition}
\begin{proof}[Proof of Proposition \ref{prop:analyticextension}]
Denote by $a_r(z)$ and $b_r(z)$ the values of $a_r,b_r$ at $z \in \C^+$, and set
$R_0(z)=(z\Id+\b(z) \cdot \oSigma)^{-1}$. Note that $\Tr R_0(z)AR_0(z)^*B$ is real and nonnegative for any positive semidefinite $A,B$. Then from (\ref{eq:arecursion}), we have
\begin{align*}
\Im a_r(z)&=-n_r^{-1} \Im \Tr R_0(z)\oSigma_r\\
&=-n_r^{-1} \Im \Tr \Big(R_0(z)\oSigma_r R_0(z)^*(z\Id+\b(z)\cdot\oSigma)^*\Big)\\
&=n_r^{-1}(\Im z)\Tr R_0(z)\oSigma_r R_0(z)^*+n_r^{-1}\sum_{s=1}^k (\Im b_s(z)) \Tr R_0(z)\oSigma_r R_0(z)^*\oSigma_s.
\end{align*}
In particular, as $\Im z>0$, $\Im b_r(z) \geq 0$, and $R_0(z)$ is invertible, we
have that either $\oSigma_r=0$ and $a_r(z) \equiv 0$ for all $z \in \C^+$, or
$\oSigma_r \neq 0$ and $\Im a_r(z)>0$ for all $z \in \C^+$. In the former case,
$a_r$ trivially extends to $a_r(z) \equiv 0$ on $\C \setminus \supp(\mu_0)$. In
the latter case, we recall from the analysis of \cite[Theorem
4.1]{fanjohnstonebulk} that each $b_r(iy)$ remains bounded as $y \to \infty$.
Then $\lim_{y \to \infty} iy \cdot a_r(iy)=-\Tr \oSigma_r/n_r$, so $a_r:\C^+ \to
\C^+$ is the Stieltjes transform of a finite measure $\nu_r$ on $\R$ with total
mass $\nu_r(\R)=\Tr \oSigma_r/n_r$ \cite[Lemma 2]{geronimohill}. Analogous to the above, we also have
\[\Im m_0(z)=p^{-1}(\Im z)\Tr R_0(z)R_0(z)^*+p^{-1}\sum_{s=1}^k (\Im b_s(z)) \Tr R_0(z)R_0(z)^*\oSigma_s,\]
and hence for all $z \in \C^+$
\[\Im a_r(z) \leq \frac{p}{n_r}\|\oSigma_r\| \cdot \Im m_0(z).\]
From the Stieltjes inversion formula, this implies $\supp(\nu_r) \subseteq \supp(\mu_0)$, and hence $a_r$ extends analytically to $\C \setminus \supp(\mu_0)$ also in this case as well.

Then we may extend $b_1(z),\ldots,b_k(z)$ to meromorphic functions on $\C
\setminus \supp(\mu_0)$ via (\ref{eq:brecursion}), potentially with poles at
points $z \in \C \setminus \supp(\mu_0)$ where $\Id+F \diag_n(\a(z))$ is
singular. We claim that no such points exist: Denote $D(z)=\diag_n(\a(z))$,
and suppose that $\Id+FD(z_0)$ is singular for some $z_0 \in \C \setminus
\supp(\mu_0)$. It is shown in \cite[Lemma C.2]{fanjohnstonebulk} that
$\Id+FD(z)$ is invertible for $z \in \C^+$.
For $z \in \C^-$, it is verified by conjugate-symmetry that $\Id+FD(z)$
is also invertible, and $\overline{b_r(z)}=b_r(\overline{z})$. Thus $z_0 \in \R
\setminus \supp(\mu_0)$, and each $b_r$ is real on $\R \setminus \supp(\mu_0)$.
Suppose, for notational convenience, that $b_1(z),\ldots,b_j(z)$ have poles at
$z_0$, and $b_{j+1}(z),\ldots,b_k(z)$ do not. (We may take $j=0$ or $j=k$ if all
or none of the $b_r$'s have poles.)
Taking the limit $z \nearrow z_0$ along the real line, we have
\begin{align*}
&\partial_z \left(-(\Id+FD(z))^{-1}F\right)\\
&=(\Id+FD(z))^{-1}F\diag(a_1'(z)\Id_{m_1},\ldots,a_k'(z)\Id_{m_k})(\Id+FD(z))^{-1}F.
\end{align*}
Assuming momentarily that $F$ is invertible, $(\Id+FD(z))^{-1}F=(F^{-1}+D(z))^{-1}$ is real and symmetric.
Then this is also true for non-invertible $F$ by continuity.
As each $a_s$ is either identically 0 or the Stieltjes transform of a measure
    $\nu_s$, we have $a_s'(z) \geq 0$ for all $s$. So the above derivative in
    $z$ is positive-semidefinite. In particular, as $z \nearrow z_0$, each
    $b_r(z)$ is increasing. So $b_1(z),\ldots,b_j(z) \to \infty$ as $z \nearrow
    z_0$, while $b_{j+1}(z),\ldots,b_k(z)$ approach finite values. This implies
    that for any $v$ in the combined column span of
    $\oSigma_1,\ldots,\oSigma_j$, we have $(z\Id+\b \cdot \oSigma)^{-1}v \to 0$
    as $z \nearrow z_0$. Then $(z\Id+\b \cdot \oSigma)^{-1}\oSigma_r \to 0$ and
$a_r(z_0)=0$ for each $r=1,\ldots,j$.
Denote by $D_2(z)$ and $F_2$ the lower-right blocks of $D(z)$ and $F$ corresponding to
$j+1,\ldots,k$, and by $M_2(z)=\Id+F_2D_2(z)$ the lower-right blocks of
$\Id+FD(z)$. Then the matrix $\Id+FD(z)$ has the block form
\[\begin{pmatrix} \Id & * \\ 0 & M_2(z) \end{pmatrix}.\]
Since $\Id+FD(z_0)$ is singular, we must have that $M_2(z_0)=\Id+F_2D_2(z_0)$ is
singular. The above argument shows that $M_2(z)^{-1}F_2$ is real-symmetric and
that $\partial_z(-M_2(z)^{-1}F_2)$ is positive-semidefinite, so this implies
that $-\Tr M_2(z)^{-1}F_2 \to \infty$ as $z \nearrow z_0$. But then
$b_r(z)=-n_r^{-1}\Tr_r [M_2(z)F_2] \to \infty$ for some
$r \in \{j+1,\ldots,k\}$, contradicting that $b_r(z_0)$ exists and is finite.
Thus, $\Id+FD(z)$ is invertible and $b_1,\ldots,b_k$ are analytic on all of $\C \setminus \supp(\mu_0)$.

We may then extend $m_0(z)$ to $\C \setminus \supp(\mu_0)$ by (\ref{eq:m0}).
Note that this must coincide with the Stieltjes transform of $\mu_0$ on $\C
\setminus \supp(\mu_0)$, by uniqueness of the analytic extension. Finally, note
that if we define $\tilde{a}_1(z),\ldots,\tilde{a}_k(z)$ on $\C \setminus
\supp(\mu_0)$ by (\ref{eq:arecursion}) from $b_1(z),\ldots,b_k(z)$, then each
$\tilde{a}_r(z)$ is a meromorphic function on $\C \setminus \supp(\mu_0)$,
possibly with poles where $z\Id+\b(z) \cdot \oSigma$ is singular. These must agree with $a_1(z),\ldots,a_k(z)$ everywhere outside of these poles, as they agree on $\C^+$. Since $a_1(z),\ldots,a_k(z)$ are analytic on $\C \setminus \supp(\mu_0)$, 
no such poles exist, $z\Id+\b(z) \cdot \oSigma$ is invertible, and $a_1(z),\ldots,a_k(z)$ satisfy (\ref{eq:arecursion}) on all of $\C \setminus \supp(\mu_0)$.
\end{proof}

We record here the following property shown in the above proof.
\begin{proposition} \label{prop:b-increasing}
    For $z \in \R \setminus \supp(\mu_0)$ and each $r \in \{1,\ldots,k\}$, the
values $a_r(z)$ and $b_r(z)$ are real, and $a_r'(z) \geq 0$ and
$b_r'(z) \geq 0$.
\end{proposition}

\subsection{Approximation lemmas} \label{app:approx-trace}

We prove Lemmas \ref{lem:master-eq} and \ref{lem:schur-comp}, and provide the
remaining details of the proof of Proposition \ref{prop:det-eq}.

\begin{proof}[Proof of Lemma \ref{lem:master-eq}]
The eigenvalues of $\hSigma$ which are not eigenvalues of $W$ are the roots of 
\[
\det\Big(R(z) (\hSigma - z \Id)\Big) = 0.
\]
Writing $\hSigma=W+P$ and recalling the notations of Section
\ref{subsec:master}, we have
    \[P = Q \wGamma^\sT \Xi^\sT F G H + H^\sT G^\sT F \Xi \wGamma Q^\sT + Q \wGamma^\sT
    \Xi^\sT F \Xi \wGamma Q^\sT,\]
from which we may compute 
\begin{align*}
R(z)(\hSigma - z \Id) &= \Id + R(z)P\\
&= \Id + \left[\begin{matrix} R(z)Q \wGamma^\sT & R(z) H^\sT G^\sT F \Xi  \wGamma + R(z) Q \wGamma^\sT \Xi^\sT F \Xi \wGamma\end{matrix}\right] \left[\begin{matrix} \Xi^\sT F G H\\ Q^\sT \end{matrix}\right].
\end{align*}
Applying the identity $\det(\Id + XY) = \det(\Id + YX)$, we find that 
\begin{align*}
0 &= \det\left(\Id + \left[\begin{matrix} \Xi^\sT F G H\\  Q^\sT
    \end{matrix}\right] \left[\begin{matrix} R(z)Q \wGamma^\sT & R(z) H^\sT G^\sT F
\Xi \wGamma + R(z) Q \wGamma^\sT \Xi^\sT F \Xi \wGamma \end{matrix}\right]
    \right)=\det \hK(z). \qedhere
\end{align*}
\end{proof}

Lemma \ref{lem:schur-comp} uses the following concentration result from
\cite{BEKYY14}.

\begin{proposition}[{\cite[Lemma 3.1]{BEKYY14}}] \label{prop:concentration}
Let $x,y \in \R^N$ be independent vectors with independent entries satisfying
\[
\E[x_i] = \E[y_i] = 0, \quad \E[x_i^2] = \E[y_i^2] = 1/N, \quad
    \E[|x_i|^k] < C_kN^{-k/2}, \quad \E[|y_i|^k] < C_kN^{-k/2}
\]
    for each $k \geq 1$ and some constants $C_k>0$. Let $A_1,A_2 \in \C^{N
    \times N}$ be any deterministic matrices and $v \in \C^N$ any deterministic
    vector. Then for any $\tau,D>0$ and all $N \geq N_0(\tau,D)$,
    \[\P[|x^\sT v| \geq N^{-1/2+\tau} \|v\|_2]<N^{-D},\]
    \[\P[|x^\sT A_1 x - \Tr A_1| \geq N^{-1+\tau}\|A_1\|_\HS]<N^{-D},
    \quad \P[|x^\sT A_2 y| \geq N^{-1+\tau}\|A_2\|_\HS]<N^{-D}.
\]
\end{proposition}

For a sufficiently large constant $C>0$, define the good event
\begin{equation}\label{eq:goodevent}
    \cE_n=\{\spec(W) \subset \supp(\mu_0)_{\delta/2},\;
\|G_r\|<C,\;\|\Xi_r\|<C \text{ for all }r=1,\ldots,k\}.
\end{equation}
From Theorem \ref{thm:sticktobulk} and Assumption \ref{assump:alpha}, we have
that $\cE_n$ holds almost surely for all large $n$.
On this event $\cE_n$,
we have $\|G\|<C$, $\|\Xi\|<C$, $\|R(z)\|<C\min(1,1/|z|)$, and
$\|R'(z)\|<C\min(1,1/|z|^2)$ for all $z \in U_\delta$ and a constant
$C>0$.

\begin{proof}[Proof of Lemma \ref{lem:schur-comp}]
    Note that $S(z)$ has blocks given by
    \[\sum_{s = 1}^k \Xi_r^\sT F_{rs} G_{s} H_{s} R(z) Q\]
     for $r=1,\ldots,k$, where $\Xi_1,\ldots,\Xi_k$ are independent of
    $G_1,\ldots,G_k$. On the event $\cE_n$, for any fixed $\eps>0$, we have
    $\|S(z)\|_\infty<\eps$ for all $|z|>C_0$ and some constant $C_0>0$.
    For $|z| \leq C_0$, note that $\|F_{rs}G_sH_sR(z)Q\|<C$
    for all $z \in U_\delta$.
    Then this bound holds for the $\ell_2$-norm of each column of
    $F_{rs}G_sH_sR(z)Q$. The entries of $\Xi_r$ satisfy the
    conditions of Proposition \ref{prop:concentration} with $N=n_r$.
    Applying the proposition
    conditional on $G_1,\ldots,G_k$ and on $\cE_n$, we get
    $\|\Xi_r^\sT F_{rs} G_{s} H_{s} R(z) Q\|_\infty<n^{-1/2+\tau}$
    with probability $1-n^{-D}$, and hence $\|S(z)\|_\infty<n^{-1/2+\tau}$
    as well. Taking a union bound over a grid of values in $U_\delta \cap \{|z|
    \leq C_0\}$ with spacing $n^{-1/2}$, and applying Lipschitz continuity of
    $S(z)$ on $\cE_n$, we get almost surely
    \[\sup_{z \in U_\delta : |z| \leq C_0} \|S(z)\|_\infty \to 0.\]
    Then $\limsup_{n \to \infty}
    \sup_{z \in U_\delta} \|S(z)\|_\infty \leq \eps$. As
    $\eps>0$ is arbitrary, this shows $S(z) \sim 0$. This implies also
    $\hK_{11}(z) \sim \Id_{\ell_+}$.

    For $\hT(z)$, note first that $S(z) \sim 0$ and $\hK_{11}(z) \sim \Id_{\ell_+}$
    imply
\[
    \hT(z) \sim \Id + Q^\sT R(z) Q \cdot \wGamma^\sT \Xi^\sT \Big(F - F G H R(z) H^\sT G^\sT F\Big) \Xi \wGamma. 
\]
Notice that $\Xi^\sT \Big(F - F G H R(z) H^\sT G^\sT F\Big) \Xi$ is a $k \times k$ block matrix with blocks
\[
    \Xi_r^\sT Y_{rs}(z)\Xi_s, \qquad Y_{rs}(z)=F_{rs} - \sum_{r', s' = 1}^k F_{rs'} G_{s'} H_{s'} R(z) H_{r'}^\sT G_{r'}^\sT F_{r's}.
\]
    On $\cE_n$, we bound $\|Y_{rs}(z)\|_\HS \leq C\sqrt{n}\|Y_{rs}(z)\|
    \leq C'\sqrt{n}$. Then, applying Proposition \ref{prop:concentration} again
    for each pair $(r,s)$ and each pair of columns of $\Xi_r$ and $\Xi_s$,
    we get for each fixed $z \in U_\delta$ that
\[
    \left\|\Xi_r^\sT Y_{rs}(z)\Xi_s - \1\{r = s\} n_r^{-1} \Tr_r[F-FGHR(z)H^\sT
    G^\sT F] \cdot \Id_{\ell_r}\right\|_\infty<n^{-1/2+\tau}
\]
    with probability $1-n^{-D}$. Applying Lipschitz continuity and a union bound
    over a grid of values $|z| \leq C_0$, a separate argument for $|z|>C_0$
    as above, and the Borel-Cantelli lemma, we obtain the lemma.
\end{proof}

\begin{proof}[Proof of Proposition \ref{prop:det-eq}]
For any $\eps > 0$, we may choose $K > 0$ so that almost surely for all large
    $n$,
\[
\sup_{z \in \mathbb{D}} \left\|\sum_{l = K + 1}^\infty z^{-l - 1} W^l \right\| < \eps, \qquad \sup_{z \in \mathbb{D}} \left\|\sum_{l = K + 1}^\infty z^{-l - 1} w^l \right\| < \eps.
\]
    Then applying the convergent series expansions of $-R(z)=(z-W)^{-1}$
    and of $(z - w)^{-1}$ on
    $\mathbb{D}$, the fact that $\{H_r\}_{r = 1}^k$, $\{G_r\}_{r = 1}^k$, and
    $\{F_{rs}\}_{r, s = 1}^k$ are almost surely uniformly
    bounded for large $n$, and the conclusion
    \[\tau(a_{rts} w^l)-N^{-1}\Tr H_r^\sT G_r^\sT F_{rt} F_{ts} G_s
    H_s W^l \to 0\]
    for each fixed $l \in \{0,\ldots,K\}$
    by \cite[Theorem 3.9]{fanjohnstonesun}, we obtain
\[
    \sup_{z \in \mathbb{D}} \left|-N^{-1} \Tr[H_r^\sT G_r^\sT F_{rt}F_{ts} G_s
    H_s R(z)] - \tau(a_{rts}(z - w)^{-1})\right| < 2 \eps.
\]
As $\eps>0$ is arbitrary, the left side converges to 0 almost surely. By Lemma
    \ref{lem:analytic-extend}, we may then replace the supremum over
$\mathbb{D}$ with one over $U_\delta$.
Applying Proposition \ref{prop:free-comp}, we find that 
\begin{align*}
\frac{1}{n_t} \Tr_t[F G H R(z) H^\sT G^\sT F]&=
    \sum_{r,s = 1}^k \frac{1}{n_t} \Tr[H_r^\sT
    G_r^\sT F_{rt}F_{ts} G_s H_s R(z)] \\
&\sim \sum_{r,s = 1}^k -\frac{N}{n_t}\tau(a_{rts} (z
    - w)^{-1})\\
& = \sum_{r,s= 1}^k -\frac{N}{n_t}
    \tau\Big(f_{ts} (e - u)^{-1} f_{rt}\Big)\\
    &=\frac{1}{n_t} \Tr_t \Big(F (\diag_n(\a^{-1}) + F)^{-1} F \Big),
\end{align*}
    the last step applying the equality of the $N^{-1}\Tr$-law of
    $\{\tilde{F}_{rs}\}$ and the $\tau$-law of
    $\{f_{rs}\}$, the definitions of $e$ and $u$,
    and Proposition \ref{prop:free-identity}(d). Notice now that by the
    Woodbury identity,
    \[F - F(\diag_n(\a^{-1}) + F)^{-1} F =
    (F^{-1} + \diag_n(\a))^{-1} = (\Id + F \diag_n(\a))^{-1} F,\]
which holds also for non-invertible $F$ by continuity. Taking the block trace
    $\Tr_t$, and comparing with the above and with the definition of $b_t$ in
    (\ref{eq:brecursion}) concludes the proof.
\end{proof}

\subsection{Proof for outlier eigenvalues} \label{app:outlier-eigenvalues}

In this section, we give a detailed proof of Theorem \ref{thm:outliers} on
outlier eigenvectors.  We require first the following preliminary results.

\begin{proposition} \label{prop:b-bound}
    There is a constant $C>0$ such that for all $z \in U_\delta$,
    $r \in \{1, \ldots, k\}$, and large enough $n$, we have $|b_r(z)| < C$.
\end{proposition}
\begin{proof}
By Proposition \ref{prop:det-eq}, almost surely as $n \to \infty$ we have
\[
\sup_{z \in U_\delta} \left|n_r^{-1} \Tr_r[FGHR(z) H^\sT G^\sT F-F]
    -b_r(z)\right| \to 0. 
\]
    On the event $\cE_n$ of (\ref{eq:goodevent}), 
    by Assumption \ref{assump:asymptotic},
    we see that for each $z \in U_\delta$,
\[
\left\|FGHR(z)H^\sT G^\sT F-F\right\| < C,
\]
    and hence $|b_r(z)|<C$ almost surely for large $n$. Then this holds
    deterministically for large $n$, since $b_r(z)$ is deterministic.
\end{proof}

\begin{proposition} \label{prop:boundedsupport}
There is a constant $C>0$ such that $\supp(\mu_0) \subset [-C, C]$.
\end{proposition}
\begin{proof}
    The law $\mu_0$ is the $\tau^c$-distribution of $w=\sum_{r,s=1}^k
    h_r^*g_r^*f_{rs}g_sh_s$ in the compressed algebra $(\A^c,\tau^c)$.
    We have $\|w\| \leq C$ for a constant $C>0$,
    hence $\supp(\mu_0)=\spec(w) \subset [-C,C]$.
\end{proof}

\begin{proposition} \label{prop:T-analytic}
The following properties hold for $\wT(z)$ and all large $n$.
\begin{itemize}
    \item[(a)] All roots of $\det \wT(z) = 0$ in $\C \setminus \supp(\mu_0)$
        are real.

\item[(b)] There exists some $R > 0$ so that all roots of $\det \wT(z) = 0$ have
    magnitude at most $R$.

\item[(c)] For $\delta > 0$, there is a constant $C > 0$ such that
\[
\sup_{z \in U_\delta} \|\wT(z)\|_\infty < C, \qquad \sup_{z \in U_\delta} |\det
        \wT(z)| < C.
\]
\end{itemize}
    The following properties hold for $\hT(z)$ almost surely for all large $n$.
\begin{itemize}
    \item[(a')] For $\delta > 0$, all roots of $\det \hT(z) = 0$ in $U_\delta$ are real.
    
    \item[(b')] There exists some $R > 0$ so that all roots of $\det \hT(z) = 0$ in $U_\delta$ have magnitude at most $R$. 
     
    \item[(c')] For $\delta > 0$, there is a constant $C > 0$ such that  we have
\[
\sup_{z \in U_\delta} \|\hT(z)\|_\infty < C, \qquad \sup_{z \in U_\delta} |\det
        \hT(z)| < C.
\]
\end{itemize}
\end{proposition}
\begin{proof}
    We first prove the statements for $\wT(z)$. For (a), applying
    $\det(\Id+XY)=\det(\Id+YX)$ and the fact that $z
    \Id +\b \cdot \oSigma$ is invertible for $z \in \C \setminus \supp(\mu_0)$
    from Proposition \ref{prop:analyticextension}, we have
    \begin{align}
        0=\det \wT(z) &\Leftrightarrow 0=\det\Big(\Id+(z\Id+\b \cdot \oSigma)^{-1}
    \Gamma^\sT \diag_\ell(\b) \Gamma\Big)\nonumber\\
        &\Leftrightarrow 0=\det\Big(z\Id+\b \cdot \oSigma+\Gamma^\sT
\diag_\ell(\b)
        \Gamma\Big).\label{eq:rotated-eq}
    \end{align}
    For $z \in \C^+$ and any $v \neq 0 \in \C^p$,
    we apply $\Im b_r(z) \geq 0$ to get
\[
\Im v^* [z\Id+\b \cdot \oSigma+\Gamma^\sT \diag_\ell(\b)
        \Gamma]v>0,\]
    and hence $z$ is not a root of (\ref{eq:rotated-eq}). A similar argument
    holds for $z \in \C^-$, which establishes (a).

Note by Proposition \ref{prop:b-bound} and Assumption \ref{assump:asymptotic}
that $\|\b
    \cdot \oSigma + \Gamma^\sT \diag_\ell(\b) \Gamma\| < C$ for large $n$.
    Then (b) follows from (\ref{eq:rotated-eq}).
For (c), note by Proposition \ref{prop:det-res-val} that
\[
\|Q^\sT R(z) Q + Q^\sT(z \Id + \b \cdot \oSigma)^{-1}Q\| < C 
\]
almost surely for large $n$. On the event $\cE_n$, $\|Q^\sT R(z) Q\|$ is
    uniformly bounded. Then so is
    $\|Q^\sT(z \Id + \b \cdot \oSigma)^{-1}Q\|$ for large
    $n$. Combining with Proposition \ref{prop:b-bound} and Assumption
    \ref{assump:asymptotic}, the first bound of (c) follows.
    Since the dimension of $\wT(z)$ is $\ell$ which is at most a constant,
    the first bound implies the second.  

We now prove the statements for $\hT(z)$. For (a'), by Lemma
    \ref{lem:schur-comp}, $\det \hK_{11}(z)$ almost surely for large $n$ does not vanish on $U_\delta$.  Thus, for $z \in U_\delta$, by the Schur complement formula, we see that 
\[
\det \hK(z) = \det \hK_{11}(z) \cdot \det \hT(z),
\]
meaning that any root of $\det \hT(z) = 0$ is also a root of $\det \hK(z) = 0$,
    hence a (real) eigenvalue of $\hSigma$ by Lemma \ref{lem:master-eq}.
    Claim (b') follows from the fact that roots of $\det \hT(z) = 0$ are
    eigenvalues of $\hSigma$, and $\|\hSigma\| < C$ almost surely for large $n$.
    For (c'), we apply $\|\wT(z)\|_\infty < C$ and $\hT(z) \sim \vT(z) \sim
    \wT(z)$ from Lemmas \ref{lem:schur-comp} and \ref{lem:master-limit}.
\end{proof}

Under our assumptions, the measure $\mu_0$ and function $\wT(z)$ are
$n$-dependent, and are not required to converge as $n \to \infty$. However,
the following technical result will allow us to pass to convergent subsequences.

\begin{proposition} \label{prop:subseq}
There exists a subsequence $\{n_l^0\}_{l = 1}^\infty$ along which $\supp(\mu_0)$ converges to a fixed closed set $V \subset \R$ in the sense that 
\begin{equation} \label{eq:suppconvergence}
\lim_{l \to \infty} \sup_{z \in V} \dist(z, \supp(\mu_0)) = 0 \qquad \text{ and } \qquad \lim_{l \to \infty} \sup_{z \in \supp(\mu_0)} \dist(z, V) = 0,
\end{equation}
and $\det \wT(z)$ converges to a fixed analytic function $D: \C \setminus V \to \C$ uniformly on compact subsets. 
\end{proposition}
\begin{proof}
    The first part of the statement follows from Proposition
    \ref{prop:boundedsupport} and sequential compactness of the metric
    space of compact subsets of $[-C, C]$ under the Hausdorff metric. For the
    second part, note that $\det \wT(z)$ is well-defined and analytic on compact
    subsets of $\C \setminus V$ at $n = n_l^0$ for all large $l$. Proposition
    \ref{prop:T-analytic} ensures that $\det \wT(z)$ is uniformly bounded on any such compact subset. Then Montel's Theorem implies that there is a further subsequence which converges uniformly over compact sets to an analytic function $D$.
\end{proof}

Using these results, we now prove Theorem \ref{thm:outliers}.

\begin{proof}[Proof of Theorem \ref{thm:outliers}]
    By the identity $\det(\Id+XY)=\det(\Id+YX)$,
    $\Lambda_0$ is also the set of roots of $0=\det \wT(z)$. Let
    $\Omega$ be the sample space, and $\Omega_0 \subset \Omega$
    the event of probability 1 on which all preceding almost sure statements
    hold.

    Fix $\omega \in \Omega_0$.
    First suppose that we pass to a subsequence satisfying the result of
    Proposition
    \ref{prop:subseq}, meaning that $\supp(\mu_0)$ converges to a fixed closed
    set $V$ and $\det \wT(z) \to D(z)$ uniformly on compact subsets of $\C
    \setminus V$.  By Proposition \ref{prop:T-analytic}(a), for all
    large $n$, all roots of $\det \wT(z) = 0$ and $\det \hT(z) = 0$ with
    distance at least $\delta/2$ to $V$ are real and
    have magnitude less than some $R > 0$.  Because $\det \wT(z) \to D(z)$, we
    see that this is true for $D$ as well. Since $D$ is analytic, this implies
    that $D$ has finitely many such roots.  Let 
\[
\lambda_1 < \cdots < \lambda_J
\]
be the distinct roots of $D$ whose distance to $V$ is at least $\delta / 2$, and let $m_j$ be the multiplicity of $\lambda_j$.

Choose $\eps$ small enough so that $\eps < \delta / 4$.  For constants $r_j, \sigma > 0$, let $\gamma_j$ be the counterclockwise contour traversing the rectangle with vertices $(\lambda_j \pm r_j) \pm i \sigma$.  Choose $r_j, \sigma$ small enough such that 
\begin{itemize}
    \item the contours $\gamma_j$ do not intersect,
    \item each $\gamma_j$ is contained within a radius $\eps / 2$ ball centered
        at $\lambda_j$, and
    \item the only root of $D(z)$ contained within or on each $\gamma_j$ is $\lambda_j$. 
\end{itemize}
Partitioning the set
\[
\{x \in \R : \dist(x, \supp(\mu_0)) > \delta/2,\;
    \dist(x, \lambda_j) > r_j\text{ for all } j,\; |x| < R\}
\]
into disjoint open intervals, for each such interval $\cI = (l, u)$, define also a counterclockwise contour $\gamma_\cI'$ traversing the rectangle with vertices $l \pm i \sigma$ and $u \pm i \sigma$.

By construction, $D(z)$ does not vanish along any of $\gamma_j$ or
    $\gamma_\cI'$, and $\det \wT(z)$ converges uniformly to $D(z)$ on each
    contour.  Hence, by Hurwitz's theorem, for all large $n$, $\det \wT(z)$ has
    $m_j$ zeros within each $\gamma_j$, which are real by Proposition
    \ref{prop:T-analytic}, and no zeros within each $\gamma_\cI'$.  Now, observe
    that by Lemmas \ref{lem:schur-comp} and \ref{lem:master-limit}, as $n \to
    \infty$,
\[
\sup_{z \in U_{\delta/4}} \|\hT(z) - \wT(z)\|_\infty \to 0,
\]
    which implies by Proposition \ref{prop:T-analytic}
    that $|\det \hT(z) - \det \wT(z)| \to 0$ uniformly on each contour
    and thus that $|\det \hT(z) - D(z)| \to 0$ uniformly on each contour.
    Applying Hurwitz's theorem again, we find that for all large $n$, $\det \hT(z)$ also has $m_j$ zeros within each $\gamma_j$, which are real by Proposition \ref{prop:T-analytic}, and no zeros within each $\gamma_\cI'$.

Taking $\Lambda_\delta$ and $\hat{\Lambda}_\delta$ as the zeros of $\det \wT(z)$
    and $\det \hT(z)$ within the contours $\gamma_j$, this yields 
\[
\ordereddist(\Lambda_\delta, \hat{\Lambda}_\delta) < \eps.
\]
By Lemma \ref{lem:schur-comp}, for $z \in U_\delta$, $\hK_{11}(z)$ is
    invertible for large $n$, so we may apply the Schur complement formula to obtain
\[
\det \hK(z) = \det \hK_{11}(z) \det \hT(z).
\]
By Lemma \ref{lem:master-eq}, we conclude that $\hLambda_\delta \subseteq
    \spec(\hSigma)$.  Further, since neither $\det \wT(z)$ or $\det \hT(z)$ have
    zeros inside $\gamma_\cI'$ or $(-\infty, R] \cup [R, \infty)$,
    we find that $\Lambda_\delta$ and
    $\hat{\Lambda}_\delta$ contain all zeros of $\det \wT(z)$ and elements of
    $\spec(\hSigma)$, respectively, which have distance at least $\delta/2$ from
    $V$.  Thus they contain all such values which have distance at least
    $\delta$ from $\supp(\mu_0)$ for all large $n$, establishing the result along this subsequence.

To conclude the proof, suppose by contradiction that there is a subset $\Omega_1
\subset \Omega_0$ of positive probability for which there is
a $\omega$-dependent subsequence $\{n_l^0\}$ such
    that for each $n = n_l^0$, no such sets $\Lambda_\delta$ and
    $\hat{\Lambda}_\delta$ satisfying the required conditions exist. By
    Proposition
    \ref{prop:subseq}, there is a further subsequence along which $\supp(\mu_0)$
    and $\det \wT(z)$ converge.  On this subsequence, our previous construction
    shows that $\Lambda_\delta$ and $\hat{\Lambda}_\delta$ satisfying the
    desired conditions exist, a contradiction. This concludes the proof.
\end{proof}

\subsection{Proof for outlier eigenvectors} \label{app:outlier-eigenvectors}

In this section, we prove Propositions \ref{prop:sec-sing} and
\ref{prop:derbounds} used in the proof of Theorem \ref{thm:eigenvectors}
for outlier eigenvectors.

\begin{proof}[Proof of Proposition \ref{prop:sec-sing}]
    Recall from (\ref{eq:rotated-eq}) that
the roots of $\det \wT(z) = 0$ are also roots of $\det M(z) = 0$ for
    $M(z) =
    z\Id+\b \cdot \oSigma+\Gamma^\sT \diag_\ell(\b)\Gamma=
    z \Id + \b \cdot \Sigma$.  By Proposition \ref{prop:b-increasing}, we see that 
\[
    \partial_z[z \Id + \b(z) \cdot \Sigma] = \Id + \b'(z) \cdot \Sigma \succeq \Id
\]
for $z \in \R \setminus \supp(\mu_0)$, meaning that the ordered eigenvalues of
    $M(z)$
    increase at a rate of at least $1$ on each interval of $\R \setminus
    \supp(\mu_0)$.  As a result, at the given isolated root
    $z = \lambda$ of $0=\det M(z)$, the matrix $M(\lambda)$ has a single
    eigenvalue equal to $0$ and remaining eigenvalues outside $(-\delta,
    \delta)$, and the second-smallest singular value of $M(\lambda)$ is at least
    $\delta$. By Propositions \ref{prop:b-bound} and
    \ref{prop:T-analytic}(b), we see that $|b_r(\lambda)|$ and $|\lambda|$ are
    bounded, so
    $\|\lambda \cdot \Id + \b(\lambda) \cdot \oSigma\| < C$ and all
    singular values
    of $(\lambda \Id + \b(\lambda) \cdot \oSigma)^{-1}$ are at least $1/C$.
    Now, letting $v \in \ker M(\lambda)$ be a unit vector, we have
    that for any $w \in \R^p$, $\|w\|^2-|v^\sT w|^2$ is the squared length
    of the component of $w$ orthogonal to $v$. Then
\begin{equation} \label{eq:rm-bound}
    \|(\lambda \Id + \b \cdot \oSigma)^{-1} M(\lambda) w\|^2 \geq (\delta/C)^2
    \cdot (\|w\|^2 - |v^\sT w|^2).
\end{equation}

Choose $U \in \R^{p \times (p - \ell)}$ so that $[Q \mid U]$ is an orthogonal matrix. Notice that because
\[
(\lambda \Id + \b \cdot \oSigma)^{-1} (\lambda \Id + \b \cdot \Sigma) U = U +
    (\lambda \Id + \b \cdot \oSigma)^{-1} \Gamma^\sT \diag_\ell(\b)\Gamma U = U, 
\]
we have that
\[
[Q \mid U]^\sT (\lambda \Id + \b \cdot \oSigma)^{-1} M(\lambda) [Q \mid U] =
    \left[\begin{matrix} \wT(\lambda) & 0 \\ U^\sT (\lambda \Id + \b \cdot
    \oSigma)^{-1} (\Gamma^\sT \diag_\ell(\b) \Gamma) Q & \Id \end{matrix}\right].
\]
Note that $(Q^\sT v, U^\sT v)$ is a unit vector in the kernel of this matrix, hence $Q^\sT
    v \in \ker \wT(\lambda)$ and $Q^\sT v \neq 0$.
    Now, for any $u_1 \in \R^{\ell}$ orthogonal to $Q^\sT v$ and \[
        u_2 = - U^\sT (\lambda \Id + \b \cdot \oSigma)^{-1} (\Gamma^\sT
        \diag_\ell(\b)\Gamma) Qu_1,
\]
define the vector $u = (u_1, u_2)$.  For this $u$, we obtain by (\ref{eq:rm-bound}) that
\[
\|\wT(\lambda) u_1\|^2 = \|[Q \mid U]^\sT (\lambda \Id + \b \cdot \oSigma)^{-1}
    M(\lambda)[Q \mid U] u \|^2 \geq (\delta/C)^2
    (\|u\|^2 - |v^\sT Q u_1 + v^\sT U u_2|^2).
\]
Since $u_1$ is orthogonal to $Q^\sT v$ and $v$ is a unit vector, we see that 
\[
|v^\sT Q u_1 + v^\sT U u_2| = |v^\sT U u_2| \leq \|u_2\|.
\]
Substituting, we obtain for any $u_1$ orthogonal to $Q^\sT v$ that 
    $\|\wT(\lambda)u_1\|^2 \geq (\delta/C)^2 \|u_1\|^2$.
This implies that $\ker \wT(\lambda)$ is one-dimensional and spanned by $Q^\sT
    v$, and the next smallest singular value of $\wT(\lambda)$ is bounded below
    by $\delta/C$, as desired.
\end{proof}

\begin{proof}[Proof of Proposition \ref{prop:derbounds}]
    By Lemma \ref{lem:schur-comp}, we have
    \[\sup_{z \in U_{\delta/2}} \|S(z)\|_\infty \to 0\]
    almost surely. For each $z \in U_\delta$,
    define a contour $\gamma(t)=\delta/2 \cdot e^{it}$ for $t \in [0,2\pi]$.
    Applying the Cauchy integral formula entrywise to $S(z)$, we get
    \[\|S'(z)\|_\infty \leq \frac{2}{\delta}
    \cdot \max_{t \in [0,2\pi]} \|S(z+\gamma(t))\|_\infty
    \leq C\sup_{z \in U_{\delta/2}} \|S(z)\|_\infty.\]
    Hence $S'(z) \sim 0$. The other statements follow similarly from
    Propositions \ref{prop:det-res-val} and \ref{prop:det-eq}.
\end{proof}

\section*{Acknowledgments}
We thank Camille Male and Roland Speicher for helpful pointers to the strong
asymptotic freeness literature.
Y.~S.~was supported by a Junior Fellow award from the Simons Foundation and NSF
Grant DMS-1701654. Z.~F.~was supported in part by NSF Grant DMS-1916198.

\bibliographystyle{alpha}
\bibliography{refs}

\end{document}